\documentclass[11pt]{amsart}

\usepackage{amscd,latexsym,amsthm,amsfonts,amssymb,amsmath,amsxtra,dsfont}

\usepackage{ytableau}
\ytableausetup{smalltableaux}

\usepackage{dynkin-diagrams}

\usepackage{MnSymbol}

\usepackage{multirow}

\usepackage[all]{xy}

\usepackage[mathscr]{eucal}
\usepackage[pdftex,bookmarks,colorlinks,pdfmenubar]{hyperref}

\usepackage{verbatim}

\renewcommand\theequation{\thesection.\arabic{equation}}

\newcommand{\mult}{\textrm{mult\,}}

\newcommand{\BA}{{\mathbb {A}}}

\newcommand{\BC}{{\mathbb {C}}}

\newcommand{\BH}{{\mathbb {H}}}

\newcommand{\BN}{{\mathbb {N}}}

\newcommand{\BQ}{{\mathbb {Q}}}
\newcommand{\BR}{{\mathbb {R}}}

\newcommand{\BZ}{{\mathbb {Z}}}

\renewcommand{\CD}{{\mathcal {D}}}

\newcommand{\CH}{{\mathcal {H}}}
\newcommand{\CI}{{\mathcal {I}}}
\newcommand{\CJ}{{\mathcal {J}}}

\newcommand{\CL}{{\mathcal {L}}}

\newcommand{\CN}{{\mathcal {N}}}
\newcommand{\CO}{{\mathcal {O}}}
\newcommand{\CP}{{\mathcal {P}}}

\newcommand{\CS}{{\mathcal {S}}}
\newcommand{\CT}{{\mathcal {T}}}

\newcommand{\CW}{{\mathcal {W}}}

\newcommand{\CY}{{\mathcal {Y}}}
\newcommand{\CZ}{{\mathcal {Z}}}

\newcommand{\FB}{{\mathfrak {B}}}

\newcommand{\FD}{{\mathfrak {D}}}

\newcommand{\FN}{{\mathfrak {N}}}

\newcommand{\FP}{{\mathfrak {P}}}

\newcommand{\FR}{{\mathfrak {R}}}
\newcommand{\FS}{{\mathfrak {S}}}
\newcommand{\FT}{{\mathfrak {T}}}

\newcommand{\FV}{{\mathfrak {V}}}

\newcommand{\Fc}{{\mathfrak {c}}}
\newcommand{\Fd}{{\mathfrak {d}}}
\newcommand{\Fe}{{\mathfrak {e}}}

\newcommand{\Fg}{{\mathfrak {g}}}

\newcommand{\Fl}{{\mathfrak {l}}}
\newcommand{\Fm}{{\mathfrak {m}}}
\newcommand{\Fn}{{\mathfrak {n}}}

\newcommand{\Fp}{{\mathfrak {p}}}

\newcommand{\Fr}{{\mathfrak {r}}}
\newcommand{\Fs}{{\mathfrak {s}}}

\newcommand{\Fu}{{\mathfrak {u}}}

\newcommand{\Fw}{{\mathfrak {w}}}

\newcommand{\RI}{{\mathrm {I}}}

\newcommand{\RO}{{\mathrm {O}}}

\newcommand{\RU}{{\mathrm {U}}}

\newcommand{\ScO}{{\mathscr {O}}}

\newcommand{\Ad}{{\mathrm{Ad}}}
\newcommand{\Aut}{{\mathrm{Aut}}}

\newcommand{\ari}{{\mathrm{ari}}}

\newcommand{\an}{{\mathrm{an}}}

\newcommand{\bp}{{\mathrm{bp}}}

\newcommand{\disc}{{\mathrm{disc}}}

\newcommand{\End}{{\mathrm{End}}}

\newcommand{\Gal}{{\mathrm{Gal}}}
\newcommand{\GL}{{\mathrm{GL}}}

\newcommand{\gen}{{\mathrm{gen}}}
\newcommand{\gp}{{\mathrm{gp}}}

\newcommand{\Hom}{{\mathrm{Hom}}}

\newcommand{\Isom}{{\mathrm{Isom}}}

\newcommand{\Mp}{{\mathrm{Mp}}}

\newcommand{\mx}{\mathrm{max}}

\newcommand{\nsd}{{\mathrm{nsd}}}

\newcommand{\rank}{{\mathrm{rank}}}

\newcommand{\SL}{{\mathrm{SL}}}

\newcommand{\SO}{{\mathrm{SO}}}

\newcommand{\SU}{{\mathrm{SU}}}

\newcommand{\sgn}{{\mathrm{sgn}}}
\newcommand{\Sp}{{\mathrm{Sp}}}

\newcommand{\st}{{\mathrm{st}}}
\newcommand{\Span}{{\mathrm{Span}}}

\newcommand{\tr}{{\mathrm{tr}}}

\newcommand{\ud}{\,\mathrm{d}}

\newcommand{\Witt}{\mathrm{Witt}}
\newcommand{\wf}{\mathrm{WF}}

\newcommand{\wm}{\mathrm{wm}}

\newcommand{\ovl}{\overline}
\newcommand{\udl}{\underline}
\newcommand{\wt}{\widetilde}
\newcommand{\wh}{\widehat}

\newcommand{\apair}[1]{\left\langle {#1} \right\rangle}

\newcommand{\ol}{\overline}
\newcommand{\ul}{\underline}

\def\bks{{\backslash}}

\def\diag{{\rm diag}}

\def\eps{{\epsilon}}

\def\sig{{\sigma}}

\newtheorem{thm}{Theorem}[section]
\newtheorem{cor}[thm]{Corollary}
\newtheorem{lem}[thm]{Lemma}
\newtheorem{prop}[thm]{Proposition}
\newtheorem {conj}[thm]{Conjecture}
\newtheorem {ques/conj}[thm]{Question/Conjecture}

\newtheorem{defn}[thm]{Definition}
\newtheorem{rmk}[thm]{Remark}

\usepackage[numbers]{natbib}

\begin{document}
\renewcommand{\theequation}{\arabic{equation}}
\numberwithin{equation}{section}

\title[Wavefront Sets and Generic $L$-packets]{Arithmetic Wavefront Sets and Generic $L$-packets}

\author{Dihua Jiang}
\address{School of Mathematics, University of Minnesota, Minneapolis, MN 55455, USA}
\email{jiang034@umn.edu}

\author{Dongwen Liu}
\address{School of Mathematical Sciences, Zhejiang University, Hangzhou 310058, Zhejiang, P. R. China}
\email{maliu@zju.edu.cn}

\author{Lei Zhang}
\address{Department of Mathematics,
National University of Singapore,
Singapore 119076}
\email{matzhlei@nus.edu.sg}


\date{\today}

\thanks{The work of the first author is supported in part by the NSF Grant DMS-2200890 and the Simons Grants: SFI-MPS-SFM-00005659 and 
SFI-MPS-TSM-00013449. The work of the second author is supported in part by National Natural Science Foundation of China No. 12171421. The work of the third author is supported by AcRF Tier 1 grants R-146-000-277-114 and R-146-000-313-144 of National University of Singapore.
}

\begin{abstract}
Let $G$ be a classical group defined over a local field $F$ of characteristic zero. Let $\pi$ be an irreducible admissible representation $\pi$ of $G(F)$, which is of Casselman-Wallach type if $F$ 
is archimedean. If $\pi$ has a generic local $L$-parameter, we define the {\it arithmetic wavefront set} $\wf_\ari(\pi)$ of $\pi$, which is a subset of $F$-rational nilpotent orbits in the 
Lie algebra $\Fg(F)$ under adjoint action of $G(F)$, by means of the arithmetic structures of the enhanced $L$-parameter $(\varphi,\chi)$ of $\pi$. 
Those arithmetic structures are discovered by 
using our method of consecutive descents of enhanced $L$-parameters, based on the rationality of the local Langlands correspondence and the local Gan-Gross-Prasad conjecture. We study the basic structure of $\wf_\ari(\pi)$ and prove that it is an invariant of $\pi$ (Theorem \ref{inv-pi}). 
Furthermore, those basic structures $\wf_\ari(\pi)^\mx$ are expected to yield 
the precise $F$-rational structure of $\wf_\ari(\pi)^\mx$, 
which has been realized, when $F$ is archimedean, in Theorems \ref{thm1} and \ref{thm2} (Theorems \ref{thm:main}, \ref{cor:wf-max-unique-archimedean}, and \ref{thm:maximal-rational-order}). 
Based on the local Langlands reciprocity, the Wavefront Set Conjecture (Conjecture \ref{conj:wfss} and Conjecture \ref{conj:main}) asserts that the wavefront sets on the $L$-parameter side should be closed related to those on the representation side, namely, 
\[
\wf_\wm(\pi)^\mx=\wf_\ari(\pi)^\mx=\wf_\tr(\pi)^\mx
\]
when $\pi$ has a generic local $L$-parameter, where the algebraic wavefront set $\wf_\wm(\pi)$ is defined by C. M\oe glin and J.-L. Waldspurger in \cite{MW87}, using the generalized Whittaker models and the analytic wavefront set $\wf_\tr(\pi)$ is defined by R. Howe in \cite{H81,Hd85} via the wavefront set of the distribution character $\Theta_\pi$ of $\pi$, and also by the asymptotic expansion of the distribution character $\Theta_\pi$ of $\pi$ near the identity (\cite{H74, HC78, BV80}). Conjecture \ref{conj:wfss} is verified for families of interesting cases and provides a means to determine $\wf_\wm(\pi)^\mx$ and $\wf_\tr(\pi)^\mx$ in terms of the computationally more accessible 
$\wf_\ari(\pi)^\mx$. 
\end{abstract}

\maketitle
\tableofcontents

\section{Introduction}



Let $F$ be a local field of characteristic zero, which is either $\BR$, $\BC$ or a finite extension of the field $\BQ_p$ of $p$-adic numbers for some prime $p$. Let $G$ be any reductive algebraic group defined over $F$. Denote by $\Pi_F(G)$ the set of equivalence classes of irreducible admissible representations of $G(F)$. Here if $F$ is archimedean, we take the representations of $G(F)$ to be of Casselman-Wallach type (\cite{Cas89, Wal92}).

The local Langlands conjecture (\cite{L70, B79, V93}) as a vast generalization of the local Artin reciprocity law over $F$ provides a classification of the set $\Pi_F(G)$ (the admissible dual of $G$) in terms of the arithmetic 
data associated with $F$. More precisely, let $\CL_F$ be the local Langlands group of $F$, whose relation with the local Weil group $\CW_F$ is given in \eqref{LLG}. Let $^LG$ be the local $L$-group of 
$G$. A local Langlands parameter or $L$-parameter is an admissible homomorphism 
\begin{align}\label{LLP}
    \varphi\ \colon\ \CL_F\longrightarrow {^LG}
\end{align}
which is compatible with the action of the local Weil group or Galois group on the both sides (\cite{B79}). The set of equivalence classes of the local $L$-parameters of $G$ is denoted by $\Phi_F(G)$. 
We refer to Section \ref{ssec-LVP} for more details. The local Langlands conjecture for $G$ over $F$ asserts that there exists a canonical correspondence (the local Langlands reciprocity map)
\begin{align}\label{LLR}
    \FR_{F,G}\ \colon\ \Pi_F(G)\longrightarrow \Phi_F(G)
\end{align}
such that for each $\varphi\in\Phi_F(G)$, the fiber $\FR_{F,G}^{-1}(\varphi)$ is a non-empty finite subset of $\Pi_F(G)$, which is denoted by $\Pi_\varphi(G)$ and is called the local $L$-packet associated with $\varphi$. 

When $F$ is archimedean, R. Langlands (\cite{L89}) establishes the local Langlands conjecture for all reductive algebraic groups over $F$ by classifying the set $\Pi_F(G)$ by means of 
the local $L$-parameters in $\Phi_F(G)$ and their associated arithmetic data. It remains a difficult problem in general when $F$ is non-archimedean. 

The local Langlands conjecture for the general linear groups over a non-archimedean lcoal field $F$ is a theorem proved by M. Harris and R. Taylor (\cite{HT01}), of G. Henniart (\cite{Hn00}) and also of P. Scholze (\cite{Sch13}). For symplectic groups and $F$-quasisplit special orthogonal groups, it can be deduced from the endoscopic classification of J. Arthur in \cite{Ar13}. The method of \cite{Ar13} 
has been extended to $F$-quasisplit unitary groups by C. P. Mok (\cite{Mok15}), and by T. Kaletha, A. Minguez, S. W. Shin, and P.-J. White for pure inner forms of unitary groups (\cite{KMSW14}). It is extended to metaplectic groups by using local theta correspondence by H. Atobe and W. T. Gan (\cite{AG17a}). 
The recent work of L. Fargues and P. Scholze (\cite{FS21}) makes an important progress on the local Langlands conjecture for general reductive algebraic groups over $p$-adic local fields. 
However, there are many issues that remain to be resolved. We refer to a recent paper by M. Harris (\cite{Hm22}) and T. Kaletha (\cite{K23}) for a detailed explanation of the state of the art about the local Langlands conjecture. 

With the local Langlands conjecture, it is fundamentally important to explore relations between information carried by the arithmetic data associated with the local $L$-parameters and that carried by the representations in the local $L$-packets, via the local Langlands reciprocity. Following the historical remarks by B. Gross in \cite{G22}, it is clear that the pioneering work of B. Gross and D. Prasad 
(\cite{GP92,GP94}) indicates that certain types of the branching problem on representations in $\Pi_F(G)$ when $G$ is a special orthogonal group are governed by the sign of the local root numbers 
associated with the corresponding local $L$-parameters, assuming the validity of the local Langlands reciprocity. The local Gross-Prasad conjecture for special orthogonal groups (\cite{GP92, GP94}) has been extended 
by W. T. Gan, B. Gross and D. Prasad in \cite{GGP12, GGP20} to cover all classical groups, which is now called the Gan-Gross-Prasad conjecture for classical groups over local fields. 
Definitely, the local Gan-Gross-Prasad conjecture is one of the most influential examples that predicts the interplay between the arithmetic information associated with the local $L$-parameters and 
the harmonic analysis associated with the local $L$-packets (or more precisely the local Vogan packets, see \eqref{lvp} for definition) attached to the corresponding local $L$-parameters via the local Langlands reciprocity. 

The local Gan-Gross-Prasad conjecture for classical groups with generic $L$-parameters is now a theorem accumulated from a series of works pioneered by J.-L. Waldspurger in \cite{W10, W12a, W12b} and then in \cite{MW12}, and followed by others in \cite{BP16, BP20, GI16, At18, X23, X, Lu20, Lu21, Ch23, Ch24, Ch25, CL22, CCZ25}. The local Gan-Gross-Prasad conjecture for more general $L$-parameters (\cite{GGP20}) is widely open in general, except for some recent progress for general linear groups (\cite{Gur22, Cha22, Cha22p, CC25}). 

It is clear that the local Gan-Gross-Prasad conjecture is closely related to the local descents of representations of classical groups via certain degenerate Whittaker models, called the Bessel models and the Fourier-Jacobi models (see \cite{GGP20} for instance). The local descents were developed as the local analogy of the automorphic descent of D. Ginzburg, S. Rallis, and D. Soudry (\cite{GRS11}) to 
prove the local functorial transfer (backward functoriality) from general linear groups to quasisplit classical groups for certain types of local $L$-parameters (see \cite{JS03, JS12, JZ18} for instance). 

Inspired by the local Gan-Gross-Prasad conjecture and the theory of local descents, two of the authors of this paper introduced in \cite{JZ18} a descent method for generic local $L$-parameters (see Section \ref{ssec-LVP} for definition). The $L$-parameter descent yields an explicit branching (spectral) decomposition at the first occurrence index (\cite[Theorem 1.7]{JZ18}) for any members of a given local $L$-packet $\Pi_\varphi(G)$ where 
$G$ is a quasisplit special orthogonal group or its pure inner form.  The method clearly works for other classical groups. Conjecture 1.8 in \cite{JZ18} asserts a close relation between the first occurrence index of the local descents of $\pi\in\Pi_\varphi(G)$ and the structure of the $L$-parameter descent of $\varphi$. 
Such an explicit branching (spectral) decomposition as given in Theorem 1.7 of \cite{JZ18} indicates that the local Langlands reciprocity may carry much deeper information from the arithmetic data associated with $\varphi$ to representations of $G(F)$ in $\Pi_\varphi(G)$ beyond what is described by the local Gan-Gross-Prasad conjecture. 

In this paper, we are going to take the approach in \cite{JZ18} forward and develop a method of {\it consecutive descents of enhanced $L$-parameters}. Based on the rationality of the 
local Langlands correspondences and the normalization of the arithmetic data from the local Gan-Gross-Prasad conjecture for generic $L$-parameters, this method enables us to define the notion of 
{\it arithmetic wavefront sets} $\wf_\ari(\pi)$ for $\pi \in \Pi_\varphi(G)$ and determine their $F$-rational structures. 

The wavefront sets of representations of $G(F)$ are important invariants that determine the size of the infinite-dimensional representations and the singular supports of the distribution characters 
of those representations. 
The notion of the {\it wavefront set} of $\pi\in\Pi_F(G)$ when $F$ is an archimedean local field was first introduced by R. Howe in 1981 (\cite{H81}) by essentially considering the 
wavefront set of the distribution character $\Theta_\pi$ of $\pi$, as a special situation, but refined version of the general notion of the wavefront set or singular support of distributions as introduced by L. H\"ormander around 1970 (\cite{Hor71}). Since then, some variants of wavefront sets have been introduced for $\pi\in\Pi_F(G)$ over all local fields $F$.

When $F$ is non-archimedean, an asymptotic expansion of the distribution character $\Theta_\pi$ near the identity was established by R. Howe (\cite{H74}) and Harish-Chandra (\cite{HC78})
\begin{align}\label{thetapi}
\Theta_\pi=\sum_{\CO\in\CN_F(\Fg)_\circ}c_\CO\wh{\mu}_\CO,\quad c_\CO\in\BC,
\end{align}
where $\CN_F(\Fg)_\circ$ denotes the set of $F$-rational nilpotent orbits in the Lie algebra $\Fg(F)$ of the group $G(F)$ under the adjoint action of $G(F)$ and
$\wh{\mu}_\CO$ denotes the Fourier transform of the (normalized) measure $\mu_\CO$ on the nilpotent orbit $\CO$. From this asymptotic expansion of $\Theta_\pi$, one defines the {\bf analytic wavefront set} of $\pi\in\Pi_F(G)$ to be
\begin{align}\label{wf-tr}
\wf_\tr(\pi):=\{\CO\in\CN_F(\Fg)_\circ\ \mid\ c_\CO\neq 0\},
\end{align}
where $c_\CO$'s are as in \eqref{thetapi} (see also \cite{Hd85, Prz90} for some relevant discussions). 
C. M\oe glin and J.-L. Waldspurger define in \cite{MW87} (see also \cite{K87}) the 
{\bf algebraic wavefront set} $\wf_\wm(\pi)$ of $\pi\in\Pi_F(G)$ by means of the {\it degenerate Whittaker models} of $\pi$ associated with any nilpotent element $X$ in  the set $\CN_F(\Fg)$ of all nilpotent elements in $\Fg(F)$ (see Section \ref{ssec-TJM} for some details).

One of the main results in \cite{MW87} shows that when $F$ is non-archimedean, the set of the maximal members, which is denoted by $\wf_\wm(\pi)^\mx$, of the algebraic wavefront set $\wf_\wm(\pi)$ coincides with the set of maximal members, which is denoted by $\wf_\tr(\pi)^\mx$, of the analytic wavefront set $\wf_\tr(\pi)$, that is,
\begin{align}\label{alg-ana-wfs}
\wf_\wm(\pi)^\mx=\wf_\tr(\pi)^\mx
\end{align}
for $\pi\in\Pi_F(G)$. We mention that \eqref{alg-ana-wfs} holds for the $F$-rational topology as given in Definition \ref{defn:order}. When $F$ is archimedean, from the work of D. Barbasch and D. Vogan (\cite{BV80}) and that of W. Rossmann (\cite{R95}), one expects that the identity in \eqref{alg-ana-wfs} should still hold. However, only some very special situations have been studied (see \cite{V78, Mat87, Mat90, Mat92, GS15, GGS17} for instance). It should be mentioned that when $F$ is archimedean, there are more invariants (associated variety, characteristic cycle and wavefront cycle, etc.) of $\pi$ closely related to the wavefront set of $\pi$. Those invariants form a rich theory. We refer to the papers of W. Rossmann (\cite{R95}), of D. Vogan (\cite{V17}), of J. Adams and D. Vogan (\cite{AV21}), and of W. Schmid and K. Vilonen (\cite{SV00}) for an excellent survey and detailed discussions of the theory, 
in particular, for the relation between the wavefront sets and the construction of the unitary dual of the real reductive groups.  

The arithmetic wavefront sets $\wf_\ari(\pi)$ for $\pi\in\Pi_F(G)$ introduced in this paper take information from the arithmetic side of the local Langlands reciprocity. 
We restrict ourselves in this paper to the situation that $G$ is a {\it classical group} and 
the representations $\pi\in\Pi_F(G)$ have {\it generic} local $L$-parameters. Under such technical assumptions, 
we develop the basic theory of arithmetic wavefront sets for all local fields of characteristic zero, by 
introducing a method of {\it consecutive descents of enhanced $L$-parameters} (Sections \ref{sec-CGRNO}--\ref{sec-AWFS}), and establish the conjectural structures and their refinements of the arithmetic wavefront sets for the archimedean case (Sections \ref{sec-ADA}--\ref{sec-FRM}). 

In order to define the arithmetic wavefront set $\wf_\ari(\pi)$, we extract arithmetic information from the enhanced $L$-parameter $(\varphi,\chi)$ associated with $\pi$ by 
the method of the consecutive descents of enhanced $L$-parameters (Sections 
\ref{ssec-LPD} and \ref{ssec-WFSA}). This method generalizes the one originally used by two of the authors in \cite{JZ18}. The technical base of this method is the 
local Gan-Gross-Prasad conjecture for classical groups (\cite{GGP12}), which is now a theorem as mentioned above. 

The classical groups $G$ considered in this paper include $F$-quasisplit groups and their pure inner forms. For completeness, the metaplectic double cover $\Mp_{2n}(F)$ of the symplectic group $\Sp_{2n}(F)$ is also included. 
According to the local Langlands conjecture for $G$ over $F$, for each local $L$-parameter $\varphi$, there exists a finite subset ($L$-packet) $\Pi_\varphi(G)$ of $\Pi_F(G)$ that is assigned to $\varphi$ via the local Langlands correspondence for $G$. Such an assignment is unique, up to a choice of the $F$-rationality, for the classical groups 
considered in this paper (see \cite{Hm22, K23} for detailed discussion of such a uniqueness issue in general). Following \cite{V93}, one may define the local Vogan packet $\Pi_\varphi[G^*]$ 
associated to $\varphi$ for each $F$-quasisplit group $G^*$ (see \eqref{lvp} for definition). 
For classical groups, the component group $\CS_\varphi$ associated to $\varphi$ is an abelian 
$2$-group. The Vogan version of the local Langlands conjecture (\cite{V93}) asserts that for each $\pi\in\Pi_\varphi[G^*]$, there exists a unique character $\chi$ of $\CS_\varphi$ assigned to $\pi$. We may write 
\begin{align}\label{pia}
   \pi=\pi_a(\varphi,\chi).  
\end{align}
Here $a\in F^\times$ indicates the dependence of the $F$-rationality of the local Langlands correspondence (the assignment), which will be explained in Section \ref{ssec-RLLC}. 
In this case, the pair $(\varphi,\chi)$ is called the {\sl enhanced $L$-parameter of $\pi$}. The local 
$L$-parameter $\varphi$ is called {\it generic} if the local Vogan packet contains at least one 
generic member, i.e. a member with a non-zero Whittaker model with a choice of Whittaker 
data. We refer to Section \ref{sec-LLC} for the details and notations. 

By a conjecture of F. Shahidi in \cite{Sh90}, all tempered local $L$-packets $\Pi_\varphi(G^*)$ of $F$-quasisplit group $G^*$ are expected to be generic. 
By the Whittaker normalization of Shahidi, we may write (non-canonically) a generic member 
in $\Pi_\varphi(G^*)$ as $\pi=\pi(\varphi,{\mathds 1})$, where ${\mathds 1}$ is the trivial character of $\CS_\varphi$. If $G\not=\GL_n$, the local $L$-packet $\Pi_\varphi(G)$ may contain more than one member in general. 
The generic representations such as $\pi=\pi(\varphi,{\mathds 1})$ in the local Vogan packet $\Pi_\varphi[G^*]$ have the largest possible size (the Gelfand-Kirillov dimension) and 
the $F$-stable nilpotent orbit $\CO^\st(\pi)$ that contains the $F$-rational nilpotent orbits in  
the wavefront set $\wf_\wm(\pi)^\mx$ or $\wf_\tr(\pi)^\mx$ of $\pi$ must be the regular nilpotent orbit in $\CN_F(\Fg)_\circ^\st$. 
It remains an interesting problem to understand the wavefront set of any member $\pi$ 
in the local Vogan packet $\Pi_\varphi[G^*]$ that {\sl may not be generic}. 

We introduce in Definition \ref{defn:pd} the descents of enhanced $L$-parameters $(\varphi,\chi)$, which is denoted by $\FD_\Fl(\varphi, \chi)$ as in \eqref{Fl-pd}, based on explicit calculations of the local root numbers in the local Gan-Gross-Prasad conjecture as explained in Section \ref{ssec-LGGP}. First, we show that for a given $(\varphi,\chi)$, there exists an integer $\Fl>0$, such that 
the $\Fl$-th descent $\FD_\Fl(\varphi, \chi)$ is non-empty (Proposition \ref{prop:DNE}). 
The largest integer $\Fl$ such that $\FD_\Fl(\varphi, \chi)$ is non-empty is denoted by 
$\Fl_0=\Fl_0(\varphi, \chi)$ and is called the {\it first occurrence index} of $(\varphi,\chi)$. 
Then we show that the $\Fl_0$-th descent $\FD_{\Fl_0}(\varphi, \chi)$, which is called 
the {\it first descent}, only consists of discrete enhanced $L$-parameters 
(Theorem \ref{thm:DFD}, which is an extension of \cite[Theorem 4.6]{JZ18}). 

In \cite{CJLZ24}, we show that if $\pi=\pi_a(\varphi,\chi)$ as in \eqref{pia}, the first occurrence index 
$\Fl_0=\Fl_0(\varphi, \chi)$ is equal to the first occurrence index for the local descent tower of $\pi$ and 
the local descent of $\pi$ at $\Fl_0$ consists of only discrete series representations whose enhanced $L$-parameters 
are those belonging to $\FD_{\Fl_0}(\varphi, \chi)$ (See \cite[Theorems 1.4, 1.5]{CJLZ24} for more details).
Some applications of this discreteness to representation theory can also be found in \cite{CJLZ24}. In particular, \cite[Theorem 1.6]{CJLZ24} established a new submodule theorem for any irreducible admissible representations with generic $L$-parameters.

In order to extract the arithmetic information from consecutive descents of enhanced $L$-parameters, we introduce the notion of ordered partitions
in Definition \ref{defn:partition}, and more notions of combinatorial nature in Section \ref{ssec-WFSA}.
For an ordered partition $\underline{\Fl}:=(\ell_1,\ell_2,\dots,\ell_k)$ of an integer $\Fn$, 
we define as in \eqref{Fs} a {\it pre-tableau} $\Fs_{\underline{\Fl}}$ associated with $\underline{\Fl}$ to be a sequence of pairs as follows:
\[
\Fs_{\underline{\Fl}}=\left((\ell_1,q^h_{1}),(\ell_2,q^h_{2}),\dots, (\ell_k,q^h_{k})\right),
\]
where $q^h_i$ is a $1$-dimensional $\epsilon_i$-Hermitian form. 
The set of {\it pre-tableaux for an enhanced $L$-parameter $(\varphi,\chi)$}, which is denoted by $\CT_a(\varphi,\chi)$ and defined in 
Definition \ref{LYT},  is a collection of the combinatorial data from the consecutive 
descents of the given enhanced $L$-parameter $(\varphi,\chi)$. 
The set of {\it $L$-descent pre-tableaux}, denoted by $\CL_a(\varphi,\chi)$, is the subset of 
$\CT_a(\varphi,\chi)$ with decreasing partitions $\udl{\Fl}$. As explained in \eqref{LY}, there is a canonical way to construct admissible sesquilinear Young tableaux from given pre-tableaux. We denote by $\CY_a(\varphi,\chi)$ the set of admissible sesquilinear Young tableaux constructed from the set of $L$-descent pre-tableaux $\CL_a(\varphi,\chi)$ and call it 
the set of {\it $L$-descent Young tableaux}. 
It is important to show that the set of $L$-descent pre-tableaux $\CL_a(\varphi,\chi)$ 
is non-empty (Proposition \ref{prop:LDPT}).  

The process to construct those combinatorial objects forms basic steps towards our 
definition of {\it wavefront sets for enhanced $L$-parameters} $\wf_a(\varphi,\chi)$ in 
Definition \ref{defn:AWFS}, which is the subset of $F$-rational nilpotent orbits in $\CN_F(\Fg_n)_\circ$ that corresponds canonically with $\CY_a(\varphi,\chi)$. 
The first main result in the theory of consecutive descents of enhanced $L$-parameters is to show that the wavefront set $\wf_a(\varphi,\chi)$ for the enhanced $L$-parameter $(\varphi,\chi)$ is independent of the choice of the rationality of the local Langlands correspondence and the normalization of the arithmetic data from the local Gan-Gross-Prasad conjecture (Theorem \ref{inv-pi}). 
Hence we are able to define the 
{\bf arithmetic wavefront set} $\wf_\ari(\pi)$ for any $\pi\in\Pi_\varphi[G^*]$ as follows:
\begin{align}\label{awfs0}
    \wf_\ari(\pi):= \wf_a(\varphi,\chi)
\end{align}
if $\pi=\pi_a(\varphi,\chi)$. See Section \ref{ssec-AWF} for more details. 

\begin{thm}[Theorem \ref{inv-pi}]\label{thm-inv}
 The arithmetic wavefront set $\wf_\ari(\pi)$ in \eqref{awfs0} is a well-defined invariant of $\pi$. 
\end{thm}

The arithmetic wavefront set $\wf_\ari(\pi)$ for any $\pi\in\Pi_\varphi[G^*]$ enriches the theory of wavefront sets for representations of $G(F)$ over any local field $F$. From the local Langlands 
reciprocity, it is natural to make the following {\it Wavefront Set Conjecture}.

\begin{conj}[Wavefront Set]\label{conj:wfss}
Let $F$ be a local field of characteristic zero and $G$ be a classical group over $F$ as considered in this paper. For any $\pi\in\Pi_F(G)$ with a generic local $L$-parameter, the following identities hold:
\[
\wf_\ari(\pi)^\mx=\wf_\wm(\pi)^\mx=\wf_\tr(\pi)^\mx,
\]
where $\wf_\square(\pi)^\mx$ denotes the subset of all maximal members in the wavefront set $\wf_\square(\pi)$, under the $F$-stable topological order over $\CN_F(\Fg)_\circ$
given in Definition \ref{defn:order}.
\end{conj}

It is clear that the local Langlands correspondence or reciprocity provides a bridge linking between the arithmetic data (local $L$-parameters and related factors) and the harmonic analysis data (representations and related local factors). The local Gan-Gross-Prasad conjecture for classical groups asserts that the certain branching properties 
(restriction problem) on the representation side (via the Bessel models and the Fourier-Jacobi models) can be detected 
by the sign of the local symplectic root numbers associated with the corresponding enhanced local $L$-parameters. From 
understanding along this line, the local {\bf Wavefront Set Conjecture} (Conjecture \ref{conj:wfss}) takes one step 
further from the local Gan-Gross-Prasad conjecture to determine the wavefront set $\wf_\wm(\pi)$ on the representation side by means of the wavefront set 
$\wf_\ari(\pi)$ defined by means of the corresponding local arithmetic data (the enhanced local $L$-parameters). 
In order to show that the arithmetic wavefront sets $\wf_\ari(\pi)$ can be completely computed  explicitly, we prove Theorems \ref{thm1} and \ref{thm2} when $F=\BR$. Hence Conjecture \ref{conj:wfss} provides an effective way to 
compute the wavefront sets $\wf_\wm(\pi)$, which will be substantial to the understanding the fundamental problems in the representation theory for classical groups defined over local fields of characteristic zero, which includes the unitary dual problem as explained by Adams and Vogan in \cite{AV21}. 
We mention that in the archimedean case \cite[Theorem F]{GGS17} gives a direct analog of Conjecture \ref{conj:wfss} for general linear groups. 

We expect that the arithmetic wavefront set $\wf_\ari(\pi)$ should be a subset of the algebraic wavefront set $\wf_\wm(\pi)$ or the analytic wavefront set $\wf_\tr(\pi)$, which 
will be discussed in our future work. 
One understanding of Conjecture \ref{conj:wfss} is that the maximal $F$-rational nilpotent orbits in either $\wf_\wm(\pi)$ or $\wf_\tr(\pi)$ 
should be completely determined by the consecutive descents of the enhanced $L$-parameter of $\pi$. 

Conjecture \ref{conj:wfss} will be restated with full details in Section \ref{ssec-MC} 
(Conjecture \ref{conj:main}). We show in Proposition \ref{exmp:generic} that if 
$G^*$ is an $F$-quasisplit classical group, and $\pi\in\Pi_F(G^*)$ is generic, then 
Conjecture \ref{conj:wfss} holds. The proof combines the theory of consecutive descents of 
enhanced $L$-parameters in Sections \ref{ssec-LPD} and \ref{ssec-WFSA}, with the work of 
Whittaker normalization of the generic local $L$-packets (\cite{Ar13}) (Shahidi's tempered $L$-packet conjecture) and the work of M\oe glin and Waldspurger (\cite{MW87}) and the work 
of Matumoto (\cite{Mat90,Mat92}). Combining Proposition \ref{exmp:generic} with 
the theory of consecutive descents of enhanced $L$-parameters in Sections \ref{ssec-LPD} and \ref{ssec-WFSA} and \cite[Lemma 3.1]{JZ18}, we show that when $F$ is $p$-adic and 
$G=\SO_{2n+1}$, if the algebraic wavefront set $\wf_\wm(\pi)$ contains an $F$-rational 
subregular nilpotent orbit in $\CN_F(\Fg)_\circ$, then Conjecture \ref{conj:wfss} holds for 
such a representation $\pi\in\Pi_F(G)$ (see Proposition \ref{prop:subr} for details). 
In \cite{JLZ-NA}, we discuss further evidence for Conjecture \ref{conj:wfss} by 
connecting our work in this paper with the work of J.-L. Waldspurger \cite{W18, W19, W20} on 
unipotent representations of $p$-adic $\SO_{2n+1}$. Indeed, an upper bound conjecture for wavefront sets, which is equivalent to the Jiang conjecture (\cite{J14}), is formulated 
in \cite{HLLS24}, which is expected to agree with the upper bound of our arithmetic wavefront sets associated with a local $L$-packet, and with the wavefront sets computed by
Waldspurger (for $\SO_{2n+1}$). This expectation has been verified in two families in \cite[Example 11.8]{HLLS24}. It will be an interesting problem to explore the relation of our work in this paper with the recent work of 
D. Ciubotaru, L. Mason-Brown, and E. Okada (\cite{CMBO21, CMBO22,  CMBO24, CMBO23}) for certain special unipotent 
representations of $p$-adic reductive groups, and with the work of J. Adams and D. Vogan 
on the explicit determination of the associated varieties for real reductive groups 
(\cite{AV21}). Also on the arithmetic side, weak Arthur packets are studied in a recent work \cite{GO24}, in terms of the geometry of the unipotent locus of the dual Langlands group, which should be compared with our results. 



The first part of this paper (Sections \ref{sec-CGRNO}--\ref{sec-AWFS}) focuses on the general theory of the arithmetic wavefront sets, while the second part of this paper (Sections \ref{sec-ADA}--\ref{sec-FRM}) is to determine the arithmetic wavefront sets $\wf_\ari(\pi)$ when $F$ is archimedean. Since $F=\BC$, if $\varphi$ is generic, 
then the local $L$-packet $\Pi_\varphi(G)$ 
contains only generic members. Hence we only need to work when $F=\BR$. In this case, we obtain 
precise structures as given in the following theorems, which illustrate the power of our consecutive descent method.  

\begin{thm}[Uniqueness]\label{thm1}
Let $F$ be an archimedean local field and 
$\varphi$ be a generic $L-parameter$ of $G_n^*$. For any 
$\pi$ belonging to the local Vogan packet $\Pi_\varphi[G_n^*]$, 
the $F$-rational nilpotent orbits in $\wf_\ari(\pi)^\mx$ determine 
the unique $F$-stable nilpotent orbit $\CO^\st(\pi)$ in $\CN_F(\Fg)_\circ^\st$ associated 
with the unique partition $\udl{p}(\pi)$. 
\end{thm}

As indicated in the work of C.-C. Tsai (\cite{Ts24}), such a uniqueness result may not hold in general when $F$ is non-archimedean, although the work (\cite{W18, W19, W20, CMBO24}) on special 
unipotent representations still suggests the uniqueness.  See also \cite{Ts23a} for subtle issues regarding different notions of wavefront sets,  and \cite{Ts23b} for a computation of wavefront sets of certain supercuspidal representations at the Lie algebra level in the flavor of the Shalika germ expansion and the Bruhat-Tits buildings. 


\begin{thm}[$F$-Rationality]\label{thm2}
Let $F$ be an archimedean local field and 
$\varphi$ be a generic $L-parameter$ of $G_n^*$. For any 
$\pi$ belonging to the local Vogan packet $\Pi_\varphi[G_n^*]$, 
the following hold.
\begin{enumerate}
    \item 
    If the component group $\CS_\varphi$ is trivial, then $\wf_\ari(\pi)^\mx$ consists of all the $F$-rational regular nilpotent orbits in $\CN_F(\Fg)_\circ$.
    \item 
    If the component group $\CS_\varphi$ is nontrivial, then
    \[
    \wf_\ari(\pi)^\mx = \{\CO(\pi)\},
    \]
    with $\CO(\pi):=\CO_a(\varphi,\chi)$ if $\pi=\pi_a(\varphi,\chi)$ for some $\chi\in\wh{\CS_\varphi}$, where $\CO_a(\varphi,\chi)$ 
    is the unique $F$-rational nilpotent orbit as given in Definition \ref{defn:maxtab}.
    \item 
    Every $F$-rational nilpotent orbit in $\wf_\ari(\pi)^\mx$  is $F$-distinguished 
    in the sense of Definition \ref{defn:dist-orbit}.
\end{enumerate}
\end{thm}

It is important to mention that compared with the calculation by means of the atlas project in \cite{AV21}, the uniqueness of $F$-rational nilpotent orbit $\CO(\pi)$ in Part (2) of Theorem \ref{thm2} is outstanding. In Theorem \ref{cor:wf-max-unique-archimedean}, we show that $\CO_a(\varphi,\chi)$ is the 
unique maximal orbit in $\wf_a(\varphi,\chi)$ under the $F$-stable topological order over $\CN_F(\Fg)_\circ$ (as defined in Definition \ref{defn:order}), while 
in Theorem \ref{thm:maximal-rational-order}, we show that 
$\CO_a(\varphi,\chi)$ is the 
unique maximal orbit in $\wf_a(\varphi,\chi)$ under the $F$-rational topological order over $\CN_F(\Fg)_\circ$ (as defined in Definition \ref{defn:order}). It is important to mention that at the representation side, B. Harris proves in 
\cite{Hb12}, Part (3) of Theorem \ref{thm2} holds when $\pi$ is tempered. Since tempered representations have generic $L$-parameters, Part (3) of Theorem \ref{thm2} covers more general situation. 
We refer to \cite{M96} for relevant discussion for the non-archimedean case. 

In the process of proving Theorem \ref{thm1}, we develop much refined structures for 
the arithmetic wavefront sets in the archimedean case, compared with the general structures discussed 
in Section \ref{sec-AWFS}. The key step to develop such refined structures is Proposition 
\ref{prop:fd}, which shows that for any enhanced $L$-parameter $(\varphi,\chi)$ with $\varphi$ generic, the first descent $\FD_{\Fl_0}(\varphi,\chi)$ consists of only a finite number of 
enhanced $L$-parameters $(\phi,\chi')$ with all $\phi$ explicitly determined by $(\varphi,\chi)$, when $F$ is archimedean. Compared with Theorem \ref{thm:DFD}, 
Proposition \ref{prop:fd} yields a much refined structure due to the nature of archimedean 
local Langlands parameters. The significance of Proposition \ref{prop:fd} and also 
Theorem \ref{thm:DFD} in the representation theory of classical groups over local fields 
is discussed in detail in \cite[Theorems 1.4, 1.5, 1.6]{CJLZ24}). On the way to proving Theorem \ref{thm1}, we prove 
a stronger result (Theorem \ref{thm:main}). The proof of Theorem \ref{thm:main} reduces to 
a technical lemma (Lemma \ref{lem8.5}), which is of combinatorial nature and whose proof needs explicit information about the relation between the descents of enhanced $L$-parameters and the {\it collapse} of enhanced $L$-parameters as developed in Section \ref{sec-CO}. 
The main refined structure is about the $F$-rational structure of the nilpotent orbits in 
$\wf_a(\varphi,\chi)^\mx=\wf_\ari(\pi)^\mx$ when $\pi=\pi_a(\varphi,\chi)$.
It is Theorem \ref{cor:wf-max-unique-archimedean}, which says that 
$\wf_a(\varphi,\chi)^\mx$, under the $F$-stable topological order over $\CN_F(\Fg)_\circ$, consists of a single $F$-rational $F$-distinguished nilpotent orbit $\CO_a(\varphi,\chi)$ if the component group $\CS_\varphi$ is nontrivial. This maximality is refined in Theorem \ref{thm:maximal-rational-order} under the $F$-rational topological order over $\CN_F(\Fg)_\circ$, which yields the final 
$F$-rational structure of the arithmetic wavefront set $\wf_\ari(\pi)$. 

It is very important to mention that we use as an input the Vogan version of the local Langlands correspondence for the classical groups considered in this paper. Since it is known over 
archimedean local fields, the archimedean theory of arithmetic wavefront sets developed in this paper is unconditional. As explained in \cite{K23}, the Vogan version of the local Langlands conjecture 
for even special orthogonal groups over non-archimedean local fields is still in progress, for instance, although it is known in many situations. Hence for accuracy, we may have to take it as an assumption for those incomplete cases, when we discuss the theory over non-archimedean local fields. 

From the computation point of view, D. Johnson computed the associated varieties for irreducible Harish-Chandra modules for $\SU(p,q)$ based on the Robinson-Schensted algorithm in \cite{Jns}. 
Based on our results in Section \ref{sec-PCR}, we also make a program based on Mathematica to compute the arithmetic wavefront for the discrete series of $\RU(p,q)$, which is available upon request. Our program obtains the same signed Young diagram as Johnson did in \cite{Jns}, which provides strong numerical evidence for the Wavefront Set Conjecture (Conjecture \ref{conj:wfss} or more precisely Conjecture \ref{conj:main}) for the discrete series. A proof of the Wavefront Set Conjecture for 
$\wf_\ari(\pi)$ and $\wf_\tr(\pi)$ has recently been found for real classical groups in a forthcoming work of D. Liu, Z. Luo, J. Ma and L. Zhang (2025). In fact, it is proved that as in 
\eqref{alg-ana-wfs}, the Wavefront Set Conjecture for $\wf_\ari(\pi)$ and $\wf_\tr(\pi)$ holds for $\BR$-rational topology, which is a refinement of Conjecture \ref{conj:wfss}. We believe that such a refinement holds in general.

It is clear that the notion of arithmetic wavefront sets and the relevant conjectures can be formulated without difficulty over finite fields. In this case, we refer to an excellent work of Z. Wang in \cite{Wan24} (see also \cite{PW23}), where the approach and algorithm, as developed in this paper, are capable of explicitly constructing the wavefront set $\wf_\ari$ and verifying the Wavefront Set Conjecture \ref{conj:wfss}. 
With the complete results of Wang for finite field case, it is natural to expect that the methods and algorithms as developed in this paper can be implemented for all local fields in general.

Finally, let us make some remarks on the approach we take towards the theory of wavefront sets of irreducible admissible representations of general reductive groups $G$ over local fields $F$. 
The main idea is to understand the relation between the arithmetic data on the one hand and the harmonic analysis on the other hand via the local Langlands correspondence. The main contribution of this paper is to set up the theory on the arithmetic side, which can be completely proved for the archimedean case. The non-archimedean case, with strong evidence provided, remains to be done in general, which will be considered in our future work. Meanwhile, it is possible to extend our theory from generic $L$-parameters to more general local $L$-parameters, which is the case through our calculation of some special cases. It is a more serious problem to extend the theory to cover general reductive algebraic groups over any local fields. It is our goal to figure out 
the impacts of this arithmetic theory to the harmonic analysis and representations of $G(F)$. 
Naturally, it is important to 
introduce the notion of arithmetic wavefront set for automorphic representations of 
$G(\BA)$ for any reductive algebraic group $G$ defined over any number fields or global fields $k$ ($\BA$ is the ring of adeles of $k$), which will be discussed in our forthcoming work. 
Some preliminary work towards this global theory was already set up in \cite{J14} with a recent progress in \cite{JL24, LS22, LS, LLS24, CK24}. It is important to mention that 
J. Adams and D. Vogan made a conjecture on the global coherence on the local wavefront sets for automorphic representations (\cite[Conjecture 3.11]{AV21}) based on the local analytic wavefront sets. It is clear that the wavefront set conjecture (Conjecture \ref{conj:wfss}) 
will explain the relation between the analytic approach in \cite{AV21} and the arithmetic approach in this paper. 

\quad

Let us describe the organization of the paper. 
In Section \ref{sec-CGRNO}, we recall the classical groups $G$ considered in this paper, which include $F$-quasisplit groups and their pure inner forms. The discussion in this paper also includes the metaplectic double cover $\Mp_{2n}(F)$ of the symplectic group $\Sp_{2n}(F)$. We discuss the $F$-stable nilpotent orbits and $F$-rational nilpotent orbits in the Lie algebra $\Fg(F)$ of $G(F)$. There are several ways to parameterize the $F$-rational nilpotent orbits in $\Fg(F)$. 
In Section \ref{ssec-RNO}, we follow the work of R. Gomez and
C.-B. Zhu (\cite{GZ14}) to parameterize the $F$-rational nilpotent orbits by admissible $\epsilon$-Hermitian Young tableaux as defined in \eqref{defn:SYT} and Definition  \ref{defn:AYT}.

In Section \ref{sec-LLC}, we first recall the formulation of the local Langlands conjecture for $G$ over local fields $F$. We write out the structure of local $L$-parameters and recall
the local $L$-packets, including the local Vogan packets in Section \ref{ssec-LVP}.
In Section \ref{ssec-RLLC} we discuss the $F$-rational structures through the local Langlands correspondences, which are important to the formulation of the local Gan-Gross-Prasad conjecture 
(Section \ref{ssec-LGGP}) and the theory of consecutive descents of enhanced $L$-parameters (Section \ref{sec-AWFS}). For those arithmetic applications, we also consider the 
contragredient of enhanced $L$-parameters in Section \ref{ssec-C}, and write down twisted 
distinguished characters in Section \ref{ssec-DC}. 

Section \ref{ssec-LGGP} is to write the explicit arithmetic information from the 
local Gan-Gross-Prasad conjecture for classical groups. 
For the formulation of the local Gan-Gross-Prasad conjecture and for the definition of algebraic wavefront sets, we introduce the twisted Jacquet modules associated with any nilpotent orbit in Section \ref{ssec-TJM}. After writing the explicit $F$-rational structure for the nilpotent orbits associated to the partitions of type $[p_1,1^{\Fn-p_1}]$ in Section 
\ref{ssec-RNO-D}, we write down explicit data defining the twisted Jacquet modules of Bessel type and of Fourier-Jacobi type in Section \ref{ssec:TJMp1}. Section \ref{ssec-ED} provides 
the explicit arithmetic data from the local Gan-Gross-Prasad conjecture, case by case. 

In Section \ref{sec-AWFS} we develop the main framework of the consecutive descents of enhanced $L$-parameters, which serves the base for the definition of the arithmetic wavefront sets for $\pi\in\Pi_F(G)$. It contains the main conjecture and the basic properties of the arithmetic wavefront sets. 
Sections \ref{ssec-LPD} and \ref{ssec-WFSA} form a theory of consecutive descents of 
enhanced $L$-parameters of generic type. 
Definition \ref{defn:pd} defines the descents of enhanced $L$-parameters $(\varphi,\chi)$ with $\varphi$ generic, which is denoted by $\FD_\Fl(\varphi, \chi)$ 
as in \eqref{Fl-pd}, based on explicit calculations of the local root numbers in the local Gan-Gross-Prasad conjecture as explained in Section \ref{ssec-LGGP}. In Section \ref{ssec-WFSA}, 
we introduce the notions of {\it pre-tableaux} $\Fs_{\udl\Fl}$ 
in \eqref{Fs}, 
{\it pre-tableaux for enhanced $L$-parameters} $\CT_a(\varphi,\chi)$ 
in Definition \ref{LYT}, 
{\it $L$-descent pre-tableaux} $\CL_a(\varphi,\chi)$ in Proposition \ref{prop:LDPT}, 
{\it $L$-descent Young tableaux} $\CY_a(\varphi,\chi)$ before Definition \ref{defn:AWFS}, 
and {\it $L$-descent partitions} $\CP_a(\varphi,\chi)$ before Proposition  \ref{prop:Lp}. 
They form basic steps towards our 
definition of {\it wavefront sets for enhanced $L$-parameters} $\wf_a(\varphi, \chi)$ in 
Definition \ref{defn:AWFS}. 
The wavefront set $\wf_a(\varphi, \chi)$ for the enhanced $L$-parameter $(\varphi, \chi)$ has nice invariance properties against the normalization of the
rationality of the local Langlands correspondence and the arithmetic data in the local 
Gan-Gross-Prasad conjecture (Theorem \ref{inv-pi}). Hence we are able to define the {\it arithmetic wavefront set} $\wf_\ari(\pi)$ in \eqref{ariwfs}, which is an invariant of $\pi$ 
(Theorem \ref{inv-pi}). Proposition \ref{prop:LDPT} asserts that the set $\CL_a(\varphi,\chi)$ is non-empty, the archimedean case of which is Proposition \ref{prop:DNE-A}, while the proof 
for non-archimedean case involves both local and global arguments and is left to Appendix \ref{App-A}. 
Proposition \ref{inv-prop} asserts that the set $\CT_a(\varphi,\chi)$ 
of pre-tableaux for enhanced $L$-parameters has the expected invariance properties against the normalization of the rationality of the local Langlands correspondence and the arithmetic data in the local Gan-Gross-Prasad conjecture, which forms a key step towards the proof of 
Theorem \ref{inv-pi}. 
In Section \ref{ssec-WFS-ZOrbits}, we introduce the notion of wavefront set associated to 
$\CZ$-orbits of enhanced $L$-parameters $\wf_a(\varphi,\ScO_\chi)$ (Definition \ref{WSO}) and study their basic properties (Proposition \ref{inv-O}). They may not be directly relevant to the main theme of this paper, but we believe that they are interesting for further exploration.
In Section \ref{ssec-MC}, we write down the wavefront set conjecture (Conjecture \ref{conj:wfss}) with full details, i.e. Conjecture \ref{conj:main}. Propositions 
\ref{exmp:generic} and \ref{prop:subr} provide some evidence supporting Conjecture \ref{conj:main}. 

Section \ref{sec-ADA} is a preparation for the proof of Theorem \ref{thm1} and Theorem \ref{thm2}, 
which contains the explicit information that are specializations to the archimedean case of the arithmetic data developed in Sections \ref{sec-LLC} and \ref{ssec-LGGP} from the rationality of the local Langlands correspondence and the local Gan-Gross-Prasad conjecture for classical groups. 

Section \ref{sec-PCR} is devoted to the proof of Theorem \ref{thm1}. In Section \ref{ssec-FD}, we are able to present
in Proposition \ref{prop:fd} the explicit $L$-parameter descent at the first occurrence index, with explicit
description of the local $L$-parameters and distinguished characters. Those explicit arithmetic data from Proposition \ref{prop:fd} serve as a base for the construction of the unique partition $\udl{p}(\varphi,\chi)$ needed for Theorem \ref{thm1} (Proposition \ref{prp:partition}), and as the key step towards the construction of 
the $F$-rational structure of the set $\wf_a(\varphi,\chi)^\mx$, which yields the unique $F$-stable orbit $\CO^\st(\varphi,\chi)$ in $\CN_F(\Fg)_\circ^\st$ corresponding 
to the partition $\udl{p}(\varphi,\chi)$.
In order to prove Theorem \ref{thm2}, it is enough to prove the maximality of the partition $\udl{p}(\varphi,\chi)$ in the set $\CP_a(\varphi,\chi)$ of $L$-descent partitions associated to the enhanced $L$-parameter $(\varphi,\chi)$. We do this by proving a stronger result in Theorem \ref{thm:main}, whose proof is reduced to Lemma \ref{lem8.5}. Lemma \ref{lem8.5} is a technical result in combinatorics and will be proved in Section \ref{sec-CO} by means of the understanding of the collapses and the descents of enhanced $L$-parameters. Finally, we discuss the $F$-rational nilpotent orbits in 
the wavefront set $\wf_a(\varphi,\chi)$ under the $F$-stable topological order over $\CN_F(\Fg)_\circ$ (Theorem \ref{cor:wf-max-unique-archimedean}), which is refined in Section \ref{sec-FRM} 
under the $F$-rational topological order over $\CN_F(\Fg)_\circ$ (Theorem \ref{thm:maximal-rational-order}). The proof of Theorem \ref{thm:maximal-rational-order} occupies Section \ref{sec-FRM}, which consists of two subsections that treat the two cases: $E=\BR$ and 
$E=\BC$ separately. The technical input is the notion of signed Young diagrams and their 
partial ordering as introduced by D. Djokovi\'c in \cite{D81, D82}. 

Appendix \ref{App-A} proves the non-archimedean case of Proposition \ref{prop:LDPT} that the set $\CL_a(\varphi,\chi)$ of $L$-descent pre-tableaux is nonempty for any enhanced $L$-parameter $(\varphi,\chi)$ with $\varphi$ generic. 


\quad

It is important to mention that the basic arithmetic datum $(\varphi,\chi)$ in this introductory section will be written as $(\varphi,\chi,q)$ in the main body of this paper with an $\Fn$-dimensional non-degenerate $\epsilon$-Hermitian form $q$, which 
defines the classical group $G$. From the discussion in the main body of this paper, $q$ is completely determined by the enhanced $L$-parameter $(\varphi,\chi)$ and the choice of the rationality of 
the local Langlands conjecture related to $a$. 

\quad

We would like to thank T. Kaletha for his communication about the current state of art on the Vogan version of the local Langlands correspondence related to his more general version as discussed in \cite{K23}; to thank D. Vogan for his warm encouragement when we were writing up this paper and providing us some important references closely related to the project; to thank J. Adams for his informative emails which provide some background references on the subject and for his explicit atlas computations of some lower rank cases by using their method in \cite{AV21}; to thank B. Gross for giving us highly encouraging feedback on a earlier version of this paper and sending us his recent paper (\cite{G22}); to thank Jiajun Ma for bringing our attention to the software based computation by D. Johnson in \cite{Jns}; to thank C.-C. Tsai for sending us his very interesting paper (\cite{Ts24}); and to thank A.-M. Aubert, R. Howe, M. Nevins, T. Przebinda, M.-F. Vigneras, and C.-B Zhu for their interests and helpful comments on an earlier version of this paper. Last but not least, We are very grateful to the referee of this paper for his/her valuable comments and thoughtful suggestions.

\section{Classical Groups and Rational Nilpotent Orbits}\label{sec-CGRNO}


\subsection{Classical groups and their pure inner forms}\label{ssec-CGIF}
The classical groups considered in this paper include the unitary groups  $\RU_{\Fn}$, special orthogonal groups $\SO_{\Fn}$, symplectic groups $\Sp_{2n}$ and metaplectic groups $\Mp_{2n}$,
following the notations in \cite{JZ17, JZ18, JZ20}, which are compatible with those in \cite{GGP12, Ar13, KMSW14, Mok15, AG17a}.

We assume that the local field $F\ne \BC$ in this paper. The reason is that since the generic $L$-packets of complex classical groups contain only one element, which is generic,
the discussion in this paper is trivial when $F=\BC$. Thus if $F$ is archimedean, we mean that $F$ is real; otherwise, we mean that $F$ is $p$-adic, which is a finite extension of the field $\BQ_p$ of $p$-adic numbers for some prime $p$.
Let $F(\delta)$ be a quadratic field extension of $F$, with $\delta=\sqrt{d}$ for a non-square $d\in F^\times$.
Let $E$ be either $F$ or $F(\delta)$.
Denote by $c\colon x\mapsto \bar{x}$ \label{pg:iota} the unique nontrivial element in the Galois group $\Gal(E/F)$ if $E\ne F$; and $c={\rm id}_E$, $\bar{x}=x$ if $E=F$.

Let $(V,q_V)$ be an $\Fn$-dimensional vector space $V$ over $E$, equipped with a  non-degenerate $\epsilon$-Hermitian form
$q=q_V$ with $\epsilon=\pm 1$.
More precisely, if $E=F$, $q$ is symmetric or symplectic; and
if $E=F(\delta)\neq F$, $q$ is Hermitian or skew-Hermitian.
Write $\eps_q=-1$ if $q$ is symplectic or skew-Hermitian; otherwise $\eps_q=1$.
Denote by $G_n=\Isom(V,q)^\circ$ (or $\Mp_{2n}(F)$ in the metaplectic case) the identity connected component of
the isometry group of the space $(V,q)$, with $n=\lfloor\frac{\Fn}{2}\rfloor$.
Let $\Fr$ be the Witt index of $(V,q)$
and $(V_{\an},q)$ be the anisotropic kernel of $(V,q)$ with dimension $\Fd_0$.
Then the $F$-rank of $G_n$ is $\Fr$ and $\Fd_0=\Fn-2\Fr$.
Consider the following polar decomposition for $V$
$$
V=V^{+}\oplus V_{\an}\oplus V^{-},
$$
where $V^{\pm}$ are maximal totally isotropic subspaces dual to each other.
Take a basis $\{e_{\pm 1},\dots,e_{\pm \Fr}\}$ for $V^{\pm}$ such that
$$
q(e_{i},e_{-j})=\delta_{i,j}
$$
for all $1\leq i,j\leq\Fr$.
Choose an orthogonal basis $\{e'_{1},\dots,e'_{\Fd_0}\}$ of $V_{\an}$ and put
$$
d_i=q(e'_{i},e'_{i})\in F^\times \text{ or }\delta\cdot F^\times, \text{ for }1\leq i\leq \Fd_0.
$$
Note that if $(V,q)$ is a Hermitian space, then
$(V,\delta \cdot q)$ is skew-Hermitian and $\Isom(V,q)=\Isom(v,\delta\cdot q)$ gives us the identical unitary groups.
However, for the induction purpose, we will consider both cases in this paper.
If $(V,q)$ is symplectic, then $\Fn=2n=2\Fr$ and $V_{\an}=\{0\}$.
We put the above bases together in the following order to form a basis of $(V,q)$
\begin{equation}\label{eq:basis}
\FB\colon e_{1},\dots,e_{\Fr},e'_{1},\dots,e'_{\Fd_0},e_{-\Fr},\dots,e_{-1},
\end{equation}
and fix the following full isotropic flag in $(V,q)$:
$$
\Span\{e_{1}\}\subset\Span\{e_{1},e_{2}\}\subset
\cdots\subset
\Span\{e_{1},\dots,e_{\Fr}\},
$$
which defines a minimal parabolic $F$-subgroup $P_0$ of $G_n$.
With respect to the order of the basis in \eqref{eq:basis}, the group $G_n$ is also defined by
the following symmetric (or symplectic) matrix:
\[
J_{\Fr}^\Fn=\begin{pmatrix}
&&w_\Fr\\&J_{0}^{\Fd_0}&\\ \epsilon_q w_{\Fr}&&
\end{pmatrix}_{\Fn\times\Fn}
\text{ where }
w_\Fr=\begin{pmatrix}
&w_{\Fr-1}\\1&	
\end{pmatrix}_{\Fr\times \Fr},
\]
which is defined inductively, and
$J_{0}^{\Fd_0}=\diag\{d_{1},\dots,d_{\Fd_0}\}$.
Moreover, the Lie algebra $\Fg_n$ of $G_n$ is defined by
\[
\Fg_n=\{A\in {\rm {End}}_E(V)\ \mid\ \bar{A}^{t}J_{\Fr}^\Fn+ J_{\Fr}^\Fn A=0\},	
\]
with the Lie bracket $[A,B]=AB-BA$.
For simplicity, set $A^*=(J_{\Fr}^\Fn)^{-1} \bar{A}^{t} J_{\Fr}^\Fn$.
In this paper, we take the following $\Ad(G_n)$-invariant non-degenerate $F$-bilinear form $\kappa$ on $\Fg_n\times \Fg_n$:
\[
\kappa(A,B)=\tr(AB^*)/2=\tr(A^*B)/2.	
\]
Alternatively, $\kappa$ can be computed by 
\begin{equation}\label{eq:kappa}
2\kappa(A,B)= \sum_{i=1}^{\Fr}\apair{A(e_{i}),B(e_{-i})}+\epsilon_q \apair{A(e_{-i}),B(e_{i})}+\sum_{j=1}^{\Fd_0}\frac{\apair{A(e'_{j}),B(e'_{j})}}{\apair{e'_{j},e'_{j}}}.	
\end{equation}
Note that this $\Ad(G_n)$-invariant symmetric bilinear form is the same as the one in \cite{GZ14} and is proportional to the Killing form.
For an $\eps$-Hermitian space $V$ of dimension $\Fn$, we define its discriminant by
\[
\disc(V)=(-1)^{\frac{\Fn(\Fn-1)}{2}}\det(V) \in   \begin{cases}  \delta^\Fn \cdot F^\times /\BN E^\times,	& \textrm{if }E=F(\delta)\textrm{ and }\epsilon = -1, \\
F^\times / \BN E^\times, & \textrm{otherwise,}
\end{cases}
\]
where $\BN E^\times:=\{x\bar{x}\ \mid\  x\in E^\times\}$, unless $E=F$
 and $V$ is a symplectic space in which case this definition is not applicable.

\subsection{Rational nilpotent orbits}\label{ssec-RNO}

Let $\CN_F(\Fg_n)$ be the set of $F$-rational nilpotent elements in the Lie algebra $\Fg_n(F)$. We denote by $\CN_F(\Fg_n)_\circ$ the set of all $F$-rational orbits on $\CN_F(\Fg_n)$ under the adjoint action of $G_n(F)$ and by $\CN_F(\Fg_n)_\circ^\st$ the set of $F$-stable orbits on $\CN_F(\Fg_n)$. 
Recall that $X_1, X_2\in \Fg_n(F)$ are $F$-stably conjugate if there exists $g\in G_n(\ol{F})$ such that ${\rm Ad}(g)(X_1)=X_2$, where $\ol{F}$ is an algebraic closure of $F$. Thus each $F$-stable orbit is a finite union of $F$-rational orbits.

\begin{defn} \label{defn:partition}
An ordered partition of $\frak n$ is a sequence $\udl{\Fl}= (\ell_1,\ell_2,\ldots, \ell_k)$ of positive integers such that $\sum^k_{i=1}\ell_i = \frak{n}$.
If moreover $\ell_1\geq \ell_2\geq \cdots \geq \ell_k$, then we call $\udl{\Fl}$ a decreasing partition and write 
\[
\udl{\Fl} = [\ell_1, \ell_2,\ldots, \ell_k].
\]
\end{defn}

Throughout this paper, by partitions we simply mean decreasing partitions. 
It is a classical theorem that the set $\CN_F(\Fg_n)_\circ^\st$ of
$F$-stable nilpotent orbits is finite and is parameterized by partitions of $\Fn$ of certain type according to the structure of $\Fg_n$. For this paper, the set $\CN_F(\Fg_n)_\circ$ of
$F$-rational nilpotent orbits of $\Fg_n(F)$ plays an essential role.
The classification of the set $\CN_F(\Fg_n)_\circ$ has been discussed with details in \cite{Car85, CM93},  \cite[Section I.6]{W01} and \cite{GZ14} for instance.

For a partition $\underline{p}$ of $\Fn$,
we also use the notation of exponential form
\begin{equation}\label{eq:partition}
\underline{p}=[p_1^{m_1},p_2^{m_2},\dots,p_r^{m_r}],	
\end{equation}
where $p_1> p_2>\cdots>p_r>0$, $m_i$ is the multiplicity of $p_i$ in $\underline{p}$
and $\sum_{i=1}^{r}m_ip_i=\Fn$. 
For an $\Fn$-dimensional non-degenerate $\epsilon$-Hermitian vector space $(V,q)$ with $G_n=\Isom(V,q)^\circ$, we define $\CP(V,q)$ to be the set of partitions $\underline{p}$
with the following constraints:
\begin{itemize}
	\item if $G_n$ is symplectic  or metaplectic, $m_i$ is even for all odd parts $p_i$;
	\item if $G_n$ is orthogonal, $m_i$ is even for all even parts $p_i$ except that the totally even partitions (those with $p_i$ even and $m_{i}$ even for all $i$) correspond to two stable orbits, which are labeled by $\underline{p}_I$ and $\underline{p}_{II}$ respectively.
\end{itemize}
Note that if $G_n$ is unitary, there is no constraint on $\underline{p}$.
From \cite{CM93}, there exists a canonical correspondence between the set $\CN_F(\Fg_n)_\circ^\st$ and the set $\CP(V,q)$ that is is one-to-one except for 
totally even partitions when $G_n$ is a special even orthogonal group (over which it is two-to-one). This correspondence can be explicitly constructed as
follows.

 Let $\CO\subset\CO^\st$ be an $F$-rational orbit contained in the  $F$-stable nilpotent orbit $\CO^\st$ in $\CN_F(\Fg_n)_\circ^\st$. Take $X\in\CO$ and form $\{X,\hbar,Y\}$, an ${\mathfrak{sl}}_2$-triple in $\Fg_n(F)$.
 The restriction of the $\Fg_n(F)$-module $(V,q)$ to the ${\mathfrak {sl}}_2$-triple yields 
 the following decomposition of the space $V$:
\begin{equation}\label{eq:V-p}
V=\bigoplus_{i=1}^r V_{(p_i)}	
\end{equation}
where $V_{(p_i)}$ is the isotypic component of an irreducible $p_i$-dimensional representations $E^{p_i}$ of ${\mathfrak {sl}}_2$ over $E$ and
$p_1>p_2>\cdots>p_r>0$.
Denote by $m_i$  the multiplicity of $E^{p_i}$ in $V_{(p_i)}$, i.e., $V_{(p_i)}\cong m_{i}E^{p_i}$.
We obtain in this way the partition $\underline{p}=[p_1^{m_1},p_2^{m_2},\dots, p_r^{m_r}]$ of $\Fn=\dim(V)$ associated to $\CO$, 
where $m_i$ and $p_i$ are uniquely determined by $V_{(p_i)}$. We claim that the so obtained partition $\udl{p}$ belongs to $\CP(V,q)$. In fact, when $G_n$ is symplectic, if a part $p_i$ is odd, then
the irreducible representation $E^{p_i}$ of ${\mathfrak {sl}}_2$ is of orthogonal type, and hence the multiplicity $m_i$ must be even. The same argument is applicable to the case when $G_n$ is orthogonal. If one takes a different $X'\in\CO$, then one produces an equivalent decomposition of the vector space $V$, and hence obtains the same partition $\udl{p}$. This yields the
correspondence from $\CO$ to $\udl{p}$. By definition, any two $F$-rational orbits $\CO_1$ and $\CO_2$ in the $F$-stable orbit $\CO^\st$ are conjugate to each other by $G_n(\ovl{F})$, where $\ovl{F}$ is the algebraic closure of $F$. From the discussion above, it is clear that $\CO_1$ and $\CO_2$ must produce the same partition $\udl{p}$.
Hence the above construction yields the explicit correspondence between the set $\CN_F(\Fg_n)_\circ^\st$ and the set $\CP(V,q)$.

In order to parameterize the $F$-rational orbits in $\CN_F(\Fg_n)_\circ$, we recall from \cite{GZ14} the notion of {\sl sesquilinear Young tableau} associated to the partition $\udl{p}$. 
For a given partition $\underline{p}=[p_1^{m_1},p_2^{m_2},\dots, p_r^{m_r}]$ of $\Fn$ as in \eqref{eq:partition} with $p_1> p_2>\cdots>p_r>0$, to each $p_i^{m_i}$, 
we first assign an $m_i$-dimensional $\eps_i$-Hermitian space $(U^{h}_{(p_i)},q^{h}_{(p_i)})$ for $1\leq i\leq r$.
In order to assign the rationality for the whole $p_i^{m_i}$, we have to introduce a Hermitian structure for the $p_i$-dimensional space $E^{p_i}$ over $E$.
For any $k\geq 1$, the  $k$-dimensional irreducible representation $E^k$ of ${\mathfrak{sl}}_2$ over $E$
is equipped an invariant Hermitian form $q_k$ of sign $(-1)^{k-1}$, which is unique up to scalar.
When $k$ is odd, choose $q_k$ such that $\disc(q_k)=1$. We then assign the rationality to $p_i^{m_i}$ with $i=1,2,\ldots,r$ to be
the $(-1)^{p_i-1}\eps_i$-Hermitian space given by
\begin{equation}\label{eq:U-h-full}
\left(U^{h}_{(p_i)}\otimes_E E^{p_i}, q^{h}_{(p_i)}\otimes q_{p_i}\right).
\end{equation}
The datum 
\begin{align}\label{defn:SYT}
\left(\udl{p},\{(U^{h}_{(p_i)}, q^{h}_{(p_i)})\}\right)
\end{align}
is called a {\it sesquilinear Young tableau} associated to the partition $\udl{p}$, following \cite[Definition 3.2]{GZ14}. Two sesquilinear Young tableaux 
\[
\left(\udl{p},\{(U^{h}_{(p_i)}, q^{h}_{(p_i)})\}\right)\quad 
\text{and}\quad 
\left(\udl{p}',\{(W^{h}_{(p'_j)}, q'^{h}_{(p'_j)})\}\right)
\]
are called {\it equivalent} if $\underline{p}=\underline{p}'$, and for each index $i$ with $i=1,2,\ldots,r$, the spaces $(U^{h}_{(p_i)},q^{h}_{(p_i)})$ and $(W^{h}_{(p_i)},q'^{h}_{(p_i)})$ are isometric.

When the partition $\ul{p}$ belongs to $\CP(V,q_V)$, following \cite[Definition 3.4]{GZ14}, we have the {\sl admissibility} 
with the given $\eps$-Hermitian vector space $(V,q_V)$ of a sesquilinear Young tableau  
associated to the partition $\udl{p}$.

\begin{defn}\label{defn:AYT}
Given an $\Fn$-dimensional $\eps$-Hermitian vector space $(V,q_V)$,
a sesquilinear Young tableau 
$\left(\udl{p},\{(U^{h}_{(p_i)}, q^{h}_{(p_i)})\}\right)$ 
as defined in \eqref{defn:SYT} is called {\rm admissible} for $(V,q_V)$ if 
\begin{enumerate}
\item the partition $\udl{p}$ belongs to $\CP(V,q_V)$;
\item the Hermitian space
\[
\left(U^{h}_{(p_i)}\otimes_E E^{p_i}, q^{h}_{(p_i)}\otimes q_{p_i}\right)
\]
as defined in \eqref{eq:U-h-full} for each $p_i^{m_i}$ is $\eps$-Hermitian for all $1\leq i\leq r$; and 
\item the $\eps$-Hermitian space 
\[
\left(\bigoplus^r_{i=1}U^{h}_{(p_i)}\otimes_E E^{p_i}, \bigoplus^r_{i=1}q^{h}_{(p_i)}\otimes q_{p_i}\right)
\]
is isometric to the given $\Fn$-dimensional $\eps$-Hermitian vector space $(V,q_V)$.
\end{enumerate}
\end{defn}

Now, we are ready to parametrize the $F$-rational orbits in $\CN_F(\Fg_n)_\circ$ in terms of 
admissible $\eps$-Hermitian Young tableaux based on the above orbit-partition correspondence 
between the set $\CN_F(\Fg_n)_\circ^\st$ and the set $\CP(V,q)$.
As above, we take $X\in\CO$ and form $\{X,\hbar,Y\}$, an ${\mathfrak{sl}}_2$-triple in $\Fg_n(F)$. For $j\in\BZ$, set
\begin{equation}\label{eq:V-i}
V_j=\{v\in V\ \mid\ \hbar(v)=jv\}	
\end{equation}
to be the weight space of $\hbar$ with weight $j$.
Note that for $j\geq 0$, $X^j\vert_{V_{-j}}\colon V_{-j}\to V_{j}$ is invertible and its inverse is denoted by $X^{-j}\vert_{V_{j}}$.
Following \eqref{eq:V-p}, we define
\begin{equation}\label{eq:X-V-i}
V^h_{(p_i)}:=V_{(p_i)}\cap V_{p_i-1}	
\end{equation}
which is the highest weight subspace of the isotypic component $V_{(p_i)}$.
Define the  $(-1)^{p_i-1}\eps$-Hermitian form $q_{(p_i)}^h$ on $V^h_{(p_i)}$ by
\begin{equation}\label{eq:X-q-i}
q^h_{(p_i)}(v,w)=q_V\left(X^{-(p_i-1)}\vert_{V^h_{(p_i)}}(v), w\right) \text{ for all }v,w\in 	V^h_{(p_i)}.
\end{equation}
From \cite[Section 3.1]{GZ14}, $q_{(p_i)}^h$ is a well-defined non-degenerate form on $V^h_{(p_i)}$.
In particular, following \cite[(3.6)]{GZ14}, when $p_i$ is odd, we have
\begin{equation}\label{eq:V-pi-odd}
(V^h_{(p_i)},q^h_{(p_i)})\cong \left(V_{0}\cap V_{(p_i)},(-1)^{\frac{p_i-1}{2}}q\vert_{V_{0}\cap V_{(p_i)}}\right).	
\end{equation}
At this point, we obtain an assignment that for each $i=1,2,\ldots,r$, it assigns to the $p_i$-parts $p_i^{m_i}$ of the partition $\udl{p}$ a non-degenerate $(-1)^{p_i-1}\eps$-Hermitian space
$(V^h_{(p_i)},q^h_{(p_i)})$. 
The partition $\udl{p}$ together with the assignment is a {\it sesquilinear Young tableau}, as defined in \eqref{defn:SYT}. 

Finally, we have to establish the admissibility of the constructed $\eps$-Hermitian Young tableau associated to the ${\mathfrak{sl}}_2$-triple $\{X,\hbar,Y\}$ in $\Fg_n$, based on \eqref{eq:X-q-i}. 
By the choice of the $(-1)^{p_i-1}$-Hermian form $q_{p_i}$ on $E^{p_i}$, we have
\begin{align}\label{Vpi}
(V_{(p_i)},q_{(p_i)})=(V_{(p_i)},q_V\vert_{V_{(p_i)}})\cong \left(V^{h}_{(p_i)}\otimes_E E^{p_i}, q^{h}_{(p_i)}\otimes q_{p_i}\right).
\end{align}
When $p_{i}$ is even, the identification in \eqref{Vpi} can be deduced from a straightforward computation; and when $p_i$ is odd, it follows from \eqref{eq:V-pi-odd}.
Hence we obtain that 
$\left(\underline{p},\{(V^h_{(p_i)},q^h_{(p_i)})\}\right)$ 
is an admissible sesquilinear Young tableau for $(V,q_V)$ corresponding to $\CO$, where $(V^h_{(p_i)},q^h_{(p_i)})$ as given in \eqref{Vpi} is assigned to $p_i^{m_i}$ for $i=1,2,\ldots,r$.

We denote by $\CY(V,q)$ the set of equivalence classes of admissible sesquilinear Young tableaux for an $\epsilon$-Hermitian vector space $(V,q)$, where $q=q_V$ for short.

\begin{prop}[\cite{GZ14}]\label{prop:GZ}
The above construction yields a one-to-one correspondence between the set $\CN_F(\Fg_n)_\circ$ of the $F$-rational nilpotent orbits in $\CN_F(\Fg_n)$ and
the set $\CY(V,q)$ of the equivalence classes of admissible sesquilinear Young tableaux
for $(V, q)$, except that when $G_n$ is an even special orthogonal group, an admissible sesquilinear Young tableau with only even parts may correspond to two $F$-rational nilpotent orbits.  
\end{prop}

It is clear that for any $a\in E^\times$, if we change $(V,q)$ to $(V,a\cdot q)$ by scaling, we obtain a natural bijection between $\CY(V, q)$ and $\CY(V, a\cdot q)$.


\section{Local Langlands Correspondence}\label{sec-LLC}


\subsection{Enhanced {$L$}-parameters and local Vogan packets}\label{ssec-LVP}

According to \cite{Ar13}, generic local $L$-parameters are the localization of the generic global Arthur parameters without the assumption of the generalized Ramanujan conjecture. They are explicitly given as follows.

Let $\CW_E$ be the local Weil group of $E$.
Then $\CW_\BC=\BC^\times$, $\CW_\BR=\BC^\times\cup j\BC^\times $,
where $j^2=-1$ and $j z j^{-1}=\bar{z}$ for $z\in \BC^\times$; and
when $E$ is $p$-adic, $\CW_E=\CI_E\rtimes \apair{\rm {Frob}}$ is the semi-direct product of the inertia group $\CI_E$ of $E$ and a geometric Frobenius element  ${\rm {Frob}}$.
The local Langlands group of $E$, also called the Weil-Deligne group, is defined by
\begin{align}\label{LLG}
\CL_E=\begin{cases}
	\CW_E &\text{if $E$ is archimedean},\\
	\CW_E\times\SL_2(\BC) &\text{if $E$ is $p$-adic}.	
\end{cases}	
\end{align}
We take the Galois version of the $L$-group ${^LG}:= G^\vee \rtimes {\rm Gal}(\overline F/F)$ of $G$, where $G^\vee$ is the complex dual group of $G$.
Note that if $G=\Mp_{2n}$, then $G^\vee=\Sp_{2n}(\BC)$; and if $G$ is $F$-split, then the action of ${\rm Gal}(\overline F/F)$ on $G^\vee$ is trivial.

Following \cite{Ar13}, when $G_n=G_n^*$ is  $F$-quasisplit,
an $L$-parameter 
\[
\varphi\ \colon\ \CL_F\longrightarrow{^LG_n}
\]
of $G^*_n$ is attached to a datum $(L^*,\varphi^{L^*},\udl{\beta})$ with the following properties:
\begin{enumerate}
\item $L^*$ is a Levi subgroup of $G_n$ of the form
$$
L^*=\GL_{n_1}(E)\times\cdots\times\GL_{n_t}(E)\times G_{n_0};
$$
\item $\varphi^{L^*}$ is a local $L$-parameter of $L^*$ given by
$$
\varphi^{L^*}:=\varphi_1\oplus\cdots\oplus\varphi_t\oplus\varphi_0\ \colon\ \CL_{F}\rightarrow {^L\!L^*},
$$
where $\varphi_j\ \colon\ \CL_E\longrightarrow\GL_{n_j}(\BC)$ is a local tempered $L$-parameter of $\GL_{n_j}(E)$ for $j=1,2,\ldots,t$, and $\varphi_0$ is a local tempered $L$-parameter
of $G_{n_0}^*$ (an $L$-parameter is called tempered if the image of the Weil group is bounded);
\item $\udl{\beta}:=(\beta_1,\cdots,\beta_t)\in \BR^t$, such that $\beta_1>\beta_2>\cdots>\beta_t>0$; and
\item the $L$-parameter $\varphi$ can be expressed as
$$
\varphi=(\varphi_1\otimes|\cdot|^{\beta_1}\oplus{}^c  \varphi_1^\vee \otimes|\cdot|^{-\beta_1})\oplus\cdots\oplus
(\varphi_t\otimes|\cdot|^{\beta_t}\oplus {}^c  \varphi_t^\vee \otimes|\cdot|^{-\beta_t})\oplus\varphi_0,
$$
where $|\cdot|$ is the absolute value of $E$ and ${}^c\varphi_j^\vee$ denotes the conjugate  dual of $\varphi_j$ for $j=1,2,\ldots,t$.
\end{enumerate}
Recall that the element $c\in \Gal(E/F)$ is introduced in Section \ref{ssec-CGIF}. Here and thereafter, by conjugate of a Weil group representation, we mean conjugation by an element of
$\CW_F$ which is mapped to $c$. Thus in the case that $G_n$ is unitary, an $L$-parameter $\varphi$ is equivalent to a conjugate self-dual $\Fn$-dimensional representation of $\CL_E$ of signature $(-1)^{\Fn-1}$ (cf. \cite{GGP12}) that is continuous and semisimple.

We denote by $\Pi_F(G_n)$ the set of equivalence classes of irreducible admissible  representations of $G_n(F)$. Here if $F$ is archimedean, we take the representations of 
$G_n(F)$ to be of Casselman-Wallach type (\cite{Cas89, Wal92}). As explained in \cite{K23}, the local Langlands conjecture for the pair $(G_n,F)$ 
predicts that there exists a unique correct partition of the set $\Pi_F(G_n)$ into local $L$-packets $\Pi_\varphi(G_n)$, which are expected to be a finite set, as $\varphi$ runs over the set of local $L$-parameters of $G_n$. 

In this paper, we are going to use the Vogan version of the 
local Langlands conjecture (\cite{V93}), which needs to define the local Vogan packet associated to $\varphi$, as defined by 
\begin{equation}\label{lvp}
{\Pi}_\varphi[G_n^*]:=\bigcup_{G_n}{\Pi}_\varphi(G_n)
\end{equation}
where $G_n$ runs over all pure inner $F$-forms of the given $F$-quasisplit $G_n^*$ and ${\Pi}_\varphi(G_n)$ is the local $L$-packet of $G_n$. The $L$-packet ${\Pi}_\varphi(G_n)$ is defined to be {\it empty} if the parameter $\varphi$ is not $G_n$-relevant (see \cite[Section 3.4]{B79} and \cite[Section 9.2]{Ar13}).
A local $L$-parameter $\varphi$ of $G_n^*$ is called {\it generic} if the associated local 
Vogan packet ${\Pi}_{\varphi}[G_n^*]$ contains a generic member when $G^*_n\ne \Mp_{2n}$, i.e., a member with a non-zero Whittaker model with respect to a certain Whittaker datum for $G_n^*$. 
When $G^*_n=\Mp_{2n}$, $\varphi$ is called {\it generic} if it is a generic $L$-parameter of $\SO_{2n+1}$ (see \cite{At18} for instance).
The set of all generic local $L$-parameters of $G_n^*$, up to equivalence, is denoted by ${\Phi}_\gen(G_n^*)$.

The Vogan version of the local Langlands conjecture for generic $L$-parameters $\varphi\in {\Phi}_\gen(G_n^*)$
asserts a one-to-one parameterization of the local Vogan packet ${\Pi}_\varphi[G_n^*]$ by the character group $\wh{\CS_\varphi}$ (the Pontryagin dual) of the component group 
$\CS_\varphi$, which is defined to be 
\[
\CS_\varphi:=S_\varphi/S_\varphi^\circ
\]
where $S_\varphi$ to be the centralizer of the image of $\varphi$ in $G^\vee$, and $S^\circ_\varphi$ to be its identity connected component group. This means that given $\varphi\in {\Phi}_\gen(G_n^*)$, for any $\pi\in {\Pi}_\varphi[G_n^*]$, there exists a unique character $\chi\in\wh{\CS_\varphi}$, such that $\pi=\pi_\Fw(\varphi,\chi)$, where $\Fw$ 
indicates 
that the uniqueness depends on the choice of the local Whittaker datum of $G_n^*$.
We refer to \cite[Sections 9--11]{GGP12} for a more detailed discussion about the Vogan version of the local Langlands conjecture for the classical groups under consideration.
In the case that $G^*_n=\SO_{2n}$, we use the weak local Langlands correspondence in \cite{AG17a} 
(also in \cite[Section 2B]{JZ18}) and replace $\Pi_\varphi(G_n)$ by the equivalence classes $\Pi_\varphi(G_n)/\sim_{\rm c}$ under the outer action of $\RO_{2n}$, i.e. conjugation by an element ${\rm c}\in \RO_{2n}\setminus \SO_{2n}$ with $\det(\rm{c})=-1$. 

This is the version of the local Langlands conjecture needed for this paper. It is known 
from the work of R. Langlands (\cite{L89}) that when $F$ is archimedean, the local Langlands 
conjecture is true, while the explicit parameterization follows from a series of work by D. Shelstad as explained in \cite[Section 2]{K23} (and also \cite[Theorem 6.3]{V93}). When $F$ is non-archimedean, as explained in \cite{K23} (and also \cite{Hm22}), after the fundamental work of J. Arthur (\cite{Ar13}), followed by \cite{Mok15}, the local Langlands conjecture for quasisplit classical groups is obtained, based on the local Langlands conjecture for $\GL_n(F)$ 
in \cite{Hn00, HT01, Sch13}. 
There are more recent works to extend those known cases to the pure inner forms of quasisplit 
classical groups or the metaplectic double cover of the symplectic groups 
from \cite{M11, GS12, MW12, KMSW14, AG17a, MR18, CZ21a, CZ21b}. As explained in \cite[Section 2]{K23}, the complete theory for the 
Vogan version of the local Langlands conjecture as stated above is still in progress for some cases of even special orthogonal groups, for instance. For accuracy, we may have to 
consider it as an assumption for those incomplete cases when we discuss 
the theory over non-archimedean local fields.

In order to discuss more explicit information from the Vogan version of the local Langlands conjecture for the classical groups considered in this paper, we are going to study more 
explicit structure of the generic $L$-parameters. 
For any generic $L$-parameter $\varphi\in{\Phi}_\gen(G_n^*)$, one may easily figure out the structure of the
abelian $2$-group $\CS_\varphi$. See \cite[Section 8]{GGP12} for instance. 
Write
\begin{equation}\label{decomp}
\varphi=\bigoplus_{i\in \RI}m_i\varphi_i,
\end{equation}
which is the decomposition of $\varphi$ into simple and generic ones.
The simple, generic local $L$-parameter $\varphi_i$ can be written as $\rho_i\boxtimes\mu_{b_i}$,  where $\rho_i$ is an 
irreducible representation of $\CW_E$ and $\mu_{b_i}$ is the irreducible representation of $\SL_2(\BC)$ of dimension $b_i$.

In the decomposition \eqref{decomp},
$\varphi_i$ is called {\it of good parity} if $\varphi_i$ is of the same type as $\varphi$. 
We denote by $\RI_{\gp}$ the subset of
$\RI$ consisting of the indices $i$ such that $\varphi_i$ is of good parity; and by $\RI_{\bp}$ the subset of
$\RI$ consisting of the indices $i$ such that $\varphi_i$ is (conjugate) self-dual, but not of good parity. We set $\RI_{\nsd}$ to be 
the set of  indices for the pairs $\{\varphi_i, {}^c\varphi_i^\vee\}$ which are not (conjugate) self-dual.
Hence we may write $\varphi\in{\Phi}_\gen(G_n^*)$ in the following more explicit way:
\begin{equation}\label{edec}
\varphi=\bigoplus_{i\in \RI_\gp}m_i\varphi_i \oplus \bigoplus_{j\in \RI_\bp}2m'_j\varphi_j \oplus \bigoplus_{k\in \RI_\nsd}m_k(\varphi_k\oplus{}^c \varphi_k^\vee),
\end{equation}
where $2m'_j=m_j$ in \eqref{decomp} for $j\in \RI_{\bp}$.
According to this explicit decomposition, it is easy to know that
\begin{equation}\label{2grp}
\CS_\varphi\cong\BZ_2^{\#\RI_\gp}\quad \text{or}\quad  \BZ_2^{\#\RI_\gp-1}.
\end{equation}
The latter case occurs if $G_n$ is even orthogonal and some orthogonal summand $\varphi_i$ for $i\in \RI_\gp$ has odd dimension, or if $G_n$ is symplectic.

In all cases, for any $\varphi\in{\Phi}_\gen(G_n^*)$ we write elements of $\CS_\varphi$ in the following form
\begin{equation}\label{eq:e-i}
(e_i)_{i\in \RI_{\gp}}\in \BZ_{2}^{\#\RI_{\gp}}, \text{ (or simply denoted by $(e_{i})$)}, 
{\rm with}\ e_i\in\{0,1\},
\end{equation}
where each $e_i\in \{0,1\}$ corresponds to $\varphi_i$-component in the decomposition \eqref{edec} for $i\in \RI_\gp$.
In the case that $\wh{G}=\SO_N(\BC)$, denote by $A_\varphi$ the component group  ${\rm Cent}_{\RO_N(\BC)}(\varphi)/{\rm Cent}_{\RO_N(\BC)}(\varphi)^\circ$.
Then $\CS_\varphi$ consists of elements in $A_\varphi$ with determinant $1$ and is a subgroup of index $1$ or $2$.
Also write elements of $A_\varphi$ in the form
$(e_i)$ where $e_i\in\{0,1\}$ corresponds to the $\varphi_i$-component in the decomposition \eqref{edec} for $i\in \RI_\gp$.
When $G_n$ is even orthogonal and some $\varphi_i$ for $i\in \RI_\gp$ has odd dimension or $G_n$ is symplectic, $(e_i)_{i\in \RI_{\gp}}$ is in $\CS_\varphi$ if and only if $\sum_{i\in\RI_\gp}e_i\dim\varphi_i$ is even.

\subsection{Rationality of local Langlands correspondence} \label{ssec-RLLC}

\subsubsection{$\CZ$-action} Define
\[
\CZ=F^\times/ \BN E^\times \cong
\begin{cases} 
F^\times/ F^{\times2}, & \textrm{if }E=F,\\
\textrm{Gal}(E/F), & \textrm{if }E=F(\delta).
\end{cases}
\]
Let $\varphi\in \Phi_\gen(G_n^*)$. For $a\in \mathcal{Z}$, define $\eta_a\in \widehat{\CS_\varphi}$ by\footnote{For $\Mp_{2n}$ the character $\eta_a$ is denoted by $\eta[a]$ in \cite{GGP12}; we change the notation for uniformity.}
\begin{equation}\label{eta}
\eta_a((e_i)_{i\in \RI_{\gp}})=
\begin{cases} 
\prod_{i\in \RI_{\gp}}(\det \varphi_i)(a)^{e_i}, & \textrm{if }E=F\textrm{ and }G_n^*\neq \Mp_{2n},\\
\prod_{i\in \RI_\gp} \left(\varepsilon(\varphi_i)\varepsilon(\varphi_i(a))(a,-1)_F^{\dim \varphi_i/2}\right)^{e_i}, & \textrm{if }E=F\textrm{ and }G_n^*=\Mp_{2n},\\
\prod_{i\in \RI_{\gp}}\omega_{E/F}(a)^{e_i\dim \varphi_i}, & \textrm{if }E=F(\delta),
\end{cases}
\end{equation}
where $(\cdot, \cdot)_F$ is the Hilbert symbol defined over $F$, and $\omega_{E/F}$ is the quadratic character of $\CZ$ given by the local class field theory. Note that $\eta_a$ is trivial if $G_n$ is odd special orthogonal.
Then we have a $\CZ$-action on $\widehat{\CS_\varphi}$ by
\[
a\ \colon\ \widehat{\CS_\varphi} \to \widehat{\CS_\varphi},\quad \chi\mapsto \chi\cdot \eta_a
\]
for every $a\in\CZ$. 

Throughout the paper, we fix a nontrivial additive character $\psi^F$ of $F$. If $G_n^*$ is a skew-Hermitian unitary group, we also fix a nontrivial  additive character $\psi^E$ of $E$ as
\begin{equation} \label{psi-rel}
\psi^E\ \colon\ E\to \BC^\times,\quad \psi^E(x)=\psi^F\left(\frac{1}{2}{\rm tr}_{E/F}(\delta x)\right),
\end{equation}
which is trivial on $F$ and is conjugate self-dual.

\subsubsection{Linear classical groups}
For a fixed Whittaker datum $\Fw$ of $G_n^*\neq \Mp_{2n}$, there exists a natural bijection (\cite[Section 10]{GGP12})
\begin{align}\label{rllc-w}
\iota_\Fw\ \colon\ \Pi_\varphi[G_n^*]\to \wh{\CS_\varphi},\quad \pi_\Fw(\varphi,\chi)\mapsto \chi\in \wh{\CS_\varphi}
\end{align}
satisfying the endoscopic and twisted endoscopic character identities (\cite{Ar13, K16}). Under a fixed bijection $\iota_\Fw$, which depends on the choice of the Whittaker datum $\Fw$, if $\pi\in {\Pi}_\varphi[G_n^*]$ satisfies the condition
\[
\pi=\pi_\Fw(\varphi,\chi),
\]
then we call the datum $(\varphi,\chi)$ the {\it enhanced $L$-parameter} of $\pi$ under the bijection $\iota_\Fw$. 

The rationality of the local Langlands correspondences for linear classical groups can be explicitly determined as follows.
\begin{itemize}
\item
If $\Fn$ is even, with $\psi^F$ and $\psi^E$ fixed as above, the Whittaker data are parameterized by $\CZ$ (\cite[Section 10]{GGP12}), and we denote by $\iota_a$, $a\in\CZ$ the corresponding bijection
\[
\iota_a\ \colon\ \Pi_\varphi[G_n^*]\to\wh{\CS_\varphi}, \quad \pi_a(\varphi,\chi)\mapsto \chi\in\wh{\CS_\varphi}.
\]
In this case it holds that
\begin{equation}\label{twist-even}
\pi_1(\varphi,\chi)=\pi_a(\varphi,\chi\cdot \eta_a).
\end{equation}
\item
If $\mathfrak{n}$ is odd, then the Whittaker datum and the  bijection \eqref{rllc-w} are unique. In this case we write 
\[
\iota_1\ \colon\ \Pi_\varphi[G_n^*]\to\wh{\CS_\varphi}, \quad \pi_1(\varphi,\chi)\mapsto \chi\in\wh{\CS_\varphi}
\]
for this unique bijection. For odd special orthogonal groups, for convenience we  put
\begin{equation} \label{twist-odd}
\pi_a(\varphi,\chi) := \pi_1(\varphi,\chi)
\end{equation}
for any $a\in\CZ$.
\end{itemize}
There are some subtle issues for unitary groups. In this paper, when $G_n$ is odd unitary we choose the $F$-quasisplit inner form $G_n^*$ to be of discriminant 1. 
Thus if $\pi_a(\varphi,\chi)\in\Pi_\varphi(G_n)$ for a Hermitian unitary group $G_n=\RU(V)$ with $a=1$ when $\Fn$ is odd, then 
\begin{equation} \label{uni-disc}
    \chi(-1_\varphi) = \omega_{E/F}(\disc(V)),
   \end{equation}
    where  $-1_\varphi$ denotes the image of $-\RI_\Fn\in G_n^\vee=\GL_\Fn(\BC)$ in $\CS_\varphi$. Note that in this case if $a\in \CZ$ is the unique nontrivial norm class, then \eqref{eta} reads that
\[
\eta_a((e_i)_{i\in \RI_{\gp}})=\prod_{i\in \RI_{\gp}}(-1)^{e_i\dim \varphi_i}.
\]
In particular, $\eta_a(-1_\varphi)=(-1)^\Fn$.
Moreover, we make the following remarks. 
\begin{itemize}
\item
 For an $\Fn$-dimensional Hermitian space $V$ with $\Fn$ even, the space $W=\delta\cdot V$  is skew-Hermitian so that
$\RU(V)\cong \RU(W)$.  Recall that the additive characters $\psi^F$ and $\psi^E$ are fixed and related via \eqref{psi-rel}.
Then for any $a\in\CZ$, the resulting bijections $\iota_a$ are the same for $\Pi_\varphi[\RU(V)]$ and $\Pi_\varphi[\RU(W)]$.

\item
The character $\eta_a$ also plays a role for odd unitary groups.
In the case of $p$-adic odd unitary groups, the two pure inner forms $\RU(V^*)$ and $\RU(V)$ are isomorphic via rescaling the $\epsilon$-Hermitian form by the unique nontrivial class $a\in \CZ$, and $\pi_1(\varphi,\chi)\in \Pi_\varphi(\RU(V^*))$  is identified with
$\pi_1(\varphi,\chi\cdot\eta_a)\in \Pi_\varphi(\RU(V))$. 
A similar remark applies to the case that $F=\BR$, via the isomorphism $\RU(p,q)\cong \RU(q,p)$.  We also stress that the local Langlands correspondence for Hermitian and skew-Hermitian odd unitary groups is the same. For convenience, we put
\begin{equation}\label{twist-odd-u}
\pi_a(\varphi,\chi) : = \pi_1(\varphi,\chi\cdot\eta_a)
\end{equation}
for any $a\in \CZ$.
\end{itemize}

\subsubsection{Metaplectic groups}
Recall that the $L$-group of $G_n^*=\Mp_{2n}$ is $\Sp_{2n}(\BC)\times {\rm Gal}(\overline F/F)$. The local Langlands correspondence 
for $\Mp_{2n}$ is realized through the theta correspondence and the local Langlands correspondence for $\SO_{2n+1}$. 
More precisely, following \cite[Section 11]{GGP12} and \cite{GS12}, a bijection
\[
\iota_1=\iota_{\psi^F}\ \colon\ \Pi_\varphi(\Mp_{2n})\longrightarrow \wh{\CS_\varphi}, \quad \pi_1(\varphi,\chi)\mapsto \chi\in \wh{\CS_\varphi}
\]
is achieved for the fixed nontrivial additive character $\psi^F$ of $F$ and the local theta correspondence gives a bijection between the local $L$-packet 
$\Pi_\varphi(\Mp_{2n})$
and  the local Vogan packet $\Pi_\varphi[\SO(V_{2n+1})]$, where  $V_{2n+1}$ is the $(2n+1)$-dimensional split orthogonal space with trivial discriminant.
Note that if we replace $\psi^F$ by $\psi^F_a(x):=\psi^F(ax)$ for $a\in \CZ$ in the local theta correspondence, then $\iota_a:=\iota_{\psi^F_a}$ satisfies that
\begin{equation} \label{twist-mp}
\pi_1(\varphi, \chi)=\pi_a(\varphi(a),\chi\cdot \eta_a),
\end{equation}
where $\varphi(a)$ denotes the tensor product of $\varphi$ with the Hilbert symbol $(\cdot, a)_F$, and we regard $\eta_a$ as a character of $\CS_{\varphi(a)}$ via the canonical isomorphism $\CS_\varphi\cong \CS_{\varphi(a)}$.


\

In summary of \eqref{twist-even}, \eqref{twist-odd}, \eqref{twist-odd-u} and \eqref{twist-mp}, we have that

\begin{prop} \label{prop:LLC}
For $\varphi\in \Phi_\gen(G_n^*)$, it holds that
\[
\pi_a(\varphi,\chi) = \begin{cases} \pi_1(\varphi, \chi\cdot\eta_a), & \textrm{if }G_n^*\neq \Mp_{2n}, \\
\pi_1(\varphi(a), \chi\cdot\eta_a), & \textrm{if }G_n^*=\Mp_{2n}
\end{cases}
\]
for any $a\in \CZ$ and $\chi \in \wh{\CS_\varphi}$. 
\end{prop}

In general, when the choice of $\iota_a$ is clear from the context, we occasionally suppress the Whittaker datum and simply write $\pi(\varphi,\chi)=\pi_a(\varphi,\chi)$.

\subsection{Contragredience}\label{ssec-C} 

For the philosophy of local descent and its connection with wavefront sets which will be discussed in the next section, it is necessary to take the contragredient for the representation of the smaller member in a relevant pair of classical groups in the local Gan-Gross-Prasad conjecture.

We refer to \cite{AV16, K13} for general discussions about contragredient representations. Let $G^*=G_n^*$ be a connected quasi-split classical group over $F$. For $a\in \CZ$, put
\[
a^\vee:=\begin{cases} a, & \textrm{if }\Fn\textrm{ is odd},\\
-a, & \textrm{if }\Fn\textrm{ is even}.\end{cases}
\]
Let $\Fc^\vee\in \Aut(G^\vee)$ be the Chevalley involution
of $G^\vee$, and let ${}^L\Fc= \Fc^\vee\rtimes \textrm{id}$ be the automorphism of ${}^LG$.  Let $\varphi\in \Phi_\gen(G^*)$ and $\chi\in \wh{\CS_\varphi}$. 
The automorphism $\Fc^\vee$ induces an isomorphism $\CS_\varphi\cong \CS_{{}^L\Fc\circ\varphi}$. By \cite[Theorem 5.9]{K13}, the contragredient representation of $\pi_a(\varphi,\chi)$ is given by
\[
\pi_a(\varphi,\chi)^\vee = \pi_{ a^\vee} ({}^L\Fc\circ\varphi, \chi \circ (\Fc^\vee)^{-1}).
\]

For an enhanced $L$-parameter $(\varphi,\chi)$ as above, we introduce
its {\it contragredient} $(\wh\varphi, \wh\chi)$ to be the enhanced $L$-parameter of $\pi_a(\varphi, \chi)^\vee$ under the same Whittaker datum $a$, that is, we require that
\[
\pi_a(\varphi, \chi)^\vee = \pi_a(\wh\varphi, \wh\chi).
\]
It is easy to see that this notion does not depend on  $a$.  The explicit formulas of $(\wh\varphi,\wh\chi)$ for various classical groups have been worked out (\cite{GI16, AG17b, At18}).  For convenience, we summarize them in the following proposition and refer to the details in \cite[Proposition B.4]{AG17b} and \cite[Section 3]{At18}. 

\begin{prop}\label{prop:dualdata}
For an enhanced $L$-parameter $(\varphi,\chi)$ as above, the contragredient $(\wh\varphi, \wh\chi)$ can be written, case by case, as follows:
\begin{itemize}
\item $G^*=\SO_\Fn^*$, $(\wh\varphi, \wh\chi) = (\varphi, \chi)$.
\item $G^*=\Sp_{2n}$, $(\wh\varphi, \wh\chi) = (\varphi, \chi \cdot \eta_{-1})$.

\item $G^*=\Mp_{2n}$,  $(\wh\varphi, \wh\chi) = (\varphi(-1), \chi\cdot \eta_{-1})$.

\item $G^* =\RU_\Fn^*$, $(\wh\varphi, \wh\chi) = (\varphi^\vee, \chi\cdot\eta)$ where $\eta$ is trivial if $\Fn$ is odd and $\eta = \eta_{-1}$ if $\Fn$ is even.
\end{itemize}
In all cases as listed above, the following equality 
\[
(\wh\varphi, \wh{\chi\cdot \eta_a}) = (\wh\varphi, \wh\chi\cdot \eta_a)
\]
holds for any $a\in \CZ$.
\end{prop}

\subsection{Twisted distinguished characters} \label{ssec-DC}

We explain some distinguished characters from the local Gan-Gross-Prasad conjecture by using local root numbers, and introduce certain twists by elements of $\CZ$ 
that will be used in the construction of $L$-parameters in Section~\ref{sec-AWFS}. For 
convenience, we adapt some convention from \cite{GGP12}. 

Assume that $\varphi$ and $\phi$ are two generic $L$-parameters
of opposite types, written in the form \eqref{edec} with index sets $\RI_\gp$ and $\RI_\gp'$, respectively. We write the elements of the corresponding component groups $\CS_\varphi$ and $\CS_\phi$ as $(e_i)_{i\in \RI_\gp}$ and $(e_{i'})_{i'\in \RI_\gp'}$ respectively. Recall the additive characters $\psi^F$ and $\psi^E$ of $F$ and $E$ respectively given by \eqref{psi-rel}.
We refer to \cite[Section 5]{GGP12} for the definition of local root numbers that will be used below.

\subsubsection{}
If  $E=F$, $\varphi$ and $\phi$ are self-dual of even dimensions, define a pair of characters $(\chi_{\varphi,\phi}, \chi_{\phi,\varphi})\in \wh{\CS_\varphi}\times\wh{\CS_\phi}$ by
\begin{equation} \label{char-O}
 \chi_{\varphi,\phi}((e_i)_{i\in \RI_\gp}):= \prod_{i\in \RI_\gp}\left(\varepsilon(\varphi_i\otimes\phi) (\det \varphi_i)(-1)^{\dim \phi/2}(\det \phi)(-1)^{\dim\varphi_i/2}\right)^{e_i},
\end{equation}
and a similar formula defines $\chi_{\phi,\varphi}$. In the above, when $\dim \varphi_i$ is odd, then $\phi$ is symplectic so that 
$\det\phi$ is trivial and $(\det \phi)(-1)^{\dim\varphi_i/2}$ is understood to be 1.
Note that the local root number
\[
\varepsilon(\varphi_i\otimes\phi):=\varepsilon(\varphi_i\otimes\phi, \psi^F)
\]
 does not depend on the choice of the additive character $\psi^F$. In this case, for $z\in \CZ$ define
 \begin{equation}
 (\chi_{\varphi, \phi}^z, \chi_{\phi,\varphi}^z) = (\chi_{\varphi, \phi}\cdot \eta_z, \chi_{\phi,\varphi}\cdot \eta_z),
 \end{equation}
 where $\eta_z$ is given by \eqref{eta}. Note that in \eqref{char-O}, we follow the convention in \cite{GGP12} to write
 \[
 (\det\varphi)(-1)^\alpha=(\det\varphi)((-1)^\alpha), 
 \]
 which will be used in the whole paper. 

\subsubsection{}
If $E=F$, $\varphi$ is odd orthogonal and $\phi$ is symplectic,  define $(\chi_{\varphi,\phi}, \chi_{\phi,\varphi})$ to be the restriction of
\[
(\chi_{\varphi_1,\phi},\chi_{\phi,\varphi_1})\in
\wh{\CS_{\varphi_1}}\times\wh{\CS_\phi}
\]
to $\CS_{\varphi}\times \CS_\phi$, where
$\varphi_1:=\varphi\oplus\BC$.
We also consider the twist $\varphi(z)$ with $z\in \CZ$, and make the canonical identification $\CS_{\varphi(z)}\cong \CS_\varphi$. Then we define
\begin{equation} \label{char-spmp}
(\chi^z_{\varphi, \phi}, \chi^z_{\phi,\varphi})=(\chi_{\varphi(z), \phi} \cdot \eta_z, \chi_{\phi, \varphi(z)})\in \widehat{\CS_\varphi}\times\widehat{\CS_\phi}.
\end{equation}

\subsubsection{}
If $E=F(\delta)$, $\varphi$ and $\phi$ are conjugate self-dual, define $(\chi_{\varphi,\phi}, \chi_{\phi,\varphi})\in \wh{\CS_\varphi}\times\wh{\CS_\phi}$ by\footnote{For unitary groups, \cite{GGP12, GI16} use $\psi^E_{-2}$
 to compute local root numbers for the Bessel case, and use $\psi^E_2$ for the Fourier-Jacobi case. Our modification will be more convenient for the uniform computation of local descent, and it causes a change in the formulation of the local Gan-Gross-Prasad conjecture in Section \ref{sec:GGP-U-B}.}
\begin{equation} \label{char-U}
 \chi_{\varphi,\phi}((e_i)_{i\in \RI_\gp})= \prod_{i\in \RI_\gp}\varepsilon(\varphi_i\otimes\phi, \psi^E_2) ^{e_i},
\end{equation}
and a similar formula defines $\chi_{\phi,\varphi}$, where $\psi^E$ is given by \eqref{psi-rel} and 
\[
\psi^E_2(x) := \psi^E(2x) = \psi^F(\tr_{E/F}(\delta x)),\quad x\in E.
\]
Assume that $\Fl= \dim \varphi - \dim\phi\geq 0$. In this case, for $z\in \CZ$ define
\begin{equation}\label{chi-a}
(\chi^z_{\varphi,\phi}, \chi^z_{\phi,\varphi})= (\chi_{\varphi, \phi}\cdot \eta_z, \chi_{\phi,\varphi}\cdot \eta_{(-1)^\Fl z}).
\end{equation}

By convention, in the above cases we allow $\varphi$ or $\phi$ to be zero. In such a case we formally take $\chi_{\varphi, \phi}$ and $\chi_{\phi,\varphi}$ to be the trivial characters of $\CS_\varphi$ and $\CS_\phi$ respectively.

\section{Local Gan-Gross-Prasad Conjecture}\label{ssec-LGGP}


It is important to point out that the local Gan-Gross-Prasad conjecture (\cite{GGP12}) 
is now a well-established theorem for classical groups and generic local $L$-parameters over all local fields of characteristic zero. More precisely, it is proved for all $p$-adic cases by J.-L. Waldspurger (\cite{W10, W12a, W12b}),
C. M\oe glin and J.-L. Waldspurger (\cite{MW12}), R. Beuzart-Plessis (\cite{BP16}),  W. T. Gan and A. Ichino (\cite{GI16}), and H. Atobe (\cite{At18}). It is proved for real unitary groups by R. Beuzart-Plessis \cite{BP20}, and H. Xue (\cite{X23, X}). Finally, it is proved for real special orthogonal groups  by  Z. Luo (\cite{Lu20, Lu21}), C. Chen (\cite{Ch23, Ch24, Ch25}), C. Chen and Z. Luo (\cite{CL22}), and finally for the Fourier-Jacobi cases over $\BR$ by C. Chen, R. Chen, and J. Zou(\cite{CCZ25}) and C. Chen (\cite{Ch23, Ch24}). 

In order to state the local Gan-Gross-Prasad conjecture, we have to introduce 
certain twisted Jacquet modules of Bessel type and Fourier-Jacobi type. We will do this in 
a more general setting because we need them for the definition of algebraic wavefront sets in 
Section \ref{sec-AWFS}.

\subsection{Certain twisted Jacquet modules}\label{ssec-TJM}
Let $G$ be a connected classical group over $F$, and $\Fg$ be its Lie algebra. Then the set $\CN_F(\Fg)$ of all nilpotent elements of $\Fg(F)$ is an algebraic conic $G$-variety defined over $F$. As in Section~\ref{ssec-RNO}, we denote by $\CN_F(\Fg)_\circ$ the set of all $F$-rational nilpotent orbits in $\CN_F(\Fg)$ under the adjoint action of $G(F)$.

Following \cite{CM93}, for $X\in \CN_F(\Fg)$, by the Jacobson-Morozov Theorem, there exists an ${\mathfrak{sl}}_2$-triple $\{X,\hbar,Y\}$. By a theorem of B. Kostant, such ${\mathfrak{sl}}_2$-triple is unique up to conjugation.
Under the adjoint action of $\hbar$ on $\Fg$,
one has the decomposition $\Fg=\bigoplus\limits_{i\in\BZ}\ \Fg^{\hbar}_i$, where
\[
\Fg^{\hbar}_i:=\{x\in\Fg \mid {\rm {ad}}(\hbar)(x)=ix\}.
\]
Denote by
\[
\Fu:=\bigoplus_{i\leq -1}\ \Fg^{\hbar}_i,\quad \Fp:=\bigoplus_{i\leq 0}\ \Fg^{\hbar}_i,\quad \text{and}\quad 	
\Fm:=\Fg^{\hbar}_0.
\]
Denote by $U_X=U:=\exp(\Fu)$,
$P:=\{g\in G\mid \Ad(g)(\Fp)\subset\Fp\}$,
and   $M:=\{g\in G\mid \Ad(\hbar)=\hbar)\}$.
Then $P$ is a parabolic subgroup of $G$ defined over $F$ and has the  decomposition $P=M\ltimes U$, whose conjugacy class only depends on the $F$-stable orbit of $X$.
That is,
\begin{equation*}
M=\Isom(V_0,q\vert_{V_0})^0\times \prod_{i>0}\GL(V_i).	
\end{equation*}
where $V_i$ is defined in \eqref{eq:V-i}.
By convention, we may take $\Fg_0=\Fg$ when $X=0$.

Let us recall the smooth oscillator representation of $U_X$ associated to $X$ from \cite[Section 3.3]{GZ14}, \cite{JLS16} or \cite{GGS17}.
Recall the fixed  nontrivial  additive character $\psi^F$ of $F$.
Fix an $\Ad(G)$-invariant non-degenerate bilinear form $\kappa$ on $\Fg$.
Let
\begin{equation}  \label{uX2}
\Fu_{X,2}:=\bigoplus_{i\leq -2}\ \Fg^{\hbar}_i
\end{equation}
and
$
U_{X,2}:=\exp(\Fu_{X,2}).
$
Define the character $\psi_{X}$ on $U_{X,2}$ by
\[
\psi_X(\exp(A)):=\psi^F(\kappa(X,A))	\qquad \text{for $A\in \Fu_{X,2}$.}
\]
Following \cite[Section 3.3]{GZ14}, define the symplectic form on $\Fg_{-1}$ by $\kappa_{-1}(A,B):=\kappa({\rm {ad}}(X)A,B)$ for $A,B\in\Fg_{-1}$.
We obtain the Heisenberg group $\CH_{X}=\Fg_{-1}\times F$ associated to the symplectic space $(\Fg_{-1},\kappa_{-1})$,
where $\{0\}\times F$ is the center of $\CH_{X}$.
In particular, $(A,0)(B,0)=(A+B,\kappa_{-1}(A,B))$ for $A,B\in\Fg_{-1}$.
This yields a surjective group homomorphism $\alpha_X\colon U_X \to \CH_X$ given by
\[
\alpha_X(\exp(A)\exp(Z))=(A,\kappa(X,Z))\qquad  \text{ for  } A\in\Fg_{-1}, Z\in \Fu_{X,2}.	
\]
The character $\psi_X$ factors through $U_{X,2}$ and defines a character of the center of $\CH_X$.
By the Stone-von Neumann theorem, there exists a unique irreducible unitarizable smooth representation of $\CH_{X}$ with central character $\psi_X$, which is denoted by $(\omega_{\psi_X},V_{\psi_X})$.
It can be lifted as a representation of $U_X$ by
$u\cdot v:=\omega_{\psi_X}(\alpha_X(u))v$ for $u\in U_X$ and $v\in V_{\psi_X}$, which is denoted by $(\omega_{\psi_X},V_{\psi_X})$ and called the smooth oscillator representation.
We remark that if $\Fg_{-1}=\{0\}$ then $U_{X,2}=U_X$ and $\omega_{\psi_X}$ is the character $\psi_X$.

Let $(\pi,V_\pi)$ be a smooth representation of $G(F)$, or possibly the metaplectic cover of $G(F)$ if $G=\Sp_{2n}$.
Following \cite{MW87, GZ14, GGS17}, define {\it the space of generalized Whittaker models of $\pi$} associated to $X$ to be
\begin{equation} \label{whit-model}
{\rm {Wh}}_X(\pi)=\Hom_{U_X}(V_\pi,V_{\psi_X}),
\end{equation}
and define the generalized Whittaker quotient to be the $U_X$-coinvariants (the twisted Jacquet module)
\begin{equation} \label{whit-quo}
\CJ_X(\pi)=V_\pi\,\wh{\otimes}\, V^\vee_{\psi_X}\ /\ \overline{ \{ \pi\otimes\omega^\vee_{\psi_X}(u)v-v\ 
\mid\  u\in U_X, v\in V_\pi\,\wh{\otimes}\, V^\vee_{\psi_X}\}}.
\end{equation}
 Here $V_\pi\,\wh{\otimes}\, V^\vee_{\psi_X}$ is the complete projective tensor product for $F=\BR$, and the closure is necessary only when $F=\BR$. Then by the Frobenius reciprocity,
 \[
 {\rm {Wh}}_X(\pi)\neq 0 \ \Longleftrightarrow \ \CJ_X(\pi)\neq 0.
 \]
Note that up to conjugation the generalized Whittaker models of $\pi$ only depend on the $F$-rational nilpotent orbit of $X$ in $\CN_F(\Fg)_\circ$.
Denote by $M_{X}:=\{g\in M\ \mid\ \Ad(g)(X)=X\}$ and $M_{X}^\circ$  the identity component $M_X$.
More precisely,
\[
M_X=\prod^{r}_{i=1}\Isom(V^{h}_{(p_i)},q^h_{(p_i)}),
\]
where $V^{h}_{(p_i)}$ and $q^h_{(p_i)}$ are defined in \eqref{eq:X-V-i} and \eqref{eq:X-q-i}.
Then $M_{X}$ is also the centralizer of the ${\mathfrak{sl}}_2$-triple $\{X,\hbar,Y\}$.
If $\Fg_{-1}\ne \{0\}$, then there exists a central cover $\wt{M}_X$ of $M_X$ and
a representation of
\[
\wt{P}_X:=\wt{M}^\circ_X\ltimes U_{X},
\]
extended from the oscillator-Heisenberg representation $(\omega_{\psi_X},V_{\psi_X})$ of $U_X$.
For convenience, we still denote by $(\omega_{\psi_X},V_{\psi_X})$ the extension of $\wt{P}_X$.
Note that ${\rm {Wh}}_X(\pi)$ is an $M_X$-module if $\Fg_{-1}=\{0\}$ or an $\wt{M}_X$-module if $\Fg_{-1}\ne \{0\}$.

\subsection{Rational nilpotent orbits for $[p_1,1^{\Fn-p_1}]$}\label{ssec-RNO-D}

As in \cite{JZ18}, in order to set up the rationality in the local Gan-Gross-Prasad conjecture and for the purpose of the local descents, we would like to make more precise the $F$-rational structure of the nilpotent orbits associated to partitions of type $[p_1,1^{\Fn-p_1}]$ with $p_1>1$, which 
produce the twisted Jacquet modules of Bessel type or Fourier-Jacobi type (Section \ref{ssec:TJMp1}).

Denote by $\CO^{\st}_{p_1}$ the unique $F$-stable nilpotent orbit associated with the partition $[p_1,1^{\Fn-p_1}]$.
Following Section \ref{ssec-RNO},
to parameterize the $F$-rational nilpotent orbits in $\CO^{\st}_{p_1}$,
we need to assign  an $\eps_1$-Hermitian form $(V^h_{(p_1)},q^h_{(p_1)})$ and an $\eps$-Hermitian form $(V_{(1)},q_{(1)})$
where $\eps_1=(-1)^{p_1-1}\eps$ and $\dim V^h_{(p_1)}=1$.
By definition,  $V_{(p_1)}$ has the polar decomposition $V_{(p_1)}^+ \oplus \Fe\oplus V_{(p_1)}^-$,
where $\Fe$ is an anisotropic line if $p_1$ is odd, and is zero otherwise, and
$V_{(p_1)}^+$ and $V_{(p_1)}^-$ are dual to each other.
In particular when $p_1$ is odd,  $q_{(p_1)}$ can be determined by $q_{(p_1)}\vert_{\Fe}=q^h_{(p_1)}$.
The assigned data yield the admissible Young tableaux for $(V,q)$ with the partition $[p_1,1^{\Fn-p_1}]$ if and only if
\[
(V,q)\cong (V_{(p_1)}^+\oplus \Fe \oplus V_{(p_1)}^-,q_{(p_1)})\oplus (V_{(1)},q_{(1)}).
\]
Furthermore, from $[p_1,1^{\Fn-p_1}]\in\CP(V)$, we obtain some constraints on $p_1$, in order for  $\CO^{\st}_{p_1}$ to be defined over $F$ (i.e., $\CO^{\st}_{p_1}\cap \CN_F(\Fg_n)\ne \varnothing$).
More precisely, if $V$ is orthogonal and $\Fn\ne 2\Fr$, then $p_1$ is odd, and
$\CO^{\st}_{p_1}$ is defined over $F$
if and only if $p_1\leq 2\Fr+1$;
if $V$ is orthogonal and $\Fn=2\Fr$, then $p_1$ is odd, and
$\CO^{\st}_{p_1}$ is defined over $F$ if and only if $p_1\leq 2\Fr-1$;
if $V$ is symplectic, then $p_1$ is even, and $\CO^{\st}_{p_1}$ is defined over $F$ for all $p_1\leq \Fn$; and finally
if $G_n$ is unitary, then there is no constrain  on $p_1$, and $\CO^{\st}_{p_1}$ is defined over $F$ if and only if $p_1\leq 2\Fr+1$.

For convenience, we take an explicit ${\mathfrak {sl}}_2$-triple $\{X,\hbar, Y\}$ for each $F$-rational orbit in $\CO^{\st}_{p_1}$ with $X\in \CO^{\st}_{p_1}$.
It follows that the $X$'s are the chosen representatives of $F$-rational orbits in $\CO^{\st}_{p_1}$.
With the basis given in \eqref{eq:basis}, we can uniformly choose $\hbar\in \Fg_n$ for all $F$-rational orbits, which is defined by, for $x\in \FB$,
\[
\hbar(x)= \begin{cases}
	\mp(p_1-2i+1)e_{\pm i}  &\text{ if $x=e_{\pm i}$ and } 1\leq i\leq  \lfloor\frac{p_1}{2}\rfloor,\\
	0 &\text{ otherwise.}
\end{cases}
\]
With such a chosen $\hbar$, we obtain the eigenspaces in \eqref{eq:V-i}.
For $v,w\in V$, define the element $A_{v,w}$ in $\Fg\subset\End(V)$ by
\[
A_{v,w}(x)= \eps q(x,v)w- q(x,w)v.	
\]
Throughout this paper, we choose and fix an vector $e$ in $V$ based on the parity of $p_1$:
\begin{itemize}
	\item if $p_1$ is odd, we take $e$ to be an anisotropic vector in $V_0$;
	\item if $p_1$ is even, we take $e=\frac{\eps}{2}e_{-m}$,
\end{itemize}
where $V_{0}=\{v\in V\ \mid\ \hbar(v)=0\}$ and  $m=\lfloor\frac{p_1}{2}\rfloor$.
Note that if $p_1$ is odd, then $V$ is not symplectic and there always exist anisotropic vectors if $V_{0}$ is not zero.
With the so chosen vector $e$, for $\varsigma\in E^\times$, define
\begin{equation} \label{xe,var}
X_{e,\varsigma}=A_{ e, \bar{\varsigma} e_{-m}}+\sum_{i=1}^{m-1}A_{ e_{i+1}, \bar{\varsigma} e_{-i}}.	
\end{equation}
Alternatively, $X_{e,\varsigma}$ is determined by
\begin{itemize}
\item
$ e_{i}\mapsto -\varsigma e_{i+1}$,
$e_{-(i+1)}\mapsto \bar{\varsigma} e_{-i}$  for $1\leq i\leq m-1$, and
$e_{-1}\mapsto 0$;	
\item
$ e_m\mapsto -\varsigma e$,
$e\mapsto \eps  \bar{\varsigma} \apair{e,e}e_{-m}$
and $X\vert_{V_{0}\cap e^\perp}=0$ when $p_1$ is odd;
\item
$e_m\mapsto \frac{\bar{\varsigma}-\eps \varsigma }{2}e_{-m}$ and $X\vert_{V_{0}}=0$ when $p_1$ is even.
\end{itemize}
Thus, $X_{e,\varsigma}$ is in $\CO^{\st}_{p_1}$ when $\eps \varsigma\ne \bar{\varsigma}$ for  $p_1$ even.
Note that if $p_1$ is even, $A_{e,\varsigma_1 e_{-m}}=A_{e,\varsigma_2 e_{-m}}$
if and only if $\bar{\varsigma}_1-\eps \varsigma_1=\bar{\varsigma}_2-\eps \varsigma_2$.
For convenience, when $p_1$ is even, we may choose
\[
\varsigma\in \{x\in E^\times\ \mid\ \eps x+\bar{x}=0\}.
\]
With the above choice, we have that $X_{e,\varsigma}(e_m)= \bar{\varsigma} e_{-m}$. Finally, we may complete $X_{e,\varsigma}$ and $\hbar$ to an ${\mathfrak {sl}}_2$-triple in $\Fg$
and obtain the decomposition $V=V_{(p_1)}\oplus V_{(1)}$ as in \eqref{eq:V-p}.
More precisely,
\[V_{(p_1)}= {\rm {Span}}\{e_1,e_2,\dots,e_{m}\}\oplus Ee \oplus {\rm {Span}}\{e_{-m},e_{-(m-1)},\dots,e_{-1}\}\]
and
$V_{(1)}=V_{(1)}^h=V_{0}\cap e^\perp$.
Note that $V_{(p_1)}^h=Ee_{-1}$ and
\begin{equation}\label{eq:X-highest-weight}
X^{p_1-1}(e_{1})=\begin{cases}
	(-\varsigma\bar{\varsigma})^{m-1}\bar{\varsigma} e_{-1} &\text{ if $p_1$ is even,}\\
	\eps\cdot(-\varsigma\bar{\varsigma})^m\apair{e,e}e_{-1} &\text{ if $p_1$ is odd.}\\	
\end{cases}	
\end{equation}
By \eqref{eq:X-q-i} and \eqref{eq:X-highest-weight}, $q^{h}_{(p_1)}(e_{-1},e_{-1})$ is equivalent to the one-dimensional $(-1)^{p_1-1}\eps$-Hermitian form defined by
\[
q'_{e,\varsigma}(e_{-1},e_{-1}):= \begin{cases}
	(-1)^{m-1}\varsigma &\text{ if $p_1$ is even,}\\
	(-1)^m \apair{e,e} &\text{ if $p_1$ is odd.}\\	
\end{cases}
\]
Denote the resulting  $\epsilon$-Hermitian form on $V_{(p_1)}\cong Ee_{-1}\otimes F^{p_1}$ by
\[
q_{(p_1), e, \varsigma}:=q'_{e,\varsigma}\otimes q_{p_1}.
\]
Therefore, we obtain the admissible Young tableau associated to $X_{e,\varsigma}$:
\begin{equation}\label{eq:X-p-1}
([p_1,1^{\Fn-p_1}],\{(Ee_{-1},q'_{e,\varsigma}),(V_{0}\cap e^{\perp},q_V\vert_{V_{0}\cap e^{\perp}})\}).	
\end{equation}
The following lemma summarizes the above calculation.
\begin{lem}
For $\Fr\geq 1$,
there is a one-to-one correspondence between the $F$-rational nilpotent orbits in $\CO^{\st}_{p_1}$ and
\begin{itemize}
	\item $\Isom(V_{\an})^\circ$-orbits on the anisotropic lines in $V_{\an}$ if $p_1=2\Fr+1$;
	\item 	$F^\times/\BN E^{\times}$ otherwise.
\end{itemize}
\end{lem} 	


\subsection{Twisted Jacquet module associated to $[p_1,1^{\Fn-p_1}]$}\label{ssec:TJMp1}

The twisted Jacquet module associated to the partition of type $[p_1,1^{\Fn-p_1}]$ splits into two types: the Bessel type and the Fourier-Jacobi type, according to $p_1$ is odd and even,
respectively. Note that when $p_1=1$ the stable nilpotent orbit corresponds to the zero nilpotent orbit.
However, the twisted Jacquet modules of the Bessel type and Fourier-Jacobi type for the classical groups can still be defined for $p_1=1$ (\cite{GGP12}), which will also be given in this section.

First, for $X_{e,\varsigma}$ as given by \eqref{xe,var}, to compute $\kappa(u,X_{e,\varsigma})$ for $u\in \Fu_{X_{e,\varsigma},2}$,
we introduce the following identity.
\begin{lem}\label{lm:kappa-X-A}
For $X\in \Fg$, we have
\[
\kappa(X,A_{v,w})=\frac{1}{2}(\apair{X(v),w}+\overline{\apair{X(v),w}}).
\]
\end{lem}
\begin{proof}
By \eqref{eq:kappa}, we have
\begin{align}
 2\kappa(X,A_{v,w})\label{eq:decomp-X-A}
=&\sum_{i=1}^{\Fr}\left(\eps \overline{\apair{e_{-i},v}}\apair{X(e_{i}),w}-\overline{\apair{e_{-i},w}}\apair{X(e_{i}),v}\right)\nonumber\\
&\qquad +\sum_{i=1}^{\Fr}\left( \overline{\apair{e_{i},v}}\apair{X(e_{-i}),w}-\eps \overline{\apair{e_{i},w}}\apair{X(e_{-i}),v}\right) \\
&\qquad\qquad +\sum_{j=1}^{\Fd_0}
\left(\eps\frac{\overline{\apair{e'_{j},v}}}{\apair{e'_{j},e'_{j}}}
\apair{X(e'_{j}),w}-
\frac{\overline{\apair{e'_{j},w}}}{\apair{e'_{j},e'_{j}}}\apair{X(e'_{j}),v}\right).\nonumber
\end{align}
According to the basis $\FB$ in \eqref{eq:basis},
for each $x\in V$, we may decompose $x$ as
\[
x=\sum^{\Fr}_{i=1}\apair{x,e_{-i}}e_{i}+\eps \sum^{\Fr}_{i=1}\apair{x,e_{i}}e_{-i}
+\sum_{i=1}^{\Fd_0}\frac{\apair{x,e'_{i}}}{\apair{e'_i,e'_i}}e'_i.
\]
Plugging the decomposition for $v$ and for $w$ into \eqref{eq:decomp-X-A}, we obtain
\[
2\kappa(X,A_{v,w})=\apair{X(v),w}-\eps \apair{X(w),v}.
\]
As $X=-X^*$, $\apair{X(w),v}=-\apair{w,X(v)}=-\eps\overline{\apair{X(v),w}}$.
This completes the proof.
\end{proof}

Recall the Lie algebra $\Fu_{X_{e,\varsigma}, 2}$ from \eqref{uX2}. By Lemma \ref{lm:kappa-X-A}, we obtain the homomorphism
\[
f\colon \Fu_{X_{e,\varsigma}, 2} \to F^m
\]
defined by, for $u\in \Fu_{X_{e,\varsigma}, 2}$,
\begin{align*}
f(u)=&(\kappa(u,A_{e_{2},\bar{\varsigma} e_{-1}}),\kappa(u,A_{e_{3},\bar{\varsigma} e_{-2}}),\dots,\kappa(u,A_{e_{m},\bar{\varsigma} e_{-(m-1)}}),\kappa(u,A_{e,\bar{\varsigma} e_{-m}}))\\
=&\frac{1}{2}(\tr_{E/F}\varsigma\apair{ue_{2}, e_{-1}},\dots,\tr_{E/F}\varsigma\apair{ue_{m}, e_{-(m-1)}},\tr_{E/F}\varsigma \apair{ue, e_{-m}})
\end{align*}
where $m=\lfloor\frac{p_1}{2}\rfloor$. Then $\kappa(\cdot, X_{e,\varsigma})$ is the composition of $f$ and the addition map $F^m\to F$. The character of $U_{X_{e,\varsigma}, 2}$ is given by
\begin{equation} \label{psix}
\psi_{X_{e,\varsigma},2}(\exp (u))=\psi^F(\kappa(u, X_{e,\varsigma})),\quad u\in \Fu_{X_{e,\varsigma}, 2}.
\end{equation}
We obtain that
\[
M_{X_{e,\varsigma}}=\Isom(V_{(1)}, q),
\]
where $V_{(1)}=V_0\cap e^\perp$. By the general construction in Section \ref{ssec-TJM}, we can formulate the twisted Jacquet modules of either Bessel type or Fourier-Jacobi type.

\subsubsection{Bessel type} 
If $\epsilon=1$ and $p_1>1$ is odd, then $\Fu_{X_{e,\varsigma}}=\Fu_{X_{e,\varsigma},2}$ and $\omega_{X_{e,\varsigma}}=\psi_{X_{e,\varsigma}}$. For a smooth representation $(\pi, V_\pi)$ of $G_n$, we define its twisted Jacquet module with respect to $(U_{X_{e,\varsigma}}, \psi_{X_{e, \varsigma}})$ by
\[
\CJ_{X_{e, \varsigma}}(\pi)=V_\pi /\overline{\{\pi(u)v-\psi_{X_{e, \varsigma}}(u)v\ \mid\  u\in U_{X_{e, \varsigma}}, v\in V_\pi\}},
\]
as a smooth module of $M^\circ_{X_{e,\varsigma}}$.

If $\epsilon=1$ and $p_1=1$, then the corresponding nilpotent orbit is zero. In this case, given any anisotropic vector $e\in V$, we  write 
$\CJ_{0_{e,\varsigma}}(\pi)$ for the restriction of $\pi$ to the connected stabilizer of $e$ in $G_n(F)$. 

\subsubsection{Fourier-Jacobi type} 
If $\epsilon=-1$ and $p_1$ is even, then we have the oscillator-Heisenberg representation
$(\omega_{\psi_{X_{e,\varsigma}}}, V_{\psi_{X_{e, \varsigma}}})$ of $\wt{M}_{X_{e,\varsigma}}\ltimes U_{X_{e,\varsigma}}$. For a smooth representation
$(\pi, V_\pi)$ of $G_n(F)$ (or $\wt{G}_n(F)$), we define the $U_{X_{e, \varsigma}}$-coinvariants
\[
\CJ_{X_{e, \varsigma}}(\pi)=V_\pi\,\wh{\otimes}\, V^\vee_{\psi_{X_{e,\varsigma}}}/\overline{\{\pi\otimes \omega^\vee_{\psi_{X_{e,\varsigma}}}(u)v-v\ \mid\  u\in U_{X_{e, \varsigma}}, v\in V_\pi
\,\wh{\otimes}\, V^\vee_{\psi_{X_{e,\varsigma}}}\}},
\]
as a smooth module of $\wt{M}_{X_{e,\varsigma}}$ (or $M_{X_{e,\varsigma}}$).

If $\epsilon =-1$ and $p_1=1$, then the corresponding nilpotent orbit is zero. In this case, we write $\omega_{\psi_{0_{e, \varsigma}}}$ for the oscillator representation 
of $\wt{G}_n(F)$ and write $\CJ_{0_{e,\varsigma}}(\pi)$ for the restriction of $\pi\,\wh{\otimes}\,\omega_{\psi_{0_{e, \varsigma}}}$ to $\wt{G}_n(F)$ (or $G_n(F)$).

In the skew-Hermitian case, choose a splitting character $\mu: E^\times/\BN E^\times \to \BC^\times$ whose restriction to $F^\times$ is the quadratic character $\omega_{E/F}$. Then we may and do remove the double cover and use the oscillator-Heisenberg representation $(\omega_{\psi_{X_{e,\varsigma}},\mu},V_{\psi_{X_{e,\varsigma}},\mu})$ of the unitary group $M_{X_{e,\varsigma}}$ in the above.

\subsection{Explicit data}\label{ssec-ED}

We are now ready to write down the exact information that we need for this paper from the local Gan-Gross-Prasad conjecture (\cite{GGP12}). We will do this case by case, in order to take into account the compatibility of the local conjectures (\cite[\S18]{GGP12}) in our formulation through the rationality of nilpotent orbits.

\subsubsection{Special orthogonal groups}\label{sec:GGP-SO}
Let $(\varphi, \phi) \in \Phi_\gen(\SO(V^*)\times \SO(W^*))$, where $V^*\supset W^*$ is a relevant pair of orthogonal spaces of dimensions $\Fn>\Fn'$ with opposite parities. Then the local Gan-Gross-Prasad conjecture asserts that up to isometry there is a unique relevant pure inner form $\SO(V)\times \SO(W)$ and  a unique pair of representations $(\pi, \pi')\in \Pi_\varphi(\SO(V))\times \Pi_\phi(\SO(W))$ up to $\sim_{\rm c}$ equivalence,  such that
\[
\Hom_{P_{X_{e,\varsigma}}}(\pi\boxtimes \pi', \psi_{X_{e,\varsigma}})\neq 0,
\]
with  $p_1=\Fn-\Fn'$, $m=\lfloor p_1/2\rfloor$, $W^\perp=\BH^m+F e$ and $X_{e, \varsigma}\in \CO^{\st}_{p_1}$. Here and below, $\BH$ denotes a hyperbolic plane (over $E$).
Moreover, the pair $(\pi, \pi')$ is determined by
\[
\pi=\pi_a(\varphi, \chi_{\varphi,\phi}),\quad \pi'=\pi_a(\phi,\chi_{\phi,\varphi}),
\]
where $(\chi_{\varphi, \phi}, \chi_{\phi, \varphi})$ is given by \eqref{char-O}, and the (weak) local Langlands correspondence is via $\iota_a$ with
\[
a=-\disc(V)\disc(W)=(-1)^\Fn\langle e, e\rangle.
\]

\subsubsection{Symplectic-metaplectic groups}   \label{sec:GGP-Sp}
Let $(\varphi, \phi) \in \Phi_\gen(\Sp(V)\times \Mp(W))$ or $\Phi_\gen(\Mp(V)\times \Sp(W))$, where $V\supset W$ is a pair of symplectic spaces of dimensions  $2n\geq 2n'$. Then the local Gan-Gross-Prasad conjecture asserts that for $\varsigma\in F^\times$,  there is  a unique pair of representations $(\pi, \pi')\in \Pi_\varphi(\Sp(V))\times \Pi_\phi(\Mp(W))$ or $\Pi_\varphi(\Mp(V))\times \Pi_\phi(\Sp(W))$ such that
\[
\Hom_{\wt{P}_{X_{e,\varsigma}}}(\pi\boxtimes \pi', \omega_{\psi_{X_{e, \varsigma}}})\neq 0,
\]
with $p_1=2n-2n'$, $m=p_1/2$, $W^\perp =\BH^m$ and $X_{e, \varsigma}\in \CO^{\st}_{p_1}$.
Moreover, the pair $(\pi, \pi')$ is determined by
\[
\pi=\pi_1(\varphi, \chi^a_{\varphi,\phi}),\quad \pi'=\pi_1(\phi,\chi^a_{\phi,\varphi}),
\]
where  $(\chi^a_{\varphi, \phi}, \chi^a_{\phi, \varphi})$ is given by \eqref{char-spmp}, $a=(-1)^{m-1}\varsigma$, and the local Langlands correspondence is via $\iota_1$. The above is a reformulation of \cite[Proposition 18.1]{GGP12}.

\subsubsection{Unitary groups: Bessel type}  \label{sec:GGP-U-B}
Let $(\varphi, \phi) \in \Phi_\gen(\RU(V^*)\times \RU(W^*))$, where $V^*\supset W^*$ is a relevant pair of Hermitian spaces of dimensions $\Fn>\Fn'$ with opposite parities. Then the local Gan-Gross-Prasad conjecture asserts that  
up to isometry there is a unique relevant pure inner form $\RU(V)\times \RU(W)$ and  a unique pair of representations $(\pi, \pi')\in \Pi_\varphi(\RU(V))\times \Pi_\phi(\RU(W))$,  such that
\[
\Hom_{P_{X_{e,\varsigma}}}(\pi\boxtimes \pi', \psi_{X_{e,\varsigma}})\neq 0,
\]
with  $p_1=\Fn-\Fn'$, $m=\lfloor p_1/2\rfloor$, $W^\perp=\BH^m+E e$ and $X_{e, \varsigma}\in \CO^{\st}_{p_1}$.
Moreover, the pair $(\pi, \pi')$ is determined by
\[
\pi=\pi_a(\varphi, \chi_{\varphi,\phi}),\quad \pi'=\pi_{-a}(\phi,\chi_{\phi,\varphi}),
\]
where $(\chi_{\varphi, \phi}, \chi_{\phi, \varphi})$ is given by \eqref{char-U}, and 
\[
a=(-1)^{\Fn-1}\disc(V)\disc(W) = \langle  e,e \rangle.
\]

\subsubsection{Unitary groups: Fourier-Jacobi type} \label{sec:GGP-U-FJ}
Fix a splitting character $\mu$ that is used to define the oscillator-Heisenberg representation. Let $(\varphi, \phi) \in \Phi_\gen(\RU(V^*)\times \RU(W^*))$, where $V^*\supset W^*$ is a relevant pair of skew-Hermitian spaces of dimensions $\Fn \geq \Fn'$ with the same parity. Then the local Gan-Gross-Prasad conjecture asserts that for $\varsigma\in F^\times$,
up to isometry there is a unique relevant pure inner form $\RU(V)\times \RU(W)$ and  a unique pair of representations $(\pi, \pi')\in \Pi_\varphi(\RU(V))\times \Pi_\phi(\RU(W))$,  such that
\[
\Hom_{P_{X_{e,\varsigma}}}(\pi\boxtimes \pi', \omega_{\psi_{X_{e,\varsigma}},\mu})\neq 0,
\]
with  $p_1=\Fn-\Fn'$, $m=p_1/2$, $W^\perp =\BH^m$ and $X_{e, \varsigma}\in \CO^{\st}_{p_1}$.
Moreover, the pair $(\pi, \pi')$ is determined by
\[
\pi=\pi_1(\varphi, \chi_{\varphi(\mu^{-1}),\phi}\cdot \eta_a),\quad \pi'=\pi_1(\phi,\chi_{\phi,\varphi(\mu^{-1})}\cdot\eta_a),
\]
where $(\chi_{\varphi(\mu^{-1}), \phi}, \chi_{\phi, \varphi(\mu^{-1})})$ is given by \eqref{char-U},
$
a= (-1)^{m-1}\varsigma,
$
and the local Langlands correspondence is via $\iota_1$.

\

The above recipes for unitary groups can be easily verified using  \cite[\S17]{GGP12} and \cite{GI16}, and the properties of local Langlands correspondence and local root numbers that we discussed earlier.


\section{Wavefront Sets}\label{sec-AWFS}


With the preparation in Sections \ref{sec-CGRNO}, \ref{sec-LLC},  and \ref{ssec-LGGP}, 
we are ready to enter the the main theme of this paper. We develop the method of consecutive descents of enhanced $L$-parameters and introduce combinatorial objects that carry the $F$-rational structures or arithmetic data.
This leads to our definition of {\it arithmetic wavefront sets} for enhanced $L$-parameters 
of generic type. We study the basic structures of the arithmetic wavefront sets and discuss the Wavefront Set Conjecture.

\subsection{Descent of enhanced $L$-parameters}\label{ssec-LPD}

Recall that $E=F$ or $F(\delta)$, and $c\in \textrm{Gal}(E/F)$ is the identity or Galois conjugation, respectively. To unify the notations, denote  $\widetilde{\Phi}_\gen(G_n^*)=\Phi_\gen(G_n^*)$ if $E=F$, and $\widetilde{\Phi}_\gen(G_n^*)= \Phi_\gen(G_n^*)\bigcupdot\Phi_\gen(G_n^*)\otimes\mu$ the set of all $\mathfrak{n}$-dimensional conjugate self-dual $\CL_E$-representations if $E=F(\delta)$, where $\mu$ is any  character of $E^\times$ such that $\mu|_{F^\times}=\omega_{E/F}$  is the quadratic character of $\CZ=F^\times/\BN E^\times$ given by the local class field theory. Recall from
Section \ref{ssec-LVP} that if $E=F(\delta)$ then $\Phi_\gen(G_n^*)$ consists of conjugate self-dual $\CL_E$-representations  of signature $(-1)^{\Fn-1}$. This modification is made to avoid the choice of a splitting character $\mu$ in the Fourier-Jacobi model of  unitary groups. By abuse of terminology, elements of $\wt{\Phi}_\gen(G_n^*)$ will be also called generic $L$-parameters. 

Let $\varphi\in\widetilde{\Phi}_\gen(G_n^*)$ be an $L$-parameter given by the form \eqref{edec}. For $\chi\in \widehat{\CS_\varphi}$, define its $\CZ$-orbit
\[
\ScO_\CZ(\chi)=\{\chi\cdot\eta_z \in \widehat{\CS_\varphi}\ \mid\ z\in\CZ\}.
\]
We are going to define the notion of parameter descent  for the enhanced $L$-parameter  $(\varphi,\chi)$ with $\varphi\in \widetilde{\Phi}_\gen(G_n^*)$ and $\chi\in
\widehat{\CS_\varphi}$, along with an $\Fn$-dimensional $\epsilon$-Hermitian vector space $(V, q)$ over $E$.

As in Section \ref{ssec-CGIF}, let $(V, q)$ be an $\Fn$-dimensional $\epsilon$-Hermitian vector space over $E$. Let $\Fl$ be an integer with $0<\Fl\leq \Fn$ such that
$
(-1)^{\Fl-1}=\epsilon
$
if $E=F$. There is no restriction on the parity of $\Fl$ if $E=F(\delta)$. That is, we will consider both Bessel and Fourier-Jacobi descents for the unitary group case. 
Consider relevant pairs of $F$-quasisplit classical groups in the sense of the local Gan-Gross-Prasad conjecture:
\[
(G_n^*, H_{\lfloor(\mathfrak{n}-\Fl)/2\rfloor}^*)=\begin{cases}
(\SO_\mathfrak{n}, \SO_{\mathfrak{n}-\Fl}), & \textrm{if }E=F\textrm{ and }\epsilon=1,\\
(\Sp_{2n}, \Mp_{2n-\Fl}) \textrm{ or }(\Mp_{2n}, \Sp_{2n-\Fl}), & \textrm{if }E=F\textrm{ and }\epsilon=-1,\\
(\RU_\mathfrak{n}, \RU_{\mathfrak{n}-\Fl}), & \textrm{if }E=F(\delta).
\end{cases}
\]
If $\Fl=\Fn$, we regard $H_{\lfloor(\mathfrak{n}-\Fl)/2\rfloor}^*$ as the trivial group. Note that $\Fl$ is odd in the Bessel case and is even in the Fourier-Jacobi case.

\begin{defn}[$L$-parameter Descent]\label{defn:pd}
For any given enhanced $L$-parameter $(\varphi,\chi)$ with 
$\varphi\in \widetilde{\Phi}_\gen(G_n^*)$ and $\chi\in\widehat{\CS_\varphi}$,
the $\Fl$-th  descent of the enhanced $L$-parameter $(\varphi,\chi)$ along $z\in \CZ$, which is denoted by $\FD_{\Fl}^z(\varphi,\chi)$, is defined to be the set of contragredients $(\wh\phi, \wh{\chi'})$ of all enhanced $L$-parameters $(\phi,\chi')$ with 
$\phi\in \widetilde{\Phi}_\gen(H_{\lfloor(\mathfrak{n}-\Fl)/2\rfloor}^*)$ and $\chi'\in\wh{\CS_\phi}$ satisfying the following conditions:
\begin{enumerate}
\item  the type of $\phi$ is opposite to that of $\varphi$, and 
\item the following 
\[
(\chi^z_{\varphi,\phi}, \chi^z_{\phi, \varphi} ) = (\chi, \chi'),
\]
holds, where $(\chi^z_{\varphi,\phi},\chi^z_{\phi,\varphi})$ is defined, case by case, as in Section {\rm \ref{ssec-DC}}.
\end{enumerate}
The $\Fl$-th  descent of the enhanced $L$-parameter $(\varphi,\chi)$  is defined to be
\begin{equation}\label{Fl-pd}
\FD_\Fl(\varphi, \chi) := \bigcupdot_{z\in \CZ} \FD_\Fl^z(\varphi, \chi).
\end{equation}
\end{defn}

Note that condition (1) in Definition \ref{defn:pd} is automatic when $E=F$, and is required when $E=F(\delta)$ due to our modification $\widetilde\Phi_
\gen(G_n^*)$.
Recall that the contragredient $(\wh\phi,\wh{\chi'})$ of $(\phi,\chi')$ is defined as in Section \ref{ssec-C} and is explicitly given in Proposition \ref{prop:dualdata}, 
case by case.  
We note from Section \ref{ssec-DC} that except for the symplectic-metaplectic case, it holds that
\begin{equation} \label{chiz}
\chi^z_{\varphi,\phi} = \chi_{\varphi, \phi}\cdot \eta_z.
\end{equation}
We will see in Section \ref{sec-LPDC} that when $F=\BR$, for the character $\chi_{\varphi, \phi}^z$ given by \eqref{char-spmp}  in the symplectic-metaplectic case,
equation \eqref{chiz} also holds.

\begin{prop}\label{prop:DNE}
For any enhanced $L$-parameter $(\varphi,\chi)$ with $\varphi\in\widetilde{\Phi}_\gen(G_n^*)$ and $\chi\in\wh{\CS_\varphi}$, there exists an integer $\Fl$ with $0<\Fl\leq\Fn$ such that
the $\Fl$-th descent $\FD_\Fl(\varphi, \chi)$ is non-empty.
\end{prop}

The archimedean case of Proposition \ref{prop:DNE} will be given in Proposition \ref{prop:DNE-A}. The non-archimedean case of Proposition \ref{prop:DNE} (Corollary \ref{cor:DNE-NA}) follows from the the proof of Proposition \ref{prop:LDPT}, which will be given in Appendix 
\ref{App-A}.

\quad

The {\it first occurrence index} of the descents of the enhanced $L$-parameter $(\varphi, \chi)$ is defined to be
\begin{equation} \label{FO}
\Fl_0=\Fl_0(\varphi, \chi):=\max\{0<\Fl\leq \Fn\ \mid\ \FD_\Fl(\varphi, \chi)\neq \varnothing\};
\end{equation}
and the $\Fl_0$-th  descent $\FD_{\Fl_0}(\varphi, \chi)$ is called the {\it first descent} of the enhanced $L$-parameter $(\varphi,\chi)$.
At the first occurrence index $\Fl_0=\Fl_0(\varphi, \chi)$, the first descent $\FD_{\Fl_0}(\varphi, \chi)$ has much better property, which is explained below.
As in \cite{MT02, JZ18}, we recall that an $L$-parameter $\varphi$ as in \eqref{decomp} or an enhanced $L$-parameter $(\varphi,\chi)$ is called discrete if every irreducible summand $\varphi_i$ in \eqref{decomp} is of good parity and has multiplicity $m_i=1$. 

\begin{thm}[Discreteness of First Descent]\label{thm:DFD}
For any enhanced $L$-parameter $(\varphi,\chi)$ with $\varphi\in\widetilde{\Phi}_\gen(G_n^*)$ and $\chi\in\wh{\CS_\varphi}$, the first descent $\FD_{\Fl_0}(\varphi, \chi)$ at 
the first occurrence index $\Fl_0=\Fl_0(\varphi, \chi)$ contains only discrete enhanced $L$-parameters. 
\end{thm}

\begin{proof}
By definition, the first descent $\FD_{\Fl_0}(\varphi, \chi)$ is non-empty. For any $(\wh{\phi},\wh{\chi'})\in \FD_{\Fl_0}(\varphi, \chi)$, we have to show that $\phi$ is a discrete $L$-parameter of 
$H_{\lfloor(\mathfrak{n}-\Fl_0)/2\rfloor}^*$. From Definition \ref{defn:pd}, both 
$\varphi$ and $\phi$ must be generic $L$-parameters, but be of opposite types. We may write 
them in the form 
\eqref{edec}, with index sets $\RI_\gp$ and $\RI'_\gp$ respectively. 
More explicitly, we may write $\phi$ as 
\[
\phi=\bigoplus_{i'\in \RI_\gp'}m_{i'}\phi_{i'} \oplus \bigoplus_{j' \in \RI_\bp'}2m'_{j'}\phi_{j'} \oplus \bigoplus_{k' \in \RI'_\nsd}m_{k'}(\phi_{k'}\oplus{}^c \phi_{k'}^\vee).
\]
Take the following summand of $\phi$:
\[
\phi_o := \bigoplus_{i'\in \RI'_\gp, \, m_{i'} \textrm{ is odd}} \phi_{i'}.
\]
It is clear that $\phi_o$ is a discrete $L$-parameter. It is enough to show that there exists a character $\chi_o\in\wh{\CS_{\phi_o}}$ such that the enhanced $L$-parameter $(\wh{\phi_o},\wh{\chi_o})$ belongs to a descent of enhanced $L$-parameter $(\varphi,\chi)$. In fact, if $(\wh{\phi_o},\wh{\chi_o})\in \FD_{\Fl_o}(\varphi, \chi)$ for some integer $\Fl_o$ with $0<\Fl_o\leq\Fn$, then from the definition of $\phi_o$ and 
by dimension counting, we must have that $\Fl_o=\Fl_0+\dim\phi-\dim\phi_o\geq\Fl_0$. Since $\Fl_0=\Fl_0(\varphi, \chi)$ is the first occurrence index of $(\varphi,\chi)$, we must have that $\Fl_o=\Fl_0$, 
which implies that $\phi=\phi_o$. Hence $\phi$ is a discrete $L$-parameter. 

Now we are going to prove that $(\wh{\phi_o},\wh{\chi_o})\in \FD_{\Fl_o}(\varphi, \chi)$ for $\Fl_o:=\Fl_0+\dim\phi-\dim\phi_o\geq \Fl_0$ and some $\chi_o\in \wh{\CS_{\phi_o}}$. By assumption, $(\wh{\phi},\wh{\chi'})\in \FD_{\Fl_0}^z(\varphi,\chi)$ for some $z\in \CZ$, which gives that $\chi_{\varphi,\phi}^z=\chi$. 
Then it suffices to prove that
\begin{equation} \label{chi=chio}
\chi_{\varphi, \phi} = \chi_{\varphi, \phi_o},
\end{equation} 
because it implies that
\[
(\wh{\phi_o}, \wh{\chi_o})\in \FD_{\Fl_o}^z(\varphi, \chi)\subset \FD_{\Fl_o}(\varphi,\chi)
\]
with $\chi_o:=\chi_{\phi_o,\varphi}^z$. Put
$
\phi':=\phi\ominus\phi_o,
$
which is of the form
\begin{equation} \label{phi'}
\phi'=\sigma\oplus{}^c\sigma^\vee
\end{equation}
for some representation $\sigma$ of $\CL_E$. Since 
$\chi_{\varphi,\phi} = \chi_{\varphi,\phi_o}\cdot \chi_{\varphi, \phi'}$, to prove \eqref{chi=chio} we only need to show that
\begin{equation} \label{chi'trivial}
\chi_{\varphi, \phi'}={\mathds 1}\in\wh{\CS_\varphi},
\end{equation}
the trivial character of $\CS_\varphi$. We prove \eqref{chi'trivial}
for $E=F$ or $F(\delta)$ separately.

\quad

{\bf Case 1.} $E=F$.  We only prove \eqref{chi'trivial} for special orthogonal groups, and the proof for the symplectic-metaplectic case is similar.  Note that in this case $\phi'=\sigma\oplus\sigma^\vee$ and in particular $\det\phi'=1$. Thus from \eqref{char-O} we have that
\[
 \chi_{\varphi,\phi'}((e_i)_{i\in \RI_\gp})= \prod_{i\in \RI_\gp}\left(\varepsilon(\varphi_i\otimes\phi') (\det \varphi_i)(-1)^{\dim \sigma}\right)^{e_i}.
\]
Since $\varphi_i$ is self-dual and $\phi'=\sigma\oplus\sigma^\vee$, by \cite[Proposition 5.1]{GGP12} we have that
\[
\varepsilon(\varphi_i\otimes\phi') = \det(\varphi_i\otimes \sigma)(-1) = (\det\varphi_i)(-1)^{\dim\sigma}(\det\sigma)(-1)^{\dim\varphi_i},
\]
which gives that
\begin{equation} \label{chi'formula}
\chi_{\varphi,\phi'}((e_i)_{i\in \RI_\gp})= \prod_{i\in \RI_\gp}\left( (\det \sigma)(-1)^{\dim \varphi_i}\right)^{e_i}.
\end{equation}
We have two cases:
\begin{itemize}
    \item $\varphi$ is symplectic. In this case $\dim\varphi_i$ is even. Hence following the convention in Section~\ref{ssec-DC} or \cite{GGP12}, we have 
    \[
    (\det \sigma)(-1)^{\dim\varphi_i}=(\det \sigma)((-1)^{\dim\varphi_i}) =1
    \]
    for each $i\in\RI_\gp$. From  \eqref{chi'formula}, we have that 
    \[
    \chi_{\varphi,\phi'}((e_i)_{i\in\RI_\gp})=1
    \]
    for any $(e_i)_{i\in\RI_\gp}\in\CS_\varphi$,  which gives  \eqref{chi'trivial}.
    
    \item $\varphi$ is orthogonal. In this case $(e_i)_{i\in\RI_\gp}\in \CS_\varphi$ if and only if $\sum_{i\in\RI_\gp}e_i\dim\varphi_i$ is even. Then by  \eqref{chi'formula} we have that
    \[
    \chi_{\varphi,\phi'}((e_i)_{i\in\RI_\gp}) = (\det\sigma)(-1)^{\sum_{i\in\RI_\gp}e_i\dim\varphi_i}
    =(\det\sigma)((-1)^{\sum_{i\in\RI_\gp}e_i\dim\varphi_i})
    =1
    \]
    for any $(e_i)_{i\in\RI_\gp}\in\CS_\varphi$,
    which gives  \eqref{chi'trivial}.
\end{itemize}

\quad

{\bf Case 2.} $E=F(\delta)$.  Recall from \eqref{char-U} that  
\[
\chi_{\varphi, \phi'}((e_i)_{i\in \RI_\gp}) = \prod_{i\in \RI_\gp}\varepsilon(\varphi_i\otimes\phi', \psi^E_2)^{e_i}.
\]
Since $\varphi_i$ is conjugate self-dual and $\phi'$ is of the form \eqref{phi'},
by \cite[Proposition 5.1]{GGP12} we have that
\[
\varepsilon(\varphi_i\otimes\phi',\psi^E_2) = \varepsilon\left((\varphi_i\otimes\sigma)\oplus{}^c(\varphi_i\otimes \sigma)^\vee, \psi^E_2\right)=1,
\]
which proves \eqref{chi'trivial}.
\end{proof}

It is worthwhile to point out that when $F=\BR$, Theorem \ref{thm:DFD} has a further refinement, which is given in Proposition \ref{prop:fd}. The significance of Theorem \ref{thm:DFD} 
in representation theory is discussed in \cite{CJLZ24} (See \cite[Theorems 1.3, 1.4, 1.5, 1.6]{CJLZ24} for instance).

\subsection{Wavefront set for enhanced $L$-parameters}\label{ssec-WFSA}

Recall from \eqref{psi-rel} that we have fixed the nontrivial additive character $\psi^F$ and $\psi^E$ of $F$ and $E$, respectively. As in Section \ref{ssec-RLLC}, for a classical group $G_n=\Isom(V,q)^\circ$ or $\Mp(V)$, a choice of
$a\in\CZ$ determines a local Langlands correspondence $\iota_a$ for $G_n$.

Let $\FT_a(G^*_n)$ be the set of triples $(\varphi,\chi,q)$, consisting of  $\varphi\in \widetilde{\Phi}_\gen(G_n^*)$, $\chi\in \widehat{\CS_\varphi}$ and an isometry class $q$ of $\Fn$-dimensional non-degenerate $\epsilon_q$-Hermitian form $(V,q)$ such that
\[
\pi_a(\varphi\otimes\mu,\chi)\in \Pi_F(G_n)
\]
with $G_n=\textrm{Isom}(V,q)^\circ$ or $\Mp(V)$, where $\mu$ is trivial or a splitting character according to whether $\varphi\in \Phi_\gen(G_n^*)$ or not  in the case that $E=F(\delta)$. Note that $q$ is determined by the enhanced $L$-parameter $(\varphi,\chi)$ 
and $a\in \CZ$
by means of the rationality of the local Langlands correspondence. We keep it in the notation in order to track the rationality of nilpotent orbits.

In order to define the wavefront set associated to each triple $(\varphi,\chi,q)$ in $\FT_a(G^*_n)$, 
we introduce a notion of {\it pre-tableaux associated with ordered partitions}, which aims to collect combinatorial data from the $L$-parameter descents (Definition \ref{defn:pd}). 

Let  $\underline{\Fl}:=(\ell_1,\ell_2,\dots,\ell_k)$ be an ordered partition of $\Fn$ (Definition \ref{defn:partition}).
Define a {\it pre-tableau} $\Fs_{\underline{\Fl}}$ associated to the given partition $\underline{\Fl}$
to be a sequence of pairs as follows:
\begin{align}\label{Fs}
\Fs_{\underline{\Fl}}=\left((\ell_1,q^h_{1}),(\ell_2,q^h_{2}),\dots, (\ell_k,q^h_{k})\right),
\end{align}
where $q^h_i$ is a $1$-dimensional $\epsilon_i$-Hermitian form with $\epsilon_i=\pm 1$. We call $\Fs_{\underline{\Fl}}$ {\it decreasing} if the partition $\udl{\Fl}$ is decreasing, i.e.
$\ell_1\geq\ell_2\geq\cdots\geq\ell_k>0$.
We call $\Fs_{\underline{\Fl}}$ {\it admissible} for $(V,q_V)$ if each $q^h_{i}$ is $(-1)^{\ell_i-1}\epsilon$-Hermitian and the $\epsilon$-Hermitian space
\[
\left(\bigoplus^k_{i=1}E\otimes_E E^{\ell_i}, \bigoplus^k_{i=1}q_i^h\otimes q_{\ell_i}\right)
\]
is isometric to $(V,q_V)$.

If $(\ell,q^h)$ is a pair consisting of a positive integer $\ell$ and a $1$-dimensional $\epsilon_0$-Hermitan form $q^h$ with $\epsilon_0=\pm 1$, we define
\[
(\ell,q^h)\star \Fs_{\underline{\Fl}}:=\left((\ell,q^h),(\ell_1,q^h_{1}),(\ell_2,q^h_{2}),\dots, (\ell_k,q^h_{k})\right)
\]
to be the pre-tableau associated to the ordered partition
$[\ell,\ell_1,\ell_2,\dots,\ell_k]$. 
If $\FS:=\{\Fs_{\udl{\Fl}}\}$ is a set of certain pre-tableaux $\Fs_{\udl{\Fl}}$ associated to some ordered partitions $\udl{\Fl}$, we define
\begin{align}\label{add-PT}
(\ell,q^h)\star \FS
:=\{(\ell,q^h)\star \Fs_{\underline{\Fl}} \ \mid \ \Fs_{\udl{\Fl}}\in \FS\}.
\end{align}By convention, if $\ell_1=0$ and $q^h_1$ is the zero form, we also put
\begin{equation}\label{compose0}
(\ell, q^h)\star\{(\ell_1, q^h_1)\} := \{(\ell, q^h)\}.
\end{equation}
For any triple $(\varphi,\chi,q)\in \FT_a(G_n^*)$, we are able to define the set of pre-tableaux associated with the given triple $(\varphi,\chi,q)$ inductively by means of the consecutive descents of enhanced $L$-parameters. 

For convenience, we introduce the following element of $\CZ$,
\[
[\disc(q)]:=\begin{cases} \disc(q)/\delta^\Fn, & \textrm{if }E=F(\delta) \textrm{ and }\epsilon_q=-1, \\
\disc(q), & \textrm{otherwise}.
\end{cases}
\]

\begin{defn}[Pre-tableaux for $(\varphi,\chi,q)$]\label{LYT}
For a given triple $(\varphi,\chi,q)\in \FT_a(G_n^*)$, the associated set of pre-tableaux, which is denoted by $\CT_a(\varphi,\chi,q)$, is defined by the following inductive process:
\begin{enumerate}
\item If $\Fn\leq 1$, we define $\CT_a(\varphi, \chi, q) := \{(\Fn, q^h_{(\Fn)})\}$, as in Definition \ref{defn:AYT}. 
\item If $\Fn\geq 2$, we define $\CT_a(\varphi, \chi, q)$ inductively  by the following formula:
\begin{equation} \label{YA}
\CT_a(\varphi,\chi,q):=\bigcup_{0<\Fl\leq \Fn} \bigcup_{z\in \CZ} \bigcup_{(\wh\phi,\wh{\chi'})\in\FD_{\Fl}^z(\varphi,\chi)}
 (\Fl,q^h_{(\Fl)}) \star  \CT_{a}(\wh\phi,\wh{\chi'},q'),
\end{equation}
where $(\wh\phi,\wh{\chi'},q')\in \FT_{a}(H_{\lfloor(\mathfrak{n}-\Fl)/2\rfloor}^*)$,
and $q_{(\Fl)}^h$'s are isometric classes of one-dimensional $(-1)^{\Fl-1}\epsilon_q$-Hermitian forms satisfying that
\begin{itemize}
	\item $q \cong q_{(\Fl)}\oplus q'$, where $q_{(\Fl)}= q_{(\Fl)}^h\otimes q_\Fl$ is as given in Definition \ref{defn:AYT} and $q'$ is the orthogonal complement. 
	\item $a$, $z$ and $q^h_{(\Fl)}$ are related by
	\[
	a \cdot z = \begin{cases} (-1)^\Fn \cdot \disc(q_{(\Fl)}^h), & \textrm{if }E=F, \\
	 \, [\disc(q^h_{(\Fl)})], & \textrm{if }E=F(\delta).
	 \end{cases}
	\]
\end{itemize}
\end{enumerate}
\end{defn}

Note that we apply the convention \eqref{compose0} in the inductive definition \eqref{YA} of $\CT_a(\varphi,\chi,q)$ when $\Fl=\Fn$. It will be useful to specify the one-dimensional sesquilinear form $q^h_{(\Fl)}$ in Definition \ref{LYT}  as follows. 

\begin{prop} \label{prop:rational}
Assume that $(\varphi,\chi,q)\in \FT_a(G_n^*)$,  $(\wh{\phi},\wh{\chi'})\in \FD_{\Fl}^z(\varphi,\chi)$ for some $0<\Fl\leq \Fn$ and $z\in\CZ$, and $q^h_{(\Fl)}$ is a one-dimensional 
sesquilinear form  such that $\CT_a(\varphi,\chi,q)$ contains a non-empty subset of the form
\[
(\Fl, q^h_{(\Fl)})\star \CT_{a}(\wh{\phi}, \wh{\chi'}, q')
\]
subject to the conditions in Definition \ref{LYT}.
Then $q^h_{(\Fl)}$ is determined explicitly as follows. 
\begin{enumerate}
    \item \label{FD-odd-SO-rational} If $G_n^*=\SO_{2n+1}^*$, then 
    \[
    \disc(q^h_{(\Fl)}) = \disc(q)\cdot \det(\phi). 
    \]
    In this case $z= -a \cdot \disc(q^h_{(\Fl)})$.
  \item \label{FD-rational} If $G_n^*\neq\SO_{2n+1}^*$,  then 
  \[
  [\disc(q^h_{(\Fl)})] = a\cdot z.
  \]
\end{enumerate}

\end{prop}

\begin{proof}
We explain the proof for  \eqref{FD-odd-SO-rational}. The other cases in \eqref{FD-rational} are direct consequences of  Definition \ref{LYT}.
If $G_n^*=\SO_{2n+1}^*$, then $(\phi, \chi')\in \FD_{\Fl}^z(\varphi,\chi)$
 is an enhanced $L$-parameter of an even special orthogonal group 
 $\SO(V', q')$. By the local Langlands correspondence,
 \[
 \disc(q') = \det(\phi).
 \]
The assertions of the proposition in this case follow from the fact that 
 \[
 \disc(q^h_{(\Fl)}) = \disc(q)\cdot \disc(q'),
 \]
 and the conditions in Definition \ref{LYT}.
\end{proof}

We denote by $\CL_a(\varphi,\chi,q)$ the subset of $\CT_a(\varphi,\chi,q)$ consisting of all decreasing pre-tableaux in $\CT_a(\varphi,\chi,q)$, which are called the {\it $L$-descent pre-tableaux} associated to $(\varphi,\chi,q)$. It is important to know that the set 
$\CL_a(\varphi,\chi,q)$ is non-empty for a given datum. 

\begin{prop}\label{prop:LDPT}
The set of $L$-descent pre-tableaux $\CL_a(\varphi,\chi,q)$ is non-empty for any given triple $(\varphi,\chi,q)\in \FT_{a}(G_n^*)$. 
\end{prop}

The archimedean case of Proposition \ref{prop:LDPT} will follow from the explicit 
construction of the first descent $\FD_{\Fl_0}(\varphi,\chi)$ in Proposition \ref{prop:fd}.
The non-archimedean case of Proposition \ref{prop:LDPT} will be proved in  Appendix \ref{App-A} by using representation theory and the result from the local Gan-Gross-Prasad 
conjecture. 

\quad

We are going to construct admissible sesquilinear Young tableaux in $\CY(V,q)$ from $L$-descent pre-tableaux in $\CL_a(\varphi,\chi,q)$, which is a non-empty set by Proposition \ref{prop:LDPT}, by the composition of 
admissible sesquilinear Young tableaux for the $\Fn$-dimensional $\epsilon$-Hermitian space $(V,q)$. 

For any admissible decreasing pre-tableau 
\begin{align}\label{Fs-decreasing}
\Fs_{\udl{\Fl}}=\left((\ell_1,q^h_{1}),(\ell_2,q^h_{2}),\dots, (\ell_k,q^h_{k})\right)
\end{align}
for $(V,q)$, each pair $(\ell_i,q_i^h)$ is associated to 
a sesquilinear Young tableau
$\left([\ell_i],\{(E,q_i^h)\}\right).$
Since the pair $(\ell_i,q_i^h)$ is admissible, we deduce that the space $(E\otimes_F F^{\ell_i},q_i^h\otimes q_{\ell_i})$ must be $\epsilon$-Hermitian. Since $\Fs_{\udl{\Fl}}$ is decreasing, 
we may write the partition 
\[
\udl{\Fl}=[p_1^{m_1},p_2^{m_2},\ldots,p_r^{m_r}]=\udl{p}
\] 
with $p_1>p_2>\cdots>p_r>0$. For each $j=1,2,\ldots,r$, to $p_j^{m_j}$ we assign the following $\epsilon$-Hermitian space
\[
\big(E^{m_j}\otimes_E E^{p_j},\big(\bigoplus_iq_i^h\big)\otimes q_{p_j}\big),
\]
where the summation index $i$ runs in the following set
\[
\{m_1+\cdots+m_{j-1}+1,m_1+\cdots+m_{j-1}+2,\dots, m_1+\cdots+m_{j-1}+m_j\}
\]
if $j>1$; and if $j=1$, the summation index $i$ runs in the set $\{1,2,\cdots,m_1\}$. From this construction, it is clear that the admissible decreasing pre-tableau $\Fs_{\udl{\Fl}}$ 
as in \eqref{Fs-decreasing} for the $\Fn$-dimensional $\epsilon$-Hermitian space $(V,q)$ determines a unique 
\begin{align}\label{LY}
\left(\udl{p},\{(V_{(p_j)}^h,q_{(p_j)}^h)\}\right),
\end{align}
which is an admissible sesquilinear Young tableau for the $\Fn$-dimensional $\epsilon$-Hermitian space $(V,q)$, with $\udl{p}=\udl{\Fl}$ and 
\[
(V_{(p_j)}^h,q_{(p_j)}^h)=\big(E^{m_j},\bigoplus_iq_i^h\big)
\]
as above. 

Recall from Definition \ref{defn:AYT} that $\CY(V,q)$ consists of the equivalence classes of 
admissible sesquilinear Young tableaux for $(V, q_V)$. 
We denote by $\CY_a(\varphi,\chi,q)$ the subset of $\CY(V,q)$ consisting of the 
admissible sesquilinear Young tableaux as displayed in \eqref{LY} 
for the $\Fn$-dimensional $\epsilon$-Hermitian space $(V,q)$ that are constructed from the set of the $L$-descent pre-tableaux $\CL_a(\varphi,\chi,q)$ associated with the given triple $(\varphi,\chi,q)\in \FT_{a}(G_n^*)$.
By Proposition \ref{prop:GZ}, the set $\CN_F(\Fg_n)_\circ$ of all $F$-rational nilpotent orbits in $\CN_F(\Fg_n)$ and the set $\CY(V,q)$ of the equivalence classes of 
admissible sesquilinear Young tableaux for $(V, q_V)$ are canonically in one-to-one correspondence.

\begin{defn}[Wavefront Set for $(\varphi,\chi)$]\label{defn:AWFS}
The wavefront set associated to a triple $(\varphi,\chi,q)\in \FT_{a}(G_n^*)$, which is denoted by
$\wf_a(\varphi,\chi,q)$, is defined to be the subset of $\CN_F(\Fg_n)_\circ$ that correspond canonically with $\CY_a(\varphi,\chi,q)$, which is constructed from the set of $L$-descent pre-tableaux $\CL_a(\varphi,\chi,q)$.
\end{defn}

In order to understand the structure of the wavefront set $\wf_a(\varphi,\chi,q)$ for each 
triple $(\varphi,\chi,q)\in \FT_{a}(G_n^*)$, it is necessary to understand the 
partitions $\udl{p}$ and the rational properties carried by the admissible sesquilinear 
Young tableau in $\CY_a(\varphi,\chi,q)$ as displayed in \eqref{LY}.
We collect all the partitions of $\Fn$ from $\CY_a(\varphi,\chi,q)$:
\[
    \udl{p}=[p_1^{m_1},p_2^{m_2},\ldots,p_r^{m_r}]
    \]
with $p_1>p_2>\cdots>p_r>0$, and denote by $\CP_a(\varphi,\chi,q)$ the subset of the partitions of $\Fn$ so obtained. Those partitions of $\Fn$ in $\CP_a(\varphi,\chi,q)$
are called the {\it $L$-descent partitions} associated with the given triple 
$(\varphi,\chi,q)\in \FT_{a}(G_n^*)$.

From the local Gan-Gross-Prasad conjecture as discussed in 
Section \ref{ssec-LGGP} and the descent of enhanced $L$-parameters in 
Definition \ref{defn:pd}, it is not hard to deduce the following proposition. 

\begin{prop}\label{prop:Lp}
For any triple $(\varphi,\chi,q)\in \FT_{a}(G_n^*)$, the $L$-descent partitions 
\[
\udl{p}=[p_1^{m_1},p_2^{m_2},\ldots,p_r^{m_r}]\in \CP_a(\varphi,\chi,q)
\]
with $p_1>p_2>\cdots>p_r>0$ have the following property.
\begin{enumerate}
    \item When $G_n$ is a special orthogonal group $\SO_\Fn$,  all the parts $p_i$ are odd integers.
    \item When $G_n$ is a symplectic group $\Sp_{2n}$ or a metaplectic group $\Mp_{2n}$, 
    all the parts $p_i$ are even integers. 
\end{enumerate}
\end{prop}

In particular,  there is a one-to-one correspondence $\CL_a(\varphi, \chi,q)\to \CY_a(\varphi,\chi,q)$ which sends an $L$-descent pre-tableau as in \eqref{Fs-decreasing}
to an admissible sesquilinear Young tableau as in \eqref{LY}. By definition, it is clear that the $L$-descent partitions in $\CP_a(\varphi,\chi,q)$ 
do not depend on the choice of the Whittaker data $a\in\CZ$ for $G_n^*$. For reference, we still keep it in the notation. 

In Section \ref{ssec-RNO}, we have explicitly constructed the canonical correspondence between 
the set $\CN_F(\Fg_n)_\circ^\st$ of the $F$-stable nilpotent orbits of $\Fg_n$ and the set 
$\CP(V,q)$ of all partitions of $\Fn$ with constraints described there. It is clear that 
the dominating order on partitions can be transferred to the topological order on the set 
$\CN_F(\Fg_n)_\circ^\st$. We may define the  {\it $F$-rational topological order} and the
{\it $F$-stable topological order} on the set $\CN_F(\Fg_n)_\circ$ as follows. 

\begin{defn} \label{defn:order}
For an $F$-rational orbit $\CO \in \CN_F(\Fg_n)_\circ$, denote by $\CO^\st\in \CN_F(\Fg_n)_\circ^\st$ the corresponding $F$-stable orbit. For two $F$-rational orbits 
$\CO_1, \CO_2 \in \CN_F(\Fg_n)_\circ$, we say that 
\begin{enumerate}
\item 
$\CO_1\leq_F \CO_2$ under the $F$-rational topological order if $\CO_1$ is contained in the Hausdorff closure of $\CO_2$.

\item 
$\CO_1 \leq_\st \CO_2$ under the $F$-stable topological order if $\CO_1^\st$ is contained in the Zariski closure of $\CO_2^\st$.

\end{enumerate}
\end{defn}

Note that in the above, if $\CO_1\leq_F\CO_2$, then $\CO_1\leq_\st \CO_2$. We denote by $\wf_a(\varphi,\chi,q)^\mx$ the subset of $\wf_a(\varphi,\chi,q)$ 
consisting of all the maximal members under the $F$-stable topological order. 

\subsection{Arithmetic wavefront set}\label{ssec-AWF}

As before, $G_n^*$ is an $F$-quasisplit classical group, which is defined as $G_n^*=\Isom(V,q)^\circ$ or $\Mp(V)$ with an $\Fn$-dimensional non-degenerate $\eps$-Hermitian vector space $(V,q)$.
By the local Langlands correspondence with the rational structure, as explained in Section \ref{ssec-RLLC}, for any generic local $L$-parameter $\varphi$ of $G_n^*$ and $a\in\CZ$, we can write any member
$\pi$ in the local Vogan packet $\Pi_\varphi[G_n^*]$ as
\[
\pi=\pi_{a}(\varphi,\chi)=\pi_{a}(\varphi,\chi,q)
\]
for a unique $\chi\in\wh{\CS_\varphi}$. We define the {\bf arithmetic wavefront set} of $\pi$ by
\begin{align}\label{ariwfs}
\wf_\ari(\pi):=\wf_a(\varphi,\chi,q)
\end{align}
where  the wavefront set $\wf_a(\varphi,\chi,q)$ of the given triple $(\varphi,\chi,q)\in\FT_{a}(G^*_n)$ is defined in Definition \ref{defn:AWFS}.

\begin{thm}\label{inv-pi}
Let $F$ be a local field of characteristic zero.
Assume that $G_n^*$ is an $F$-quasisplit classical group, defined as $G_n^*=\Isom(V,q)^\circ$ for an $\Fn$-dimensional non-degenerate $\eps$-Hermitian vector space $(V,q)$, or
the metaplectic double cover $\Mp_{2n}(F)$ of the symplectic group $\Sp_{2n}(F)$. Let $\varphi$ be a generic local $L$-parameter of $G_n^*$ as given in Section {\rm \ref{ssec-LVP}}.
For any $\pi$ belonging to the local Vogan packet $\Pi_\varphi[G_n^*]$, the arithmetic wavefront set $\wf_\ari(\pi)$ is an invariant of $\pi$.
\end{thm}

\begin{proof}
We need to show that $\wf_a(\varphi,\chi,q)$ only depends on $\pi$, and is independent of the choice of the Whittaker datum $a$ for $G_n^*$.
Let
$
b = a \cdot z',
$
where $z'\in\CZ$.
By the properties of the local Langlands correspondence given in Proposition \ref{prop:LLC},
\[
\pi = \begin{cases} \pi_{b}(\varphi, \chi\cdot \eta_{z'}), & \textrm{if }G_n^*\neq \Mp_{2n}, \\
\pi_{b}(\varphi(z'), \chi\cdot \eta_{z'}), & \textrm{otherwise}.
\end{cases}
\]
By Proposition \ref{inv-prop} below,  $\CT_a(\varphi, \chi, q)$ is an invariant of $\pi$. 
Hence its subset $\CL_a(\varphi,\chi, q)$ and the corresponding $\CY_a(\varphi,\chi, q)$ 
are invariants of $\pi$, and so is 
the wavefront set $\wf_a(\varphi,\chi,q)$. 
\end{proof}

\begin{prop} \label{inv-prop}
 Let $(\varphi,\chi,q)\in \FT_{a}(G_n^*)$, and let  $b=a\cdot z'$ for some  $z'\in\CZ$. Then the following hold.
\begin{enumerate}
\item If $G_n^*\neq \Mp_{2n}$, then
\begin{equation}  \label{inv}
\CT_b(\varphi, \chi\cdot\eta_{z'}, q) = \CT_a(\varphi, \chi, q).
\end{equation}
\item If $G_n^*=\Mp_{2n}$, then
\begin{equation} \label{inv-mp}
\CT_b(\varphi(z'), \chi\cdot\eta_{z'}, q) = \CT_a(\varphi, \chi, q).
\end{equation}
\end{enumerate}
 \end{prop}

 \begin{proof}
We prove the proposition by induction on $\Fn$, and the induction process interlocks different types of classical groups.

{\bf Case 1.} $G_n^*=\SO_\Fn^*$. If we start from the case that $\Fn$ is odd, then by induction hypothesis we are reduced to the case of an even special orthogonal group, and vice versa. 
Let us assume that $G_n^* = \SO_{2n}^*$ and give the proof in this case, and the proof for $G_n^*=\SO_{2n+1}^*$ is similar. 

In this case we have that $H_{\lfloor(\Fn-\Fl)/2\rfloor}^* = \SO^*_{2n-\Fl}$, where $\Fl$ is odd. Recall from \eqref{eta} that for any $\phi\in\Phi_\gen(\SO_{2n-\Fl}^*)$ and $z\in\CZ$, the character $\eta_z\in \wh{\CS_\phi}$ is trivial. By Definition \ref{defn:pd} and direct calculations, we see that
\[
 \FD_\Fl^{zz'}(\varphi, \chi\cdot \eta_{z'}) = \FD^z_\Fl(\varphi, \chi).
\]
We find from Definition \ref{LYT} that above two local descents give the same quadratic form $q_{(\Fl)}^h$:
\[
b\cdot zz' = a \cdot z = \disc(q_{(\Fl)}).
\]
By  the induction hypothesis on odd special orthogonal groups, we have that
\begin{equation} \label{Ya-inv}
\bigcup_{(\wh\phi, \wh{\chi'})\in  \FD_\Fl^{zz'}(\varphi, \chi\cdot \eta_{z'}) } \CT_{b}(\wh\phi,\wh{\chi'}, q') = \bigcup_{(\wh\phi, \wh{\chi'})\in  \FD^z_\Fl(\varphi, \chi)} \CT_a(\wh\phi, \wh{\chi'} ,q').
\end{equation}
Composing with the tableaux $(\Fl, q_{(\Fl)}^h)$ and taking the union over $\Fl$ and $z\in \CZ$  give \eqref{inv}.

{\bf Case 2.} Assume that $G_n^*=\RU_{\Fn}^*$. Then $H^*_{\lfloor n-\Fl/2\rfloor}=\RU_{\Fn-\Fl}^*$.  
In this case we have that
\begin{equation}\label{2b}
\FD_\Fl^{zz'}(\varphi, \chi\cdot \eta_{z'}) = \FD^z_\Fl(\varphi, \chi)\cdot \eta_{z'},
\end{equation}
where the right hand side is defined as
\[
 \FD^z_\Fl(\varphi, \chi)\cdot \eta_{z'}:=\left\{ (\wh\phi, \wh{\chi'}\cdot \eta_{z'})\ \mid\  (\wh{\phi}, \wh{\chi'})\in \FD^z_\Fl(\varphi, \chi)\right\}.
\]
By Definition \ref{LYT}, the two descents in \eqref{2b} give the same Hermitian form $q^h_{(\Fl)}$:
\[
b\cdot zz' = a \cdot z =  [ \disc(q^h_{(\Fl)})].
\]
By the induction hypothesis, for any $(\wh\phi, \wh{\chi'})\in \FD^z_\Fl(\varphi, \chi)$ we have that
\begin{equation} \label{even-ind}
\CT_{b}(\wh{\phi},\wh{\chi'}\cdot \eta_{z'} , q') = \CT_a(\wh\phi, \wh{\chi'}, q').
\end{equation}
Then \eqref{2b} and \eqref{even-ind} show that equation \eqref{Ya-inv} holds in this case, which implies \eqref{inv}.

{\bf Case 3.} $G_n^* = \Sp_{2n}$ or $\Mp_{2n}$. We  first recall from the proof of \cite[Proposition 18.1]{GGP12} that, for odd orthogonal and symplectic $L$-parameters $\varphi$ and $\phi$ respectively, it holds that
\begin{equation} \label{spmpchi}
\chi_{\varphi(z'),\phi}  =  \chi_{\varphi, \phi(z')},\quad \chi_{\phi, \varphi(z')}\cdot \eta_{z'} = \chi_{\phi(z'), \varphi},
\end{equation}
where $z'\in \CZ$.

{\bf Case 3.1.} $G_n^*=\Sp_{2n}$. Let $(\varphi, \chi, q) \in \FT_{a}(\Sp_{2n})$.  For $z\in \CZ$, using \eqref{spmpchi} we check from  \eqref{spmpchi} and Definition \ref{LYT} that
\begin{equation} \label{sp-mp}
\FD^{zz'}_\Fl(\varphi, \chi\cdot \eta_{z'}) = \{(\wh{\phi}(z'), \wh{\chi'}\cdot\eta_{z'})\ \mid\  (\wh\phi, \wh{\chi'}) \in \FD^z_\Fl(\varphi, \chi)\}.
\end{equation}
Note that $b \cdot zz' = a \cdot z$, which give the same  $\disc(q^{h}_{(\Fl)})$. By the induction hypothesis \eqref{inv-mp} on $\Mp_{2n-\Fl}$,
\[
\CT_{b}(\wh{\phi}(z'), \wh{\chi'}\cdot\eta_{z'}, q') = \CT_a(\wh{\phi}, \wh{\chi'}, q')
\]
for any $(\wh\phi, \wh{\chi'})\in \FD^{z}_\Fl(\varphi, \chi)$. This reduces to the case of metaplectic groups.

{\bf Case 3.2.} $G_n^*=\Mp_{2n}$. Let $(\varphi,\chi, q)\in \FT_{a}(\Mp_{2n})$.  Switching the notation $\varphi$ and $\phi$ in \eqref{spmpchi}, we check that
\begin{equation} \label{mp-sp}
\FD^{zz'}_\Fl(\varphi(z'),\chi\cdot \eta_{z'},q) = \{ (\wh\phi, \wh{\chi'}\cdot \eta_{z'})\ 
\mid\ (\wh\phi, \wh{\chi'})\in \FD^z_\Fl(\varphi, \chi)\}.
\end{equation}
Again, $b \cdot zz' = a \cdot z$ gives the same discriminant $\disc(q^h_{(\Fl)})$. By the induction hypothesis \eqref{inv} on $\Sp_{2n-\Fl}$,
\[
\CT_{b}(\wh{\phi}, \wh{\chi'}\cdot \eta_{z'}, q') = \CT_a(\wh\phi, \wh{\chi'}, q')
\]
for any $(\wh\phi, \wh{\chi'})\in \FD^z_\Fl(\varphi, \chi)$.  This reduces to the case of symplectic groups.

We combine {\bf Case 3.1} and {\bf Case 3.2} together and finish the induction process.
\end{proof}

By Theorem \ref{inv-pi}, for any $\pi\in\Pi_\varphi[G_n^*]$, if $\pi=\pi_a(\varphi,\chi,q)\in\Pi_\varphi(G_n)$ for some $\chi\in\wh{\CS_\varphi}$ and some $F$-pure inner form $G_n=\Isom(V,q)^\circ$ 
of $G_n^*$, then we have
\begin{equation}\label{pi-inv}
\wf_\ari(\pi)^\mx  =\wf_a(\varphi,\chi,q)^\mx.
\end{equation} 

\subsection{Wavefront set and $\CZ$-orbit}\label{ssec-WFS-ZOrbits}
In order to understand refined structures of the arithmetic wavefront sets $\wf_\ari(\pi)$ when $\pi$ runs in the local Vogan packet $\Pi_F[G_n^*]$, we are going to define the wavefront set associated to a $\CZ$-orbit as follows. 

\begin{defn}[Wavefront Set for $(\varphi,\ScO_\chi)$] \label{WSO}
Let $\varphi\in \widetilde{\Phi}_\gen(G_n^*)$ and let $\ScO_\chi\subset \wh{\CS_\varphi}$ be a pointed $\CZ$-orbit, that is, a $\CZ$-orbit $\ScO$ with base point $\chi\in\ScO$. The wavefront set associated to $(\varphi,\ScO_\chi)$ and a  Whittaker datum $a\in\CZ$ for $G_n^*$, which is denoted by $\wf_a(\varphi,\ScO_\chi)$, is defined to be
\[
\wf_a(\varphi,\ScO_\chi) = 
\begin{cases} 
\bigcup_{\chi'\in\ScO_\chi}\wf_a(\varphi,\chi', q'), & \textrm{if }G_n^*\neq\Mp_{2n},\\
\bigcup_{z'\in\CZ}\wf_a(\varphi(z'), \chi\cdot\eta_{z'}, q), & \textrm{if }G_n^*=\Mp_{2n},
\end{cases}
\]
where $q'$ is uniquely determined by $(\varphi,\chi')$ such that $(\varphi,\chi',q')\in \FT_a(G^*_n)$.
\end{defn}

Note that the base point  plays a role only for $G_n^*=\Mp_{2n}$.
In view of Definition \ref{WSO}, the union 
\[
\CT_a(\varphi,\ScO_\chi) := 
\begin{cases}
  \bigcup_{\chi'\in\ScO_\chi}\CT_a(\varphi,\chi', q'), & \textrm{if }G_n^*\neq\Mp_{2n},\\
\bigcup_{z'\in\CZ}\CT_a(\varphi(z'), \chi\cdot\eta_{z'}, q), & \textrm{if }G_n^*=\Mp_{2n},
\end{cases}
\]
is also an interesting object. Let $\FP_a(\varphi, \chi, q)$ be the set of ordered partitions associated to the set of pre-tableaux $\CT_a(\varphi, \chi, q)$.  Then we have the following invariance property for $\FP_a(\varphi, \chi, q)$.

\begin{prop} \label{inv-O}
The set of partitions $\FP_a(\varphi, \chi, q)$ satisfies the following properties.
 \begin{enumerate}
 \item If $G_n^*\neq \Mp_{2n}$, then $\FP_a(\varphi, \chi, q)$ only depends on the $\CZ$-orbit $\ScO_\chi\subset \wh{\CS_\varphi}$ of $\chi$, and does not depend on $a$.
 \item If $G_n^* =\Mp_{2n}$, then for $z'\in \CZ$ it holds that
 \begin{equation} \label{Mp-inv}
 \FP_a(\varphi, \chi,q ) = \FP_a(\varphi(z'), \chi\cdot \eta_{z'}, q).
 \end{equation}
 \end{enumerate}
\end{prop}

\begin{proof}
For $z'\in \CZ$, write  
$q_{z'} =z'\cdot q$ for short.  As before, let $b = a\cdot z'$.  To prove Part (1),  we need to show that
\begin{equation} \label{P-inv}
\FP_a(\varphi, \chi, q) = \FP_a(\varphi, \chi\cdot\eta_{z'}, q_{z'}),
\end{equation}
and that it does not depend on $a$, namely
\begin{equation} \label{a-inv}
\FP_{a}(\varphi, \chi, q) = \FP_{b}(\varphi, \chi, q).
\end{equation}
Moreover we need to prove \eqref{Mp-inv} for $G_n^*=\Mp_{2n}$, which gives Part (2). 

The proof of \eqref{Mp-inv}--\eqref{a-inv} in general cases  is again by induction on $\Fn$.

{\bf Case 1.}
$G_n^*=\SO_{\Fn}^*$.  If $\Fn$ is odd,  then the $\CZ$-action on $\wh{\CS}_\varphi$ is trivial. In this case \eqref{P-inv} holds trivially, and \eqref{a-inv} follows from Part (1) of 
Proposition \ref{inv-prop}.

Assume that $G_n^*= \SO_{2n}^*$, and that
\begin{equation}  \label{PO}
(\Fl, q^h_{(\Fl)})\star \CT_{a}(\wh\phi, \wh{\chi'}, q')\subset \CT_a(\varphi,\chi, q),
\end{equation}
where
$(\wh\phi, \wh{\chi'})\in \FD^z_\Fl(\varphi,\chi)$, that is,
\[
\chi_{\varphi, \phi}=\chi\cdot\eta_z,\quad \chi_{\phi, \varphi} = \chi'.
\]
Then we have $(\wh\phi, \wh{\chi'})\in \FD^{zz'}(\varphi, \chi\cdot\eta_{z'})$, and we check from Definition \ref{LYT} that
\begin{equation} \label{PO'}
(\Fl, z'\cdot q^h_{(\Fl)})\star \CT_{a}(\wh\phi, \wh{\chi'}, q'_{z'})\subset \CT_a(\varphi,\chi\cdot\eta_{z'}, q_{z'}).
\end{equation}
Clearly \eqref{PO} and \eqref{PO'} give the same ordered partitions. This proves \eqref{P-inv}.

By Part (1) of Proposition \ref{inv-prop}, we have that 
\[
\CT_{b}(\varphi, \chi, q_{z'}) = \CT_a(\varphi, \chi\cdot\eta_{z'}, q_{z'}).
\]
Combining with \eqref{P-inv} we obtain that
\[
\FP_{b}(\varphi, \chi, q) = \FP_{b}(\varphi, \chi, q_{z'}) = \FP_a(\varphi, \chi, q),
\]
which proves \eqref{a-inv}.

{\bf Case 2.} $G_n^* = \RU_\Fn^*$.  Assume that 
\begin{equation}  \label{PU}
(\Fl, q^h_{(\Fl)})\star \CT_{a}(\wh\phi, \wh{\chi'}, q')\subset \CT_a(\varphi,\chi, q),
\end{equation}
where $(\wh\phi, \wh{\chi'})\in \FD^z_\Fl(\varphi,\chi)$. Then $\chi_{\varphi, \phi} =\chi\cdot\eta_z$, $\chi_{\phi,\varphi} = \chi'\cdot\eta_{(-1)^\Fl z}$, which gives that 
\[
(\wh\phi, \wh{\chi'\cdot\eta_{z'}}) \in \FD^{zz'}(\varphi, \chi\cdot\eta_{z'}).
\]
Since $a \cdot z = [\disc(q^h_{(\Fl)})]$,  we have that 
\begin{equation} \label{PU'}
(\Fl, z'\cdot q^h_{(\Fl)})\star \CT_{a}(\wh\phi, \wh{\chi'\cdot\eta_{z'}}, q')\subset \CT_a(\varphi,\chi\cdot\eta_{z'}, q_{z'}).
\end{equation}
Then \eqref{PU} and \eqref{PU'} give the same ordered partitions by the induction hypothesis.
The independence of $a$  follows from  \eqref{P-inv} and Part (1) of Proposition \ref{inv-prop},  by similar arguments as used in {\bf Case 1} for even special orthogonal groups.

Finally, we prove the rest of the proposition for symplectic and metaplectic groups together.

{\bf Case 3.1.} $G_n^*=\Sp_{2n}$. Assume that
\[
(\Fl, q^h_{(\Fl)})\star \CT_{a}(\wh\phi, \wh{\chi'}, q') \subset \CT_a(\varphi,\chi,q),
\]
where $(\wh\phi, \wh{\chi'})\in \FD^z_\Fl(\varphi,\chi)$.  By \eqref{sp-mp}, we have that 
\[
(\Fl, z'\cdot q^h_{(\Fl)})\star \CT_{a}(\wh\phi(z'), \wh{\chi'}\cdot \eta_{z'}, q') \subset \CT_a(\varphi, \chi\cdot\eta_{z'}, q).
\]
By the induction hypothesis \eqref{Mp-inv} on $\Mp_{2n-\Fl}$, we have that 
\[
\FP_{a}(\wh\phi(z'), \wh{\chi'}\cdot \eta_{z'}, q') = \FP_{a}(\wh\phi, \wh{\chi'}, q').
\]
This reduces to the case of metaplectic groups.

{\bf Case 3.2.} $G_n^*=\Mp_{2n}$. Assume that
\[
(\Fl, q^h_{(\Fl)})\star \CT_{a}(\wh\phi, \wh{\chi'}, q') \subset \CT_a(\varphi,\chi,q),
\]
where $(\wh\phi, \wh{\chi'})\in \FD^z_\Fl(\varphi,\chi)$.  By \eqref{mp-sp}, we have that 
\[
(\Fl, z'\cdot q^h_{(\Fl)})\star \CT_{a}(\wh\phi, \wh{\chi'}\cdot\eta_{z'}, q') \subset \CT_a(\varphi(z'), \chi\cdot \eta_{z'}, q).
\]
By the induction hypothesis \eqref{P-inv} on $\Sp_{2n-\Fl}$, we have that 
\[
\FP_{a}(\wh\phi, \wh{\chi'}\cdot\eta_{z'}, q') = \FP_{a}(\wh\phi, \wh{\chi'}, q').
\]
This reduces to the case of symplectic groups.

The above reductions in {\bf Case 3.1} and {\bf Case 3.2} finish the proof of \eqref{P-inv} and \eqref{Mp-inv} for $G_n^*=\Sp_{2n}$ or $\Mp_{2n}$. The independence of $a$ for $\Sp_{2n}$ follows from \eqref{P-inv}
and Part (1) of Proposition \ref{inv-prop}. Similarly, the independence of $a$ for $\Mp_{2n}$ follows from \eqref{Mp-inv} and Part (2) of Proposition \ref{inv-prop}. We are done. 
\end{proof}

\subsection{Wavefront set conjecture}\label{ssec-MC}

We are ready to state our conjecture on the wavefront sets defined from three different 
perspectives. 

The notion of the {\bf algebraic wavefront set} $\wf_\wm(\pi)$ of $\pi\in\Pi_F(G)$ for 
general connected reductive group $G$ defined over $F$ can be defined by means of 
the generalized Whittaker model ${\rm {Wh}}_X(\pi)$ (see \eqref{whit-model}) associated to any nilpotent element $X$ in $\CN_F(\Fg)$, as defined in Section \ref{ssec-TJM}. More precisely, we define 
\begin{equation}\label{wf-wm}
\wf_\wm(\pi)={\{\CO \in \CN_F(\Fg_n)_\circ\ \mid\ {\rm {Wh}}_X(\pi)\ne 0, \  X\in \CO\}}.		\end{equation}

\begin{conj}[Wavefront Set]\label{conj:main}
Let $F$ be a local field of characteristic zero.
Assume that $G_n^*$ is an $F$-quasisplit classical group, defined as $G_n^*=\Isom(V,q)^\circ$ for an $\Fn$-dimensional non-degenerate $\eps$-Hermitian vector space $(V,q)$, or
the metaplectic double cover $\Mp_{2n}(F)$ of the symplectic group $\Sp_{2n}(F)$. Let $\varphi$ be a generic local $L$-parameter of $G_n^*$ as given in Section \ref{ssec-LVP}.
For any $\pi$ belonging to the local Vogan packet $\Pi_\varphi[G_n^*]$, the following identities hold:
\[
\wf_\wm(\pi)^\mx=\wf_\ari(\pi)^\mx=\wf_\tr(\pi)^\mx,
\]
where the algebraic wavefront set $\wf_\wm(\pi)$ is defined in \eqref{wf-wm}, the analytic wavefront set $\wf_\tr(\pi)$ is defined as in \eqref{wf-tr},
and $\wf_\square(\pi)^\mx$ denotes the subset of all maximal members in the wavefront set $\wf_\square(\pi)$ under the $F$-stable topological order on $\CN_F(\Fg_n)_\circ$.
\end{conj}

As first evidence, the Wavefront Set Conjecture (Conjecture \ref{conj:main}) holds in the case that the algebraic wavefront set contains regular or subregular nilpotent orbits. The regular case amounts to the case of generic representations, which is essentially known by the properties of the local Langlands correspondence. We will present the proof in the subregular case for $p$-adic odd special orthogonal groups. 

\begin{prop}\label{exmp:generic}
Assume that $\pi$ is an irreducible generic representation of an  $F$-quasisplit $G^*_n$. Then Conjecture \ref{conj:main} holds for $\pi$.
\end{prop}

\begin{proof}
Since $\pi$ is generic, there exists $a\in \CZ$ such that $\pi  = \pi_a(\varphi,{\mathds 1})$ under the local Langlands correspondence $\iota_a$.
Recall the $\CZ$-action on $\widehat{\CS_\varphi}$. By Proposition \ref{prop:LLC}, the Whittaker data with respect to which $\pi$ is generic are parametrized by 
$a\cdot{\rm Stab}_{\CZ}({\mathds 1})$, where 
\[
{\rm Stab}_{\CZ}({\mathds 1}) = \{z\in\CZ \mid \eta_z = {\mathds 1}\}.
\]
Thus by Proposition \ref{inv-prop} and \eqref{pi-inv}, $\wf_\ari(\pi)^\mx=\wf_\wm(\pi)^\mx$  consists of the regular nilpotent orbits corresponding to $a\cdot {\rm Stab}_{\CZ}({\mathds 1})$. 
In addition, from \cite{Mat92, MW87}, one has $\wf_\wm(\pi)^\mx=\wf_\tr(\pi)^\mx$.
Therefore, Conjecture \ref{conj:main} holds for generic representations, i.e.,
if $\pi$ is generic, then
\[
\wf_\ari(\pi)^\mx=\wf_\wm(\pi)^\mx=\wf_\tr(\pi)^\mx.
\]
\end{proof}

\begin{prop}\label{prop:subr}
Let $G_n$ be an odd special orthogonal group over a $p$-adic field $F$, and let $\pi$ be an irreducible representation of $G_n$ with generic $L$-parameter. Assume that $\wf_\wm(\pi)^\mx$ contains an $F$-rational subregular nilpotent orbit $([\Fn-2,1^{2}],\{(Fe_{-1},q'_{e,\varsigma}),(V_{0}\cap e^{\perp},q_V\vert_{V_{0}\cap e^{\perp}})\})$ for some anisotropic vector $e$ $($cf. \eqref{eq:X-p-1}$)$. Then Conjecture \ref{conj:main} holds for $\pi$.
\end{prop}

\begin{proof}   By assumption, $\pi$ has a generalized Whittaker model associated to the partition $[\Fn-2,1^2]$ but $\pi$ is not generic by Proposition \ref{exmp:generic}. It follows that the subregular stable nilpotent orbit $\CO^{\rm st}_{\Fn-2}$ is defined over $F$, i.e., $\CO^{\rm st}_{\Fn-2}$ contains $F$-rational subregular orbits. Since $\pi$ is not generic, every member of $\wf_\wm(\pi)^\mx$ is contained in $\CO^{\rm st}_{\Fn-2}$, and $M_X^\circ$ is isomorphic to $\SO_2$, where $X=X_{e,\varsigma}$ in the notation of Section \ref{ssec-RNO-D}.

Let $\varphi$ be the $L$-parameter of $\pi$. By \cite[Lemma 3.1]{JZ18}, there exists a character $\tau$ of $M_X^\circ$ such that
\[
\Hom_{M_{X}^\circ\ltimes U_X}(V_\pi, \tau\otimes\psi_{X_{e,\varsigma}})\ne 0.
\]
Let $\phi$ be the $L$-parameter of $\tau$.
By the local Gan-Gross-Prasad conjecture as discussed in Section \ref{sec:GGP-SO}, we may take the local Langlands correspondence $\iota_a$  such that  $(\phi,\chi')$   belongs to $\FD^{1}_{\Fn-2}(\varphi,\chi)$ where $\pi=\pi(\varphi,\chi)$ and $\tau=\pi_a(\phi,\chi')$.
By Definition \ref{defn:AWFS}, we obtain that
\[
([\Fn-2,1^2],\{(Fe_{-1},q'_{e,\varsigma}),(V_0\cap e^\perp,q\vert_{V_0\cap e^\perp})\})\in \wf_a(\varphi,\chi,q).
\]
It follow that $\wf_\wm(\pi)^\mx\subset \wf_\ari(\pi)^\mx$.
Since every member of  $\wf_\ari(\pi)^\mx$ is contained in $\CO^{\rm st}_{\Fn-2}$ as well, the local Gan-Gross-Prasad conjecture is applicable for all $F$-rational nilpotent orbits in $\wf_\ari(\pi)^\mx$.
Thus $\wf_\ari(\pi)^\mx\subset \wf_\wm(\pi)^\mx$ and we conclude that
\[
\wf_\ari(\pi)^\mx=\wf_\wm(\pi)^\mx.
\]
This finishes the proof. We remark that since $\pi$ is not generic, it can be shown that $M_X^\circ$ is non-split, i.e., $M_X^\circ\not\cong F^\times$.
\end{proof}

\section{Arithmetic Data: Archimedean Case}\label{sec-ADA}


In order to determine the explicit structures of the arithmetic wavefront sets when $F=\BR$, we are going to write down the arithmetic data explicitly for $F=\BR$, which are 
more precise than those for general local fields of characteristic zero.

\subsection{Weil group representation and local root number}\label{sec-LRN}
We first collect some standard facts on Weil groups and their representations in the archimedean case.
Let $|\cdot|$ be the usual absolute value on $\mathbb{C}$, and $|\cdot |_\mathbb{C}=|\cdot |^2$. Denote $[z]=z/|z|$, $z\in \BC^\times$. We have the Weil groups $\CW_\mathbb{C}=\mathbb{C}^\times$ and $\CW_\mathbb{R}=\mathbb{C}^\times\cup\mathbb{C}^\times j$, where
$j^2=-1$, $jzj^{-1}=\bar{z}$,  and $z\in\mathbb{C}^\times.$ 
 The local class field theory in this case yields
\begin{align}\label{cft}
\CW_\mathbb{R}^\textrm{ab}\stackrel{\approx}{\longrightarrow}\mathbb{R}^\times,\quad z\mapsto z\bar{z}=|z|_\mathbb{C}, \quad j\mapsto -1.
\end{align}
The  representations of Weil groups can be listed as follows.
\begin{itemize}
\item Characters $\varsigma_{t, \alpha}$ of $\CW_\mathbb{C}$ with $t\in \BC$ and $\alpha\in \BZ$  defined by\footnote{The character $\chi_{2\alpha}: z\mapsto (z/\bar{z})^\alpha$ of $\BC^\times$ with $\alpha\in\frac{1}{2}\BZ$ defined in \cite{At20}, is the character $\varsigma_{0, -2\alpha}$ in our notation; we put the minus sign for technical reasons which will be clear later.}
\[
\varsigma_{t, \alpha}(z)= |z|_\mathbb{C}^t[z]^{-\alpha};
\]
\item Characters $\sigma_{t, \pm}$ of $\CW_\mathbb{R}$ with $t\in\mathbb{C}$  defined by
\[
\sigma_{t, \pm}(z)=|z|^{2t}=|z|^t_\mathbb{C},\quad \sigma_{t,\pm}(j)=\pm1,
\]
which factor through $\mathbb{R}^\times$ via the isomorphism in \eqref{cft};
\item Two-dimensional representations $\sigma_{t, \alpha}$ of $\CW_\mathbb{R}$ with $t\in\mathbb{C}$ and $\alpha\in \mathbb{Z}$  defined by
\[
\sigma_{t, \alpha}(re^{i\theta})=\begin{pmatrix} r^{2t}e^{i\alpha\theta} & \\ & r^{2t}e^{-i\alpha\theta} \end{pmatrix} \quad {\rm and} \quad 
\sigma_{t, \alpha}(j)=\begin{pmatrix} & (-1)^\alpha \\ 1 & \end{pmatrix}.
\]
Equivalently, one has that
\[
\sigma_{t, \alpha}=\textrm{Ind}^{\CW_\mathbb{R}}_{\CW_\mathbb{C}} (\varsigma_{t, \alpha}),
\]
where $\varsigma_{t,\alpha}$ is the character of $\CW_\BC$ given in the above. It is clear that
\[
\sigma_{t,\alpha}|_{\CW_\mathbb{C}}\cong \varsigma_{t,\alpha}\oplus \varsigma_{t,-\alpha}.
\]
\end{itemize}
The following are well-known:
\begin{itemize}
\item $\sigma_{t,\alpha}\cong \sigma_{t,-\alpha}$;
\item $\sigma_{t,\alpha}$ is irreducible if and only if $\alpha\neq 0$;
\item $\sigma_{t,0}\cong \sigma_{t,+}\oplus \sigma_{t,-}$.
\end{itemize}
It follows that $\sigma_{t,\pm}$ and $\sigma_{t, \alpha}$ with $t\in\mathbb{C}$ and $\alpha>0$ are all the irreducible continuous representations of $\CW_\mathbb{R}$.
Some more useful facts are recalled below:
\begin{itemize}
\item {\it Conjugate contragredient}:
\[
{}^c \varsigma_{t,\alpha}^\vee=\varsigma_{-t,\alpha},\quad  \sigma_{t,\pm}^\vee\cong \sigma_{-t,\pm},\quad
  \sigma_{t,\alpha}^\vee \cong \sigma_{-t,\alpha};
\]
\item {\it Determinant formula} of $\sigma_{t,\alpha}$,
\[
\det \sigma_{t,\alpha}=\sigma_{2t,\varepsilon(\alpha+1)},  
\]
where $\varepsilon(\alpha+1)$ is the sign of $(-1)^{\alpha+1}$. In particular, $\det\sigma_{0,\alpha}={\mathds 1}$ or sgn for $\alpha$ odd or even, with $\mathds 1$
and sgn the trivial and sign characters of $\BR^\times$ respectively;
\item {\it Tensor product decomposition} of Weil group representations
\begin{align*}
 \varsigma_{t_1,\alpha_1}\otimes \varsigma_{t_2,\alpha_2}&\cong \varsigma_{t_1+t_2,\alpha_1+\alpha_2},\\
\sigma_{t_1, \alpha}\otimes \sigma_{t_2,\pm} &\cong \sigma_{t_1+t_2,\alpha},\\ 
\sigma_{t_1, \alpha_1}\otimes \sigma_{t_2, \alpha_2}&\cong \sigma_{t_1+t_2, \alpha_1+\alpha_2} \oplus \sigma_{t_1+t_2, \alpha_1-\alpha_2}.
\end{align*}
\end{itemize}

Fix the nontrivial additive character $\psi^\BR(x)=e^{2\pi ix}$ of $\BR$, so that\footnote{Our $\psi^\BC$ is a twist of the one chosen in \cite{At20} by $-1$, so that 
our $\psi^\BC_2$ appears as $\psi^\BC_{-2}$ in {\it loc. cit}. In particular
the base point for the local Langlands correspondence for even unitary groups in this paper is different from 
that of {\it loc. cit.}}
\[
\psi^\BC_2(z)=\psi^\BR(\textrm{Tr}_{\BC/\BR}(iz))=\exp(2\pi(\bar{z}-z)).
\]
For a representation $\varphi$ of $\CW_E$ where $E=\BR$ or $\BC$,  let $\varepsilon(s,\varphi,\psi^E)$ be the $\varepsilon$-factor of $\varphi$ with respect to
$\psi^E$. The central value $\varepsilon(1/2,\varphi,\psi^E)$ is the local root number of $\varphi$ with respect to $\psi^E$, and will be abbreviated as $\varepsilon(\varphi,\psi_E)$.
The relevant local root numbers enjoy the following properties:
\begin{itemize}
\item $\varepsilon(\varphi_1\oplus\varphi_2,\psi^E)=\varepsilon(\varphi_1,\psi^E)\cdot \varepsilon(\varphi_2,\psi^E)$;
\item $\varepsilon(\xi\oplus{}^c\xi^{-1},\psi^\BC_2)=1$ for any character $\xi$ of $\CW_\BC$;
\item  $\varepsilon(\varsigma_{0,\alpha},\psi^\BC_2)=
\begin{cases} 
+1, & \textrm{ if }\alpha \textrm{ is even, or}\ \textrm{ if }\alpha \textrm{ is odd and}\  \alpha>0,\\
-1, & \textrm{ if }\alpha \textrm{ is odd and}\ \alpha<0;
\end{cases}$
\item $\varepsilon(\sigma_{t,+}, \psi^\BR)=1, \quad \varepsilon(\sigma_{t,-}, \psi^\BR)=-1,\quad \varepsilon(\sigma_{t,\alpha},\psi^\BR)=i^{\alpha+1}$, $\alpha>0.$
\end{itemize}

As a consequence of these properties, we have

\begin{lem}
Assume that $E=\BR$ and $\alpha, \beta$ are positive integers with $\alpha$ even and $\beta$ odd. Then
\[
\varepsilon(\sigma_{t_1,\alpha}\otimes\sigma_{t_2,\beta},\psi^\BR)=
\begin{cases} 
+1, & \textrm{if }\alpha<\beta,\\
-1, & \textrm{if }\alpha>\beta.
\end{cases}
\]
\end{lem}

\subsection{$L$-parameter and distinguished character} \label{sec-LPDC}

In this section we specify the $L$-parameters and distinguished characters of their component groups prescribed by the local Gan-Gross-Prasad conjecture for each individual family of real classical groups.
For more details of generic Vogan packets, see \cite{MW12} for the non-archimedean case and \cite{Ch23, Ch24} for the archimedean case. 

To simplify notations, denote
\begin{align}\label{0alpha}
\varsigma_\alpha=\varsigma_{0,\alpha},\quad
 \sigma_{\alpha}=\sigma_{0,\alpha},\quad
  \sigma_+=\sigma_{0,+}={\mathds 1}, \quad
  \sigma_-=\sigma_{0,-}=\sgn,
\end{align}
 where $\alpha\in \BZ$. Note that $\sigma_0=\sigma_+\oplus\sigma_-$. 

\subsubsection{Unitary groups}

An $L$-parameter in $\widetilde{\Phi}_\gen(\RU_\Fn^*)$ is conjugate self-dual hence of the form
\begin{equation}\label{upa1}
\varphi=m_1\varsigma_{\alpha_1}\oplus\cdots \oplus m_u\varsigma_{\alpha_u}\oplus \xi\oplus {}^c\xi^\vee,
\end{equation}
where $m_i$'s are positive, $\alpha_1<\cdots <\alpha_u$ are integers of the same parity,  and
$\xi$ is a representation of $W_\BC$ that has no irreducible conjugate self-dual  summands of the same type as $\varphi$, such that $\dim\varphi=\Fn$.
Note that $\varphi\in \Phi_\gen(\RU_\Fn^*)$ if $\alpha_i\equiv \Fn-1$ mod 2. We have the component group
$\CS_\varphi\cong \BZ_2^u$.

Take another $L$-parameter $\phi$ in $\widetilde{\Phi}_\gen(\RU_{\Fn'}^*)$ of a similar form
\[
\phi=m_1'\varsigma_{\beta_1}\oplus \cdots\oplus m'_v \varsigma_{\beta_v}\oplus \xi'\oplus {}^c\xi^{'\vee},
\]
where $\alpha_i$'s and $\beta_j$'s have opposite parities, so that $\varphi$ and $\phi$ have opposite types. Similarly one has the component group
$
\CS_\phi\cong \BZ_2^{v}.
$

To simplify the presentation, denote $\alpha_i$ and $\beta_j$ symbolically the corresponding basis vectors of $\CS_\varphi$ and $\CS_\phi$, respectively. 
The characters $(\chi_{\varphi,\phi}, \chi_{\phi,\varphi})\in \widehat{\CS_\varphi}\times  \wh{\CS_\phi}$ are given by
\begin{align}
\label{chi1-C} &  \chi_{\varphi,\phi}(\alpha_i)=\varepsilon(\varsigma_{\alpha_i}\otimes \phi, \psi^\BC_2)=\prod_{j: \ -\beta_j>\alpha_i}(-1)^{m'_j},\\
\label{chi2-C} &  \chi_{\phi,\varphi}(\beta_j)=\varepsilon(\varsigma_{\beta_j}\otimes\varphi, \psi^\BC_2)=\prod_{i  : \ -\alpha_i>\beta_j}(-1)^{m_i}.
\end{align}

\subsubsection{Special orthogonal groups}\label{sec:L-parameter-SO}

An equivalence class of $L$-parameters in $\Phi_\gen(\SO^*_{2n})/\sim_{\rm c}$ is of the form
\[
[\varphi]_{\rm c}=m_1\sigma_{\alpha_1}\oplus\cdots \oplus m_u \sigma_{\alpha_u} \oplus m_+\sigma_+ \oplus m_-\sigma_-\oplus \xi\oplus \xi^\vee,
\]
where $m_i$'s are positive, $0<\alpha_1<\cdots<\alpha_u$ are even, $m_++m_-$ is even, and $\xi$ is a representation of $W_\BR$ that has no irreducible orthogonal summands, such that $\det\varphi=\BC(d)$ with $d:=\disc(V)$, and $\dim\varphi=2n$. We need to consider two cases.
\begin{itemize}
\item If $m_+=m_-=0$, then we have $\CS_\varphi=\BZ_2^u$. In this case, the class $[\varphi]_{\rm c}$ contains two $L$-parameters $\{\varphi, \varphi^*\}$, and the representations $\pi_a(\varphi,\chi)$ and $\pi_a(\varphi^*,\chi)$, where $\chi\in \widehat{\CS_\varphi}$, of a pure inner form $\SO(V)$ of $\SO^*_{2n}$ are conjugate under the outer action of $\RO(V)$. See \cite{AG17a, JZ18} for more explanation about the weak local Langlands correspondence for even special orthogonal groups.
\item If $m_++m_->0$, then $[\varphi]_{\rm c}$ is a singleton. In this case, the component group $\CS_\varphi=\BZ_2^u$ (resp. $\BZ_2^{u+1}$)
if exactly one of $m_+$ and $m_-$ is nonzero (resp. both $m_+$ and $m_-$ are nonzero).
\end{itemize}
In summary of the above discussions and in view of $\sigma_0=\sigma_+\oplus\sigma_-$, we make a reformulation and equivalently represent $[\varphi]_{\rm c}$ in the form
\begin{equation}\label{evenpa}
[\varphi]_{\rm c} = m_1\sigma_{\alpha_1}\oplus\cdots\oplus m_u\sigma_{\alpha_u} \oplus \varphi_{\rm quad} \oplus \xi\oplus\xi^\vee,
\end{equation}
where $m_i$'s are positive, $0\leq\alpha_1<\cdots<\alpha_u$ are even, $\varphi_{\rm quad}=m_\epsilon\sigma_\epsilon$ for some $\epsilon\in\{+,-\}$ and $m_\epsilon$ even, and $\xi$ has no irreducible orthogonal summands,
such that $\det\varphi=\BC(d)$ with $d:=\disc(V)$, and $\dim\varphi=2n$. Then we have the component group $\CS_\varphi \cong \BZ_2^u$.

An $L$-parameter in $\Phi_\gen(\SO^*_{2n'+1})$ is of the form
\begin{equation}\label{oddpa}
\phi=m_1'\sigma_{\beta_1}\oplus\cdots \oplus m_v' \sigma_{\beta_v} \oplus\xi'\oplus\xi^{'\vee},
\end{equation}
where $m_j'$'s are positive, $0<\beta_1<\cdots<\beta_v$ are odd, and $\xi$ has no irreducible symplectic summands, such that $\dim\phi=2n'$. We have the component group $\CS_\phi\cong\BZ_2^v$.

Similarly, denote $\alpha_i$ and $\beta_j$ symbolically the basis vectors in $\CS_\varphi$ and $\CS_\phi$ respectively. Then we have the characters $(\chi_{\varphi,\phi}, \chi_{\phi,\varphi}) \in \widehat{\CS_\varphi} \times \widehat{\CS_\phi}$  given by
\begin{align}
\label{chi1-R} \chi_{\varphi,\phi}(\alpha_i)& =\varepsilon(\sigma_{\alpha_i}\otimes\phi, \psi^\BR)\cdot\det\sigma_{\alpha_i}(-1)^{\dim\phi/2}
\cdot\det\phi(-1)^{\dim\phi_{\alpha_i}/2}=\prod_{j: \ \beta_j>\alpha_i}(-1)^{m'_j}, \\
  \label{chi2-R}\chi_{\phi,\varphi}(\beta_j)& = \varepsilon(\sigma_{\beta_j}\otimes\varphi, \psi^\BR)\cdot\det\sigma_{\beta_j}(-1)^{\dim\varphi/2}
  \cdot\det\varphi(-1)^{\dim\phi_{\beta_j}/2}\\
  &= \det\varphi(-1) \cdot \prod_{i:\ \alpha_i>\beta_j}(-1)^{m_i}= \prod_{i: \ \alpha_i<\beta_j} (-1)^{m_i}. \nonumber
\end{align}
These two formulas can be derived directly from the various facts from Section \ref{sec-LRN}, and we omit details of the calculation.

\subsubsection{Symplectic-metaplectic groups}

An $L$-parameter in $\Phi_\gen(\Sp_{2n})$ is of the form
\[
\varphi=m_1\sigma_{\alpha_1}\oplus\cdots\oplus m_u \sigma_{\alpha_u}\oplus m_+\sigma_+\oplus m_-\sigma_-\oplus \xi\oplus\xi^\vee,
\]
where $m_i$'s are positive, $0<\alpha_1<\cdots<\alpha_u$ are even, $m_++m_-$ is odd, and $\xi$ has no irreducible orthogonal summands, such that
$\det\varphi=1$ and $\dim\varphi=2n+1$.
To give the distinguished characters in this case, we also need  the $(2n+2)$-dimensional orthogonal parameter
\[
\varphi_1=\varphi\oplus \sigma_+.
\]
There are three cases:
\begin{itemize}
\item If $m_+\neq 0$, $m_-=0$, then
$
\CS_\varphi=\CS_{\varphi_1}=\BZ_2^u;
$
\item If $m_+=0$, $m_-\neq0$, then
$
\CS_\varphi=\BZ_2^u\hookrightarrow \CS_{\varphi_1}=\BZ_2^{u+1}
$
is of index 2;
\item If $m_+\neq 0$, $m_-\neq 0$, then
$
\CS_\varphi=\CS_{\varphi_1}=\BZ_2^{u+1}.
$
\end{itemize}
In summary of the above discussions and in view of $\sigma_0=\sigma_+\oplus\sigma_-$ again, we make a reformulation and equivalently represent $\varphi$ in the form
\begin{equation}\label{sppa}
\varphi=m_1\sigma_{\alpha_1}\oplus\cdots\oplus m_u \sigma_{\alpha_u}\oplus \varphi_{\rm quad}\oplus \xi\oplus\xi^\vee,
\end{equation}
where $m_i$'s are positive, $0\leq\alpha_1<\cdots<\alpha_u$ are even, $\varphi_{\rm quad}= m_\epsilon\sigma_\epsilon$ with some $\epsilon\in\{+,-\}$ and  $m_\epsilon$ odd, and $\xi$ has no irreducible orthogonal summands, such that
$\det\varphi=1$ and $\dim\varphi=2n+1$. Then we have the component group $\CS_\varphi \cong \BZ_2^u$.

By the local Langlands correspondence for metaplectic groups using theta correspondence, an $L$-parameter $\phi$ in $\Phi_\gen(\Mp_{2n'})$ is of the form \eqref{oddpa}, with $\CS_\phi\cong\BZ_2^v$.

Let $\chi_{\varphi_1,\phi}\in \widehat{\CS_{\varphi_1}}$ be defined as in  \eqref{chi1-R}, and
$\chi_{\varphi,\phi}:=\chi_{\varphi_1,\phi}|_{\CS_\varphi}\in \widehat{\CS_\varphi}$ be the restriction of $\chi_{\varphi_1,\phi}$ to $\CS_\varphi$. We also let $\chi_{\phi,\varphi}:=\chi_{\phi,\varphi_1}\in \widehat{\CS_\phi}$ be defined as in \eqref{chi2-R}. Then it follows that
\begin{align}
\label{chi1-R'} \chi_{\varphi,\phi}(\alpha_i)& =\prod_{j: \ \beta_j>\alpha_i}(-1)^{m'_j}, \\
  \label{chi2-R'}\chi_{\phi,\varphi}(\beta_j)& =\prod_{i: \ \alpha_i>\beta_j}(-1)^{m_i},
\end{align}
where for \eqref{chi2-R'} we used the fact that $\det\varphi=\det\varphi_1=1$.
With the notation of \eqref{char-spmp}, it is easy to check that the twist by $-1\in \BR^\times/ \BR^{\times 2}$ in this case gives that
\begin{equation}  \label{realspmp}
(\chi^{-1}_{\varphi, \phi}, \chi_{\phi,\varphi}^{-1}) =  (\chi_{\varphi, \phi}\cdot \eta_{-1},\chi_{\phi,\varphi}\cdot \eta_{-1}).
\end{equation}
That is, the formula in \eqref{chiz} uniformly holds for $F=\BR$, including the symplectic-metaplectic case.

\subsection{Reformulation}\label{ssec-userho}
For the convenience of later sections, we rewrite the expression of the generic $L$-parameters.  
By definition, when $F=\BR$, one always has  
\[
\CZ=\BR^\times/\BR^{\times 2}=\{\pm1\}.
\]
Recall from Definition \ref{defn:pd} the descent of $L$-parameters
$\FD_\Fl^z(\varphi,\chi)$ along $z\in \CZ$. To give a uniform description of the first descent $\FD_{\Fl_0}(\varphi,\chi)=\bigcupdot_{z\in\CZ}\FD_{\Fl_0}^z(\varphi,\chi)$ for all the cases, where $\Fl_0=\Fl_0(\varphi,\chi)$ is the first occurrence index of $(\varphi,\chi)$ defined by \eqref{FO}, we write
\begin{align}\label{rhoalpha}
\rho_\alpha=
\begin{cases}
\varsigma_\alpha, \quad \alpha\in \BZ, & \textrm{if }E=\BC; \\
\sigma_\alpha, \quad \alpha\geq 0, & \textrm{if }E=\BR,
\end{cases}
\end{align}
where $\varsigma_\alpha$ and $\sigma_\alpha$ are defined as in \eqref{0alpha}.
By Section \ref{sec-LPDC},  any $\varphi\in\widetilde{\Phi}_\gen(G_n^*)$ can be written as
\begin{equation} \label{par-unif}
\varphi=\bigoplus_{i=1}^r m_i\rho_{\alpha_i}\oplus \varphi_{\rm quad} \oplus \xi \oplus {}^c\xi^\vee,
\end{equation}
where 
\begin{itemize}
    \item $\alpha_1<\cdots<\alpha_r$ are integers of the same parity (and are nonnegative if $E=\BR$), with multiplicities $m_i>0$, $i=1,\ldots,r$;
    \item $\varphi_{\rm quad}$ is of the form $m_\epsilon\sigma_\epsilon$ with $\epsilon\in\{+,-\}$, which may occur only if $E=\BR$ and $\varphi$ is orthogonal;
    \item $\xi$ has no irreducible (conjugate) self-dual summands of the same type as $\varphi$. 
\end{itemize} 
Then $\CS_\varphi\cong \BZ_2^r$, and a basis of $\CS_\varphi$ is labeled by the integers $\alpha_i$ with $i=1,\ldots, r$.

The first occurrence index $\Fl_0=\Fl_0(\varphi,\chi)$ of the enhanced $L$-parameter 
$(\varphi,\chi)$ is one of the substantial ingredients in the theory. 
When $\CS_\varphi$ is trivial, i.e. $r=0$, it is straightforward to write the first occurrence index $\Fl_0=\Fl_0(\varphi,\chi)$ of the enhanced $L$-parameter 
$(\varphi,\chi)$ as follows. For $\Fn>0$, we have that 
\[
\Fl_0=\Fl_0(\varphi,{\mathds 1}) = \begin{cases} \Fn, & \textrm{if }G_n^*\neq \SO_{2n}^*, \\
\Fn-1, & \textrm{if }G_n^* = \SO_{2n}^*,
\end{cases}
\]
where $\mathds 1$ stands for the trivial character of the trivial group $\CS_\varphi$. Moreover, for each $z=\pm1$ we have that
\[
\FD^z_{\Fl_0}(\varphi,{\mathds 1}) = \begin{cases}
\{(0, {\mathds 1})\}, & \textrm{if }G_n^* \neq \Mp_{2n}, \\
\{(\phi_+, {\mathds 1})\}, & \textrm{if }G_n^* = \Mp_{2n}.
\end{cases}
\]
Note that the triviality of $\CS_\varphi$ implies that the $F$-anisotropic kernel of $(V,q)$ has dimension $\Fd_0\leq 1$ in the notation of Section \ref{ssec-CGIF}, and also rules out the case of odd unitary groups. 

By the proof of Proposition \ref{exmp:generic}, in this case $\wf_a(\varphi,{\mathds 1},q)^\mx$ consists of all the $F$-rational regular nilpotent orbits in $\CN_F(\Fg_n)_\circ$.

\begin{defn} \label{defn:maxtabtrivial}
Assume that $\CS_\varphi$ is trivial. Define $\udl{p}(\varphi,{\mathds 1})$ to be the unique partition of $\Fn$ corresponding to the regular nilpotent orbits in $\CN_F(\Fg_n)_\circ$. More precisely, if $\Fn>0$, then
\[
\udl{p}(\varphi,{\mathds 1}) = 
\begin{cases} 
\,[\Fn], & \textrm{if }G_n^*\neq \SO_{2n}^*, \\
\,[\Fn -1,1], & \textrm{if }G_n^* = \SO_{2n}^*.
\end{cases}
\]
\end{defn}

\begin{prop}\label{prop:trivial}
Theorems \ref{thm1} and \ref{thm2}
hold for any enhanced $L$-parameters $(\varphi,\chi)$ if the component group $\CS_\varphi$ is trivial. 
\end{prop}

\section{Structure of $\wf_\ari(\pi)$: Archimedean Case}\label{sec-PCR}


We study the structures of the arithemtic wavefront sets $\wf_\ari(\pi)$ when $F=\BR$ and for any enhanced $L$-parameters $(\varphi,\chi)$ with a non-trivial the component group $\CS_\varphi$. Hence we assume in this section that the component group $\CS_\varphi$ is nontrivial, i.e. the rank 
\begin{align}\label{r>0}
\rank_{\BZ_2}\CS_\varphi=r>0.
\end{align}
When $F=\BR$, we must have that $E=\BC$ for unitary groups and $E=\BR$ for other classical groups.

\subsection{Rational structure of first descents}\label{ssec-FD}

As a first preparatory step towards the proof of Theorem \ref{thm1}, 
we write down explicitly the first descent of an enhanced $L$-parameter $(\varphi, \chi)$ where $\varphi\in \widetilde{\Phi}_\gen(G_n^*)$ and $\chi\in \widehat{\CS}_\varphi$, using the formulas of distinguished characters in Section \ref{sec-LPDC}, which is the statement of Proposition \ref{prop:fd}. Those explicit arithmetic data are substantial to the proof of Theorem \ref{thm1} and 
to that of Theorem \ref{thm2}.

From now on, we change the notation and denote the elements in $\FD^z_\Fl(\varphi,\chi)$ by $(\phi, \chi')$ for convenience, so that  by Definition \ref{defn:pd},  
\begin{equation} \label{eq:changedual}
\chi = \chi^z_{\varphi,\wh\phi}, \quad \wh{\chi'} = \chi^z_{\wh\phi,\varphi}.
\end{equation}
We remark that if  $E=\BR$ (and $\phi$ only has summands of good parity when $G_n^*=\Mp_{2n}$, cf. \eqref{edec}), then $\wh\phi =\phi$ by Proposition \ref{prop:dualdata} (and the results in Section \ref{sec-LRN}). Thus when $E=\BR$, \eqref{eq:changedual} is equivalent to that
\[
\chi = \chi^z_{\varphi, \phi}, \quad \wh{\chi'} = \chi^z_{\phi,\varphi}.
\]
Recall from Section \ref{ssec-C} that in any case we have the natural identification  $\CS_\phi\cong \CS_{\wh\phi}$, which will be used for convenience without further explanation. 

\begin{defn}[Sign Alternating Set]\label{defn:SAS} 
For $\varphi\in\widetilde{\Phi}_\gen(G_n^*)$, which is of the form 
\[
\varphi=\bigoplus_{i=1}^r m_i\rho_{\alpha_i}\oplus \varphi_{\rm quad} \oplus \xi \oplus {}^c\xi^\vee
\]
as given by \eqref{par-unif}, and $\chi\in\widehat{\CS_\varphi}$, 
the {\it sign alternating set} of $\chi$, denoted by $\sgn(\chi)$, is defined as follows:
\begin{itemize}
\item
If $G_n^*\neq \SO_{2n+1}^*$, define 
\[
\sgn(\chi)=\{1\leq i<r\ \mid\  \chi(\alpha_i)\chi(\alpha_{i+1})=-1\};
\]
\item
If $G_n^*=\SO_{2n+1}^*$, formally write $\alpha_0:=-1$ and $\chi(\alpha_0):=1$, and define
\[
\sgn(\chi)=\{0\leq i<r\ \mid\ \chi(\alpha_i)\chi(\alpha_{i+1})=-1\}.
\]
\end{itemize}
\end{defn}

With $\sgn(\chi)$ defined as above, put
\begin{align}\label{n-chi}
s_\chi:=\#\textrm{sgn}(\chi),\quad \epsilon_\chi:=\textrm{sign}(-1)^{s_\chi} \in \{+,-\}.
\end{align}
Denote the indices in $\sgn(\chi)$ by
\begin{equation}\label{ind-sgn}
i_1<i_2<\cdots < i_{s_\chi}.
\end{equation}
It is clear that these notions only depend on the orbit $\ScO_\CZ(\chi)\subset \wh{\CS_\varphi}$ of $\chi$. Finally, define $z(\chi)\in\CZ=\{\pm1\}$ by
\begin{equation}\label{zchi}
z(\chi): = \chi(\alpha_r).
\end{equation}
Define an integer $\Fl_0$, with $(-1)^{\Fl_0-1}=\epsilon$ if $E=\BR$ (cf. Section \ref{ssec-LPD}), by the following identities:
\begin{align}\label{Fl0}
s_\chi = \begin{cases}
\Fn-\Fl_0, & \textrm{if }E=\BC, \\
\lfloor(\Fn-\Fl_0)/2\rfloor, & \textrm{if }E=\BR,
\end{cases}
\end{align}
with $s_\chi$ as defined in \eqref{n-chi}. 
It is clear that the integer $\Fl_0$ as defined in \eqref{Fl0} satisfies $0<\Fl_0\leq\Fn$. 
Moreover we can prove the following, which is the archimedean case of Proposition \ref{prop:DNE}.

\begin{prop}\label{prop:DNE-A}
For any enhanced $L$-parameter $(\varphi,\chi)$ with $\varphi\in\widetilde{\Phi}_\gen(G_n^*)$ 
and $\chi\in\wh{\CS_\varphi}$, the $\Fl_0$-descent descent $\FD_{\Fl_0}(\varphi,\chi)$ is 
non-empty. 
\end{prop}

\begin{proof}
By Definition \ref{defn:pd}, it suffices to show that there exists $\phi\in \wt{\Phi}_\gen(H^*_{\lfloor (\Fn-\Fl_0)/2\rfloor})$ such that 
\begin{equation} \label{pd:nonempty}
\chi_{\varphi, \wh{\phi}} \in \ScO_\CZ(\chi).
\end{equation}
Take a discrete $L$-parameter 
\[
\phi=\begin{cases} 
\bigoplus^{s_\chi}_{j=1}\varsigma_{\beta_j}, & \textrm{if }E=\BC, \\
\bigoplus^{s_\chi}_{j=1}\sigma_{\beta_j}, & \textrm{if }E=\BR, \ G_n^*\neq \Mp_{2n}, \\
\bigoplus^{s_\chi}_{j=1}\sigma_{\beta_j} \oplus \sigma_{\epsilon_\chi}, & \textrm{if }G_n^* = \Mp_{2n},
\end{cases}
\]
where $\beta_j\in(\alpha_{i_j}, \alpha_{i_j+1})$ has parity opposite to $\alpha_i$, $j=1,\ldots, s_\chi$. Here if $s_\chi=0$, then the above direct sums $\bigoplus^{s_\chi}_{j=1}\varsigma_{\beta_j}$ and $\bigoplus^{s_\chi}_{j=1}\sigma_{\beta_j}$ are interpreted as zero. Using \eqref{chi1-C}, \eqref{chi2-C}, \eqref{chi1-R}, \eqref{chi2-R}, \eqref{chi1-R'} and \eqref{chi2-R'}, it is straightforward to verify that \eqref{pd:nonempty} holds. 

For instance, if $G_n^*=\Sp_{2n}$, then we can verify that $\chi=\chi_{\varphi,\phi}\cdot \eta_{z(\chi)}$, which implies that 
\[
(\phi,\chi^{-z(\chi)}_{\phi,\varphi})\in\FD_{\Fl_0}^{z(\chi)}(\varphi,\chi)\subset \FD_{\Fl_0}(\varphi,\chi). 
\]
Hence $\FD_{\Fl_0}(\varphi,\chi)$ is not empty.
In fact, by definition of $\sgn(\chi)$, following notation in Definition \ref{defn:SAS}, we have 
\[
\chi(\alpha_i)=(-1)^{i_0}\chi(\alpha_r)=(-1)^{i_0}z(\chi)
\]
for all $1\leq i\leq r$, 
where $i_0=\#\{j\in[1,s_\chi]\ \mid\ \alpha_{i_j}\geq \alpha_i\}$ and $\sgn(\chi)=\{i_1<\cdots<i_{s_\chi}\}$.
On the other hand, by \eqref{chi1-R'}, we have that 
\[
\chi_{\varphi,\phi}(\alpha_i)=(-1)^{\#\{j\in[1,s_\chi]\ \mid\ \beta_j> \alpha_i\}}.
\]
By definition of $\beta_j$, we must have that $\beta_j>\alpha_i$ if and only if $\alpha_{i_j}\geq \alpha_i$. This implies that 
\[
\{j\in[1,s_\chi]\ \mid\  \beta_j> \alpha_i\}=\{j\in [1,s_\chi]\ \mid\ \alpha_{i_j}\geq \alpha_i\}.
\]
which yields that $\chi(\alpha_i)=\chi_{\varphi,\phi}(\alpha_i)\cdot \eta_{z(\chi)}(\alpha_i)$ for all $1\leq i\leq r$. We are done.
\end{proof}

A case-by-case analysis in a similar pattern, combined with the contragredient of enhanced $L$-parameters given by Proposition \ref{prop:dualdata}, gives the following result, which shows that the first descent $\FD_{\Fl_0}(\varphi,\chi)$ consists of finitely many discrete $L$-parameters interlacing the sign alternating set.

\begin{prop}\label{prop:fd}
Assume that $\varphi\in\widetilde{\Phi}_\gen(G_n^*)$ is written in Definition \ref{defn:SAS} and $\chi\in \wh{\CS_\varphi}$. 
The integer $\Fl_0=\Fl_0(\varphi,\chi)$ as given in \eqref{Fl0} is the first occurrence index 
for the enhanced $L$-parameter $(\varphi,\chi)$. 
Moreover, the first descent $\FD_{\Fl_0}(\varphi,\chi)$ is given as follows. Let $i_1<i_2<\cdots <i_{s_\chi}$ be the indices in $\sgn(\chi)$ as in \eqref{ind-sgn}. 
\begin{enumerate}
\item \label{FD-U} If $G_n^*=\RU_\Fn^*$, then $\FD_{\Fl_0}(\varphi, \chi) =\FD_{\Fl_0}^{z(\chi)}(\varphi, \chi)$, which equals
\[
\left\{(\phi,\chi')\ \mid\ \phi=\bigoplus^{s_\chi}_{j=1} \varsigma_{\beta_j}, \, \wh{\chi'} =\chi^{z(\chi)}_{\wh{\phi},\varphi}, \, \beta_j\in (\alpha_{i_j}, \alpha_{i_j+1}) \textrm{ has  parity opposite to }\alpha_i  \right\};
\]
\item \label{FD-odd-SO} If $G_n^*=\SO_{2n+1}^*$, then $\FD_{\Fl_0}^z(\varphi,\chi)$ with $z=\pm1$ equals 
\[
\left\{(\phi,\chi^z_{\phi,\varphi})\ \mid\   \phi=\bigoplus^{s_\chi}_{j=1} \sigma_{\beta_j},\, \beta_j\in (\alpha_{i_j}, \alpha_{i_j+1}) \textrm{ is even}\right\}.
\]
\item \label{FD-even-SO}
If $G_n^*=\SO_{2n}^*$, then $\FD_{\Fl_0}(\varphi, \chi) =\FD_{\Fl_0}^{z(\chi)}(\varphi, \chi)$, which equals
\[
\left\{(\phi,\chi_{\phi,\varphi})\ \mid\   \phi=\bigoplus^{s_\chi}_{j=1} \sigma_{\beta_j},\, \beta_j\in (\alpha_{i_j}, \alpha_{i_j+1}) \textrm{ is odd}\right\}.
\]
\item \label{FD-Sp}
If $G_n^*=\Sp_{2n}$, then $\FD_{\Fl_0}(\varphi, \chi) =\FD_{\Fl_0}^{z(\chi)}(\varphi, \chi)$, which equals
\[
\left\{(\phi,\chi^{-z(\chi)}_{\phi,\varphi})\ \mid\ \phi=\bigoplus^{s_\chi}_{j=1} \sigma_{\beta_j},\, \beta_j\in (\alpha_{i_j}, \alpha_{i_j+1}) \textrm{ is odd}\right\}.
\]
\item \label{FD-Mp}
If $G_n^*=\Mp_{2n}$, then $\FD_{\Fl_0}(\varphi, \chi) =\FD_{\Fl_0}^{z(\chi)}(\varphi, \chi)$,  which equals
\[
\left\{(\phi,\chi^{-z(\chi)}_{\phi,\varphi})\ \mid\  \phi = \bigoplus^{s_\chi}_{j=1} \sigma_{\beta_j} \oplus \sigma_{\epsilon_\chi}, \, \beta_j\in (\alpha_{i_j}, \alpha_{i_j+1}) \textrm{ is even}\right\}.
\]
\end{enumerate}
\end{prop}

\begin{proof}
We give the proof of \eqref{FD-U}, \eqref{FD-odd-SO} and \eqref{FD-Mp}.  The cases \eqref{FD-even-SO} and \eqref{FD-Sp} are similar to the case \eqref{FD-U}, which will be omitted. 

\quad 

\eqref{FD-U} $G_n^*=\RU_\Fn^*$. Assume that for some $\Fl>0$, the $\Fl$-th descent $\FD_\Fl(\varphi,\chi)$ is not empty. We take 
\[
(\phi, \chi')\in \FD^z_\Fl(\varphi,\chi)
\]
for some $z\in \CZ$. From \eqref{chi-a} and Definition \ref{defn:pd}, we must have that 
\begin{equation} \label{chi_z}
\chi=\chi^z_{\varphi,\wh\phi}= \chi_{\varphi,\wh\phi}\cdot \eta_z, \quad \wh{\chi'} = \chi^z_{\wh{\phi},\varphi} = \chi_{\wh{\phi},\varphi} \cdot\eta_{(-1)^{\Fl}z}.
\end{equation}
By  $\chi(\alpha_{i_j})\chi(\alpha_{i_j+1})=-1$ for all $i_j\in\sgn(\chi)$ and \eqref{chiz},
we have $\chi_{\varphi,\wh\phi}(\alpha_{i_j})\chi_{\varphi,\wh\phi}(\alpha_{i_j+1})=-1$.
Due to \eqref{chi1-C}, the sum of the multiplicities of irreducible summands $\sig_{\beta_j}$ of $\phi$ with $\beta_j\in (\alpha_{i_j},\alpha_{i_j+1})$ is odd.  
Thus, $\phi$ has at least one summand $\sig_{\beta_j}$ with odd multiplicity and $\beta_j\in (\alpha_{i_j},\alpha_{i_j+1})$  having the parity opposite to the $\alpha_i$'s. 
It follows that
\[
\dim \phi = \Fn-\Fl \geq s_\chi = \#\sgn(\chi),
\]
and the equality holds if and only if $\phi$ is of the form
\[
\phi = \bigoplus^{s_\chi}_{j=1}\varsigma_{\beta_j},
\]
where $\beta_j \in (\alpha_{i_j}, \alpha_{i_j+1})$, $j=1,\ldots, s_\chi$, has parity opposite to the $\alpha_i$'s.
In this case using \eqref{chi1-C} again we find that 
$
\chi_{\varphi,\wh\phi}(\alpha_r) =1,
$
which by \eqref{chi_z} forces that
\[
z = \chi(\alpha_r).
\]
By taking $\Fl_0 = \Fn-s_\chi$ as given in \eqref{Fl0}, we must have that 
$\Fl\leq \Fl_0$. By Proposition \ref{prop:DNE-A}, $\FD_{\Fl_0}(\varphi,\chi)$ is not empty. 
Hence $\Fl_0=\Fl_0(\varphi,\chi)$ is the first occurrence index of $(\varphi,\chi)$. Finally, 
by the same argument in the proof of Proposition \ref{prop:DNE-A}, we obtain that 
the first descent $\FD_{\Fl_0}(\varphi,\chi)$ is given as in  \eqref{FD-U} for unitary groups.

\quad

\eqref{FD-odd-SO} $G_n^*=\SO_{2n+1}^*$. Assume that for some $\Fl>0$, the $\Fl$-th descent $\FD_\Fl(\varphi,\chi)$ is not empty. Take some 
$(\phi,\chi')\in \FD^z_\Fl(\varphi,\chi)$ with $z=\pm 1$. 
By Definition \ref{defn:pd} and Proposition \ref{prop:dualdata}, we have that
\[
\chi = \chi_{\varphi,\phi},\quad \chi' = \chi_{\phi,\varphi}\cdot\eta_z.
\]
Similar to the above, by \eqref{chi2-R} for every $i_j\in \sgn(\chi)$, $\phi$ has at least one summand $\sigma_{\beta_j}$, where $\beta_j\in(\alpha_{i_j}, \alpha_{i_j+1})$ is even. If we put
\[
\phi=\bigoplus^{s_\chi}_{j=1}\sigma_{\beta_j},
\]
then using \eqref{chi2-R} we find that
$
\chi_{\varphi, \phi}=\chi.
$
It is easy to check that $\Fl\leq 2n+1-2s_\chi$. By taking $\Fl_0= 2n+1-2s_\chi$ as in \eqref{Fl0}, Proposition \ref{prop:DNE-A} shows that the 
$\Fl_0$-th descent $\FD_{\Fl_0}(\varphi,\chi)$ is not empty. Hence $\Fl_0=\Fl_0(\varphi,\chi)$ is the first occurrence index of $(\varphi,\chi)$ in this case. Finally
the first descent $\FD_{\Fl_0}(\varphi,\chi)$ is given as in  \eqref{FD-odd-SO} for odd special orthogonal groups.

\quad

\eqref{FD-Mp} $G_n^*=\Mp_{2n}$. Assume that for $\Fl>0$, the $\Fl$-th descent 
$\FD_{\Fl}(\varphi,\chi)$ is not empty. Take  $(\phi,\chi')\in \FD^z_{\Fl}(\varphi,\chi)$ for some $z\in\CZ$. By Definition 
\ref{defn:pd}, Proposition \ref{prop:dualdata} and \eqref{realspmp}, we have that
\begin{equation}\label{chiz_Mp}
\chi = \chi_{\varphi,\wh{\phi}}\cdot\eta_{z}, \quad
\chi' = \chi_{\wh{\phi},\varphi}\cdot \eta_{-z}.
\end{equation}
Similar to the above, by \eqref{chi2-R'} for every $i_j\in \sgn(\chi)$, $\phi$ has at least one summand $\sigma_{\beta_j}$, where $\beta_j\in(\alpha_{i_j}, \alpha_{i_j+1})$ is even. 
It follows that $\phi$ contains a subrepresentation 
\[
\phi' = \bigoplus^{s_\chi}_{j=1}\sigma_{\beta_j},
\]
where $\beta_j\in (\alpha_{i_j}, \alpha_{i_j+1})$ is even. 
Then $\dim\phi=2n-\Fl+1\geq 2s_\chi$.
Since $\Fl$ is even, it implies that $\dim\phi\geq 2s_\chi+1$ and $\Fl \leq 2n-2s_\chi$. 

Similarly, by taking $\Fl_0 = 2n-2s_\chi$ in \eqref{Fl0}, we have $\Fl\leq \Fl_0$. 
Proposition \ref{prop:DNE-A} implies that $\FD_{\Fl_0}(\varphi,\chi)$ is non-empty and  $\phi$ is of the form
\[
\phi = \bigoplus^{s_\chi}_{j=1}\sigma_{\beta_j} \oplus \sigma_{\epsilon_\chi},
\]
where the choice of $\epsilon_\chi$ guarantees $\det(\phi)=1$ as desired.
In addition, by $\beta_{j}< \alpha_{r}$ and \eqref{chi2-R'}, $\chi_{\varphi,\phi}(\alpha_r)=1$ and then
$z=\chi(\alpha_r)$ due to \eqref{chiz_Mp}.
Hence $\Fl_0=\Fl_0(\varphi,\chi)$ is the first occurrence index of $(\varphi,\chi)$,
and the first descent $\FD_{\Fl_0}(\varphi,\chi)$ is given as in \eqref{FD-Mp} for the metaplectic groups. 
\end{proof}

For later discussions, we make a simple but important observation. 

\begin{defn}\label{defn:ord-pre}
Let  $\varphi, \varphi'$ be two $L$-parameters of the same type such that 
\[
\rank_{\BZ_2}\CS_\varphi=\rank_{\BZ_2}\CS_{\varphi'}=r.
\]
Let $\alpha_1<\cdots<\alpha_r$ and $\alpha_1'<\cdots<\alpha_r'$ denote symbolically the natural bases of $\CS_\varphi$ and $\CS_{\varphi'}$, respectively. Define the unique {\it order-preserving} isomorphism
\begin{equation}\label{ord-pre}
\iota_{\varphi,\varphi'}:\ \CS_\varphi\cong  \CS_{\varphi'},\quad \alpha_i\mapsto \alpha_i',
\end{equation}
which gives a bijection 
\[
\wh{\CS_{\varphi'}}\to \wh{\CS_{\varphi}}, \quad \chi'\mapsto \chi'\circ\iota_{\varphi,\varphi'}.
\]
\end{defn}

\begin{prop} \label{prop:pretab}
Assume that $(\varphi,\chi,q)\in \FT_a(G_n^*)$, and  let $\Fl_0=\Fl_0(\varphi,\chi)$. There is a unique one-dimensional sesquilinear form $q^h_1$ such that $\CT_a(\varphi,\chi,q)$ contains a non-empty subset of the form
\[
(\Fl_0, q_{(\Fl_0)}^h)\star \CT_{a}(\varphi_1, \chi_1, q^{(1)})
\]
where $(\varphi_1,\chi_1)\in \FD_{\Fl_0}(\varphi,\chi)$ subject to the conditions in Definition \ref{LYT} if and only if  $q^h_{(\Fl_0)} = q^h_1$.
Moreover, $q^h_1$ is determined explicitly as follows. 
\begin{enumerate}
    \item \label{FD-odd-SO-pretab} If $G_n^*=\SO_{2n+1}^*$, then 
    \[
    \disc(q^h_1) = \disc(q)\cdot \det(\varphi_1) = \disc(q)\cdot (-1)^{s_\chi}. 
    \]
    In this case $z=-a\cdot \disc(q^h_1)$.
  \item \label{FD-pretab} If $G_n^* \neq \SO_{2n+1}^*$, then 
  \[
  [\disc(q^h_1)] = a\cdot z(\chi).
  \]
\end{enumerate}
\end{prop}

\begin{proof} This proposition follows from Propositions \ref{prop:rational} and \ref{prop:fd}, and the explicit arithmetic data in the archimedean case. We give the proof for completeness. 

By Proposition \ref{prop:fd}, for any $(\varphi_1,\chi_1)\in \FD_{\Fl_0}(\varphi,\chi)$, we have that 
\[
\rank_{\BZ_2}\CS_{\varphi_1} = s_\chi =\#\sgn(\chi). 
\]
Moreover, if $(\varphi_1,\chi_1)$, $(\varphi_1', \chi_1')\in \FD^z_{\Fl_0}(\varphi,\chi)$ for some $z\in \CZ$, then it is easy to see that 
\begin{equation} \label{chi1-eq}
\chi_1 = \chi_1'\circ \iota_{\varphi_1,\varphi_1'}.
\end{equation}
under the order-preserving isomorphism $\iota_{\varphi_1,\varphi_1'}$ given by Definition \ref{defn:ord-pre}.

We again only explain the proof for \eqref{FD-odd-SO-pretab}. The other cases in 
\eqref{FD-pretab} are direct consequences of Propositions \ref{prop:rational} and \ref{prop:fd}, which will be omitted. 
If $G_n^*=\SO_{2n+1}^*$, then $(\varphi_1, \chi_1)\in \FD_{\Fl_0}(\varphi,\chi)$
 is an enhanced $L$-parameter of an even special orthogonal group 
 $\SO(V^{(1)}, q^{(1)})$. By the local Langlands correspondence,
 \[
 \disc(q^{(1)}) = \det(\varphi_1) = (-1)^{\rank_{\BZ_2}\CS_{\varphi_1}}=(-1)^{s_\chi},
 \]
 which does not depend on the choice of $(\varphi_1,\chi_1)$. The proposition in this case follows from that 
 \[
 \disc(q^h_1) = \disc(q)\cdot \disc(q^{(1)}). 
 \]
\end{proof}

By Proposition \ref{prop:fd} and Proposition 
\ref{prop:pretab}, we may define a {\it first descent sequence} of a triple $(\varphi,\chi, q)\in \FT_a(G_n^*)$. Take $(\varphi_0,\chi_0, q^{(0)})=(\varphi,\chi, q)$,  and define a pre-tableau $(p_1, q^1_h)$, where
\[
p_{1}:= \Fl_0(\varphi_0, \chi_0)
\]
is the first occurrence index of the $L$-parameter descent of $(\varphi_0,\chi_0)$, and $q^h_1$ is the unique one-dimensional sesquilinear form determined in Proposition \ref{prop:pretab}. Then
$\CT_{a}(\varphi_0, \chi_0, q^{(0)}) = \CT_a(\varphi,\chi, q) $ contains a non-empty subset of the form
\[
(p_1, q^h_1)\star \CT_{a}(\varphi_1,\chi_1, q^{(1)}),
\]
subject to the conditions of Definition \ref{LYT}. By Proposition \ref{prop:fd}, the enhanced $L$-parameter $(\varphi_1,\chi_1)\in \FD_{\Fl_0}(\varphi,\chi)$ is discrete. Thus we are able to continue the induction process. 

For $i\geq 1$, if $\CS_{\varphi_{i-1}}$ is nontrivial, define \[ p_i:=\Fl_0(\varphi_{i-1},\chi_{i-1}).\] 
In this case by Proposition \ref{prop:pretab} again, there exists a unique one-dimensional sesquilinear form $q^h_{i}$ such that $\CT_{a}(\varphi_{i-1},\chi_{i-1}, q^{(i-1)})$ contains a non-empty subset of the form
\[
(p_{i}, q^h_{(p_{i})})\star \CT_{a}(\varphi_{i},\chi_{i}, q^{(i)}),
\]
where $(\varphi_{i},\chi_{i})\in \FD_{p_i}(\varphi_{i-1},\chi_{i-1})$ subject to the conditions in Definition \ref{LYT} if and only if 
$q^h_{(p_{i})} =q^h_{i}$. Then we obtain the pre-tableau $(p_{i},q^h_{i})$.  We recall from Definition \ref{LYT} that
\[
q^{(i-1)} \cong (q^h_i \otimes q_{p_i})\oplus q^{(i)}.
\]

If $\CS_{\varphi_{i-1}}$ is trivial, we use Proposition \ref{prop:trivial} and the relevant 
discussions in Section \ref{ssec-userho}. 
As a discrete $L$-parameter, $(\varphi_{i-1},\chi_{i-1})$ has to be the unique enhanced $L$-parameter of $G^*=\SO_0, \RU_0, \Sp_0, \Mp_0$ or $\SO_1$. In this case if $G^*\neq \SO_1$ then we stop, otherwise we define $(p_i, q^h_i) = (1, q^{(i-1)})$.

It is clear that this inductive process ends with a finite number of steps. We assume that 
for a triple $(\varphi,\chi, q)\in \FT_a(G_n^*)$, the first descent sequence goes as follows:
\begin{align}\label{descentseq}
    \begin{matrix}
    (\varphi_0,\chi_0, q^{(0)}),&(\varphi_1,\chi_1,q^{(1)}),&\ldots,&(\varphi_{k-1},\chi_{k-1},q^{(k-1)})\\
    (p_1, q^h_1),& (p_2, q^h_2), &\ldots,& (p_k, q^h_k)
    \end{matrix}
\end{align}
such that $p_1,\ldots, p_k$ are positive. 
It is clear that the sequence $((p_1, q^h_1),\ldots, (p_k, q^h_k))$ is a pre-tableau in $\CT_a(\varphi,\chi,q)$.

\begin{prop} \label{prop:unique}
The pre-tableau $((p_1, q^h_1),\ldots, (p_k, q^h_k))\in \CT_a(\varphi,\chi,q)$ as given in \eqref{descentseq} is independent of the choice of each $(\varphi_{i}, \chi_{i})$ in $\FD_{p_{i}}(\varphi_{i-1},\chi_{i-1})$ for $i=1,2,\dots, k-1$.
\end{prop}

\begin{proof}
This follows from the following more general statement, applied to the special case that $(\varphi, \chi) = (\varphi', \chi')$.
\end{proof}

\begin{lem}\label{pre-tableau-eq}
Assume that $(\varphi,\chi, q)$, $(\varphi', \chi', q)\in \FT_a(G_n^*)$, where
\[
\varphi=\bigoplus_{i=1}^r m_i\rho_{\alpha_i}\oplus \varphi_{\rm quad} \oplus \xi \oplus {}^c\xi^\vee \quad{\rm and}\quad
\varphi'=\bigoplus_{i=1}^r m_i'\rho_{\alpha_i'}\oplus \varphi'_{\rm quad} \oplus \xi' \oplus {}^c {\xi'}^\vee
\]
are of the form \eqref{par-unif}. Assume that 
$m_i\equiv m_i' \mod 2$,  $i=1,\ldots, r$, and that
\begin{equation}\label{chi-eq}
\chi = \chi'\circ \iota_{\varphi,\varphi'}
\end{equation}
where $\iota_{\varphi,\varphi'}$ is given by \eqref{ord-pre}.
Let $((p_1, q^h_1),\ldots, (p_k, q^h_k))\in \CT_a(\varphi,\chi,q)$ be given by \eqref{descentseq}. Similarly choose a first descent sequence 
\begin{align*}
    \begin{matrix}
    (\varphi_0',\chi'_0, q^{(0)'}),&(\varphi'_1,\chi'_1, {q^{(1)'}}),&\dots,&(\varphi'_{k'-1},\chi'_{k'-1}, {q^{(k'-1)'}})\\
    (p'_1, {q_1^{h'}}),& (p'_2, {q^{h'}_2}), &\dots,& (p'_{k'}, {q^{h'}_{k'}})
    \end{matrix}
\end{align*}
for $(\varphi', \chi', q)$. Then 
\[
((p_1, q^h_1),\ldots, (p_k, q^h_k)) = ((p'_1, {q_1^{h'}}),\cdots, (p'_{k'}, {q^{h'}_{k'}})). 
\]
\end{lem}

\begin{proof}
We use induction on $\Fn$. The lemma can be verified directly when $\Fn= 1$ and $E=\BC$, or $\Fn\leq 3$ and $E=\BR$.

The assumption \eqref{chi-eq} implies that 
\begin{equation} \label{rank-z-eq}
\rank_{\BZ_2}\CS_{\varphi_1} = s_\chi = s_{\chi'} = \rank_{\BZ_2}\CS_{\varphi_1'}, \quad z(\chi)=z(\chi').
\end{equation}
Then $(p_1, q^h_1)=(p_1',q^{h'}_1)$ by Propositions \ref{prop:fd} and \ref{prop:pretab}.
In any case, we have that   
\begin{equation}
    \label{z-eq}
    (\varphi_1,\chi_1)\in \FD^z_{p_1}(\varphi,\chi) \quad \textrm{and} \quad (\varphi_1',\chi_1')\in \FD^z_{p_1}(\varphi',\chi')
\end{equation}
for some $z\in \CZ$. By the formulas for distinguished characters in Section \ref{sec-LPDC} and Proposition \ref{prop:fd}, it holds that
\begin{equation} \label{chi1=chi1'}
\chi_1 = \chi_1'\circ \iota_{\varphi_1,\varphi_1'}. 
\end{equation}
Note that $(\varphi,\chi, q)$ corresponds to a generic representation if and only if $\rank_{\BZ_2}\CS_{\varphi_1}= s_\chi = 0$, and  the same is true for $(\varphi',\chi',q)$. If $s_\chi = s_{\chi'}=0$, then the induction is done.

Assume that $s_\chi = s_{\chi'}>0$. By \eqref{chi1=chi1'}, we can apply the induction hypothesis for the  triples
\[
(\varphi_1, \chi_1, q^{(1)}), \ (\varphi_1', \chi_1', q^{(1)})\in \FT_{a}(H_{\lfloor(\Fn-p_1)/2\rfloor}^*)
\]
to conclude that
\[
((p_2, q^h_2),\dots, (p_k, q^h_k)) = ((p'_2, {q^{h'}_2}),\dots, (p'_{k'}, {q^{h'}_{k'}})),
\]
which finishes the induction.
\end{proof}

Proposition \ref{prop:unique} leads to the following definition. 

\begin{defn}[Pre-tableau $\Fs_a(\varphi,\chi,q)$ and Ordered Partition $\udl{p}(\varphi,\chi)$]  \label{defn:maxtab}
Define the pre-tableau $\Fs_a(\varphi,\chi,q)$  and the ordered partition $\udl{p}(\varphi,\chi)$
associated with the triple $(\varphi,\chi,q)\in \FT_{a}(G_n^*)$ as 
\[
\Fs_a(\varphi,\chi,q) = ((p_1, q^h_1),\ldots, (p_k, q^h_k))\in \CT_a(\varphi,\chi,q)
\]
and
\[
\underline{p}(\varphi,\chi)=[p_1,p_2,\dots,p_k],
\]
where the first descent sequence $(p_1,q^h_1), \ldots, (p_k, q^h_k)$ is given in \eqref{descentseq}.
\end{defn}

\begin{prop}\label{prp:partition}
With the ordered partition $\underline{p}(\varphi, \chi)$ of $\Fn$ and the pre-tableau $\Fs_a(\varphi,\chi,q)$ as defined in Definition \ref{defn:maxtab}, the following hold. 
\begin{enumerate}
\item \label{prp:partition1} The ordered partition $\underline{p}(\varphi, \chi)$ of $\Fn$ is decreasing with $p_1\geq\cdots\geq p_k$, hence is an $L$-descent partition in $\CP_a(\varphi,\chi,q)$.
\item  \label{prp:partition2} The pre-tableau $\Fs_a(\varphi,\chi,q)$ belongs to the set of $L$-descent tableaux $\CL_a(\varphi,\chi,q)$, and defines an $F$-rational nilpotent orbit 
$\CO_a(\varphi,\chi,q)\in \wf_a(\varphi,\chi,q)$.  
\end{enumerate}
\end{prop}

\begin{proof}
By definition, Part \eqref{prp:partition2} follows from Part \eqref{prp:partition1}. The proof of Part \eqref{prp:partition1} goes as follows. 

From  Definition \ref{defn:maxtab},  the ordered partition can be written as 
\[
\underline{p}(\varphi,\chi)=[p_1,p_2,\dots,p_k],
\]
where $p_1,\ldots, p_k$ are as in the construction of $\Fs_a(\varphi,\chi,q)$.
Clearly it suffices to show that $p_1\geq p_2$, since $\underline{p}(\varphi,\chi)$ is defined inductively through the consecutive descents of enhanced $L$-parameters. 
Recall from \eqref{par-unif} that $m_i$, $1\leq i\leq r$, denotes the multiplicity of $\rho_{\alpha_i}$ in the 
$L$-parameter $\varphi$. 
From \eqref{ind-sgn}, the indices in $\textrm{sgn}(\chi)$ satisfy the following inequalities:
$i_1<\cdots<i_{s_\chi}.$ 
For $j=0,\ldots, s_\chi$, put 
\begin{equation} \label{eq:l_j}
l_j=m_{i_j+1}+\cdots+m_{i_{j+1}},
\end{equation}
where $i_0:=0$ and $i_{s_\chi+1}:=r$. Then $l_j>0$ if $j>0$. 

For  any enhanced $L$-parameter $(\varphi_1,\chi_1)\in \FD_{\Fl_0}(\varphi, \chi)$, where $\Fl_0=\Fl_0(\varphi,\chi)$ is the first occurrence index of $(\varphi,\chi)$ defined by \eqref{FO}, denote $\beta_1,\ldots, \beta_{s_\chi}$ the basis vectors of $\CS_{\varphi_1}$ such that \[\alpha_{i_j}<\beta_j<\alpha_{i_j+1}.\]
Then we can use the formulas in Section \ref{sec-LPDC} to compute 
$
s_{\chi_1} =\#\sgn(\chi_1)
$
in various cases. 

\quad

    {\bf (i) $G_n^* =  \SO_{2n}^*$:} In this case, we have that  $l_0>0$, and
    $
    s_{\chi_1} =\#\{0\leq j<s_\chi\ \mid \  l_j\textrm{ is odd}\}.
    $
    By Proposition \ref{prop:fd}, we obtain that 
    \[
    s_\chi  = \frac{\Fn-p_1-1}{2},\quad n_1 = \frac{\Fn-p_1-p_2}{2},
    \]
    which imply the following inequalities:
    \begin{equation}\label{eq:p1-p2-SO-even}
\frac{p_1-p_2}{2} = \frac{\Fn}{2}+s_{\chi_1}-2s_\chi -1 \geq \sum^{s_\chi}_{j=0}l_j+s_{\chi_1}-2s_\chi-1\geq 0.
\end{equation}

{\bf (ii) $G_n^* = \SO_{2n+1}^*$:} In this case, we have that 
$
s_{\chi_1} =  \{0<j<s_\chi\ \mid \ l_j\textrm{ is odd}\}.
$
By Proposition \ref{prop:fd}, we obtain that 
 \[
    s_\chi = \frac{\Fn-p_1}{2},\quad s_{\chi_1} = \frac{\Fn-p_1-p_2-1}{2},
    \]
    which yield the following inequalities:
    \begin{equation}\label{eq:p1-p2-SO-odd}
    \frac{p_1-p_2}{2}  = \frac{\Fn+1}{2}+s_{\chi_1}-2s_\chi  \geq \sum^{s_\chi}_{j=0}l_j+1+s_{\chi_1}-2s_\chi \geq l_{s_\chi}-1\geq0.
    \end{equation}

{\bf (iii) $G_n^* =\Sp_{2n}$ or $\Mp_{2n}$:} In those cases, we have that  $l_0>0$, and 
\[
s_{\chi_1}=  \#\{0< j<s_\chi \ \mid \  l_j\textrm{ is odd}\}.
\]
By Proposition \ref{prop:fd}, we obtain that 
    \[
    s_\chi = \frac{\Fn-p_1}{2},\quad s_{\chi_1} = \frac{\Fn-p_1-p_2}{2},
    \]
    which imply the following inequalities:
    \begin{equation} \label{eq:p1-p2-sp-mp}
    \frac{p_1-p_2}{2} =   \frac{\Fn}{2}+s_{\chi_1}-2s_\chi \geq \sum^{s_\chi}_{j=0}l_j +s_{\chi_1}-2s_\chi\geq l_{s_\chi}-1\geq 0.
    \end{equation}
    
   {\bf (iv)  $G_n^*=\RU_\Fn^*$:}  In this case, we have that  $l_0>0$, and
$
s_{\chi_1} =\#\{0< j<s_\chi\ \mid \ l_j\textrm{ is odd}\}.
$
By Proposition \ref{prop:fd}, we get that 
 \[
    s_\chi = \Fn-p_1, \quad s_{\chi_1}=s_\chi- p_2,
    \]
    which yield the following inequalities:
\begin{equation} \label{eq:p1-p2-u}
p_1-p_2=\Fn+s_{\chi_1}-2s_\chi\geq \sum^{s_\chi}_{j=0} l_j + s_{\chi_1}-2s_\chi\geq l_0+l_{s_\chi}-2\geq 0.
\end{equation}

Thus we finish the proof that $p_1\geq p_2$ in all cases.
\end{proof}







\subsection{Proof of Theorem \ref{thm1}}\label{ssec-PC}

In order to prove Theorem \ref{thm1}, it remains to show that the partition $\underline{p}(\varphi, \chi)$ in Definition \ref{defn:maxtabtrivial} and Definition \ref{defn:maxtab} corresponds to the unique stable nilpotent orbit that meets all the $F$-rational maximal nilpotent orbits in the arithmetic wavefront set $\wf_\ari(\pi)=\wf_a(\varphi,\chi, q)$ with $\pi=\pi_a(\varphi,\chi)$.

Denote by $\FP(\Fn)$ the set of ordered partitions of $\Fn$, which are not required to be decreasing.
We extend the usual partition order to $\FP(\Fn)$ as follows:
for $\underline{p}=[p_1,\ldots,p_k]$ and $\underline{p}'=[p_1',\ldots,p_{k'}']$ in $\FP(\Fn)$,
define that $\underline{p}'\leq \underline{p}$ if and only if
\begin{equation} \label{part-order}
\sum^{i}_{j=1}p_j'\leq \sum^{i}_{j=1}p_j
\end{equation}
for all $i=1,\ldots, \max\{k,k'\}$,
where $p_j$, $j>k$ and $p_j'$, $j>k'$ are interpreted as zero.

For any triple $(\varphi,\chi,q)\in \FT_a(G_n^*)$, 
the set of $L$-descent partitions 
$\CP_a(\varphi,\chi,q)$ associated to the set of $L$-descent pre-tableaux $\CL_a(\varphi,\chi,q)$ 
is as defined in Section \ref{sec-AWFS}. Let $\FP_a(\varphi,\chi,q)$ be 
the set of ordered partitions associated to the set of pre-tableaux $\CT_a(\varphi,\chi,q)$. It is clear that 
$\CP_a(\varphi,\chi,q)$ is contained in $\FP_a(\varphi,\chi,q)$.
By Proposition \ref{prp:partition}, we have that $\udl{p}(\varphi,\chi)\in \CP_a(\varphi,\chi,q)$. 
Thus Theorem \ref{thm1} follows from the following stronger result.

\begin{thm}[Maximality of Partition]\label{thm:main}
The partition $\underline{p}(\varphi,\chi)$  in Definition \ref{defn:maxtabtrivial} and Definition \ref{defn:maxtab} is the unique maximal element in
$\FP_a(\varphi,\chi,q)$.
\end{thm}

Since $\CP_a(\varphi,\chi,q)$ is contained in $\FP_a(\varphi,\chi,q)$ and $\underline{p}(\varphi,\chi)$ belongs to $\CP_a(\varphi,\chi,q)$,
Theorem~\ref{thm:main} implies that the partition $\underline{p}(\varphi,\chi)$ is the unique maximal element in  $\CP_a(\varphi,\chi,q)$.
Hence Theorem \ref{thm1} is finally proved. 





When the component group $\CS_\varphi$ is trivial, Theorem \ref{thm:main} is given by Proposition \ref{prop:trivial}. It remains to prove Theorem \ref{thm:main} when $\CS_\varphi$ is nontrivial. In this case, we have that 
 \[
 \udl{p}(\varphi,\chi) = [p_1,\ldots, p_k]
 \]
 is given as in Definition \ref{defn:maxtab}.

 Assume that  $\underline{p}'=(p_1',\ldots,p_{k'}')\in \FP_a(\varphi,\chi,q)$ is associated to a pre-tableau
 \begin{equation} \label{pretableau-s'}
 \Fs' = ((p_1', {q^{h'}_1}),\ldots,
 (p'_{k'}, {q^{h'}_{k'}}))\in \CT_a(\varphi,\chi, q).
 \end{equation}
There exists a {\it descent sequence} of triples  $(\varphi_i',\chi_i', {q^{(i)'}})$, $i=0,\ldots, k'-1$, where $(\varphi_0',\chi_0', {q^{(0)'}})=(\varphi,\chi, q)$ and for $i=1,\ldots, k'-1$ it holds that
 	\begin{itemize}
 	\item $(\varphi_i', \chi_i')\in \FD_{p_i'}(\varphi_{i-1}', \chi_{i-1}')$,
 	\item
 	$(p'_i, {q^{h'}_i})\star \CT_{a}(\varphi_{i-1}', \chi_{i-1}', {q^{(i-1)'}})\subset \CT_{a}(\varphi_i', \chi_i', {q^{(i)'}}).
 	$
 	\end{itemize}
Let $(\varphi_i, \chi_i, q^{(i)})$, $i=0,\ldots, k-1$ be a first descent sequence of triples as in \eqref{descentseq}. Then in order to show Theorem \ref{thm:main}, it is sufficient to show that $k'\geq k$ and 
\begin{equation} \label{ineq}
   \dim \varphi_i' \geq \dim\varphi_i,\quad i=1,\ldots, k-1,
\end{equation}
because we have
\[
\dim {q^{(i)'}}=\Fn-\sum_{j=1}^{i}p'_j\quad \textrm{and}
\quad \dim q^{(i)}=\Fn-\sum_{j=1}^{i}p_j.
\]
Recall the enhanced $L$-parameters  $(\varphi_i,\chi_i)$ in the inductive definition 
of $\Fs_a(\varphi,\chi,q)$. We will prove the following statement (Lemma \ref{lem8.5}) in the next section, which implies \eqref{ineq} and thereby proves  Theorem \ref{thm:main}.

\begin{lem}\label{lem8.5}
Let $\Fs' \in \CT_a(\varphi,\chi,q)$ and $(\varphi_i', \chi_i', {q^{(i)'}})$ be a descent sequence of triples as above. 
Then for each $1\leq i\leq k-1$ there exists a summand $\varphi_i''$ of $\varphi_i'$ such that
\begin{itemize}
\item
$\varphi_i''$ and $\varphi_i'$ are of the same type, and $\mathrm{rank}_{\BZ_2}\CS_{\varphi_i''} =\mathrm{rank}_{\BZ_2}\CS_{\varphi_i}$;
\item
With $\chi_i '' :=\chi_i' \vert_{\CS_{\varphi_i''}}$ and $\iota_{\varphi_i, \varphi_i''}$ given by \eqref{ord-pre}, one has
\[
\chi_i'' \circ \iota_{\varphi_i, \varphi_i ''} \in \ScO_\CZ( \chi_i).	
\]
\end{itemize}
\end{lem}

\subsection{$F$-distinguished orbits}\label{ssec-conjSOS}

We are going to define and study $F$-distinguished nilpotent orbits (cf. \cite{N98, Hb12}) and prove Theorem 
\ref{cor:wf-max-unique-archimedean}. 

\begin{defn}[Distinguished $F$-rational nilpotent orbit] \label{defn:dist-orbit}
An $F$-rational nilpotent orbit $\CO\in \CN_F(\Fg_n)_\circ$ is called $F$-distinguished if it does not meet any $F$-rational proper Levi subalgebra of $\Fg_n$.
\end{defn}

Let us point out that an $F$-rational nilpotent orbit $\CO$ is $F$-distinguished does not imply that the partition associated to $\CO$ is distinguished in the sense of \cite{CM93}. 


\begin{prop} \label{prp:dist}
The $F$-rational nilpotent orbit $\CO_a(\varphi,\chi,q)\in \wf_a(\varphi,\chi,q)$ in Proposition~\ref{prp:partition} is $F$-distinguished.
\end{prop}

\begin{proof}
It is known that an $F$-rational nilpotent orbit $\CO\in \CN_F(\Fg_n)_\circ$ is distinguished if and only if the Levi component of the stabilizer of an element of $\CO$ in $G_n$ is anisotropic. 
Let $q^h_i$, $i=1,\ldots, k$ be the one-dimensional sesquilinear forms in Definition \ref{defn:maxtab}.
It is sufficient to show that if $p_1=p_2$, then $q_1^h \cong q_2^h$, i.e. $[\disc(q^h_1)] = [\disc(q^h_2)]\in\CZ$. 

We prove this case-by-case below, and we follow the notations in Definition \ref{defn:maxtab} and the proof of Proposition \ref{prp:partition}.

\quad

{\bf (i) $G_n^*=\SO^*_{2n}$:}  
If $p_1=p_2$, by \eqref{eq:p1-p2-SO-even}, we have that 
\[
\frac{\dim \varphi +\dim \varphi_2}{2}=2s_\chi+1,
\]
which implies that 
$
\disc(q)=\det(\varphi)=-\det(\varphi_2)=-\disc(q^{(2)}).
$
Since 
\[
\disc(q_1^h)=-\disc(q)\disc(q^{(1)}),\quad  \disc(q_2^h)=\disc(q^{(1)})\disc(q^{(2)}),
\]
we obtain that $\disc(q_1^h)=\disc(q_2^h)$.

\quad

{\bf (ii) $G_n^*=\SO^*_{2n+1}$:} 
if $p_1=p_2$, by \eqref{eq:p1-p2-SO-odd}, we have that $l_{s_\chi}=1$. By Proposition \ref{prop:fd}, if 
$(\varphi_1,\chi_1)\in \FD^z_{p_1}(\varphi,\chi)$, then  
\[
z(\chi_1) = \chi_1(\beta_{s_\chi})= z\cdot \chi_{\varphi_1,\varphi}(\beta_{s_\chi})=-z.
\]
By Definition \ref{LYT}, we have that 
\[
a\cdot z= - \disc(q_1^h),\quad a\cdot z(\chi_1) = \disc(q_2^h),
\]
which implies that  $\disc(q_1^h)=\disc(q_2^h)$.

\quad

{\bf (iii) $G_n^*=\Sp_{2n}$ or $\Mp_{2n}$:} 
If $p_1=p_2$, by \eqref{eq:p1-p2-sp-mp}, we have that $l_{s_\chi}=1$. By Proposition \ref{prop:fd}, we obtain that $(\varphi_1,\chi_1)\in\FD^{z(\chi)}_{p_1}(\varphi,\chi)$. Hence 
we have that 
\[
z(\chi_1) = \chi_1(\beta_{s_\chi}) = -z(\chi)\cdot \chi_{\varphi_1,\varphi}(\beta_{s_\chi}) = z(\chi).
\]
By Definition \ref{LYT}, we have that 
\[
\disc(q^h_1) = a\cdot z(\chi) = a\cdot z(\chi_1) = \disc(q^h_2).
\]

\quad

{\bf (iv) $G_n^*=\RU_\Fn^*$:} 
If $p_1=p_2$, by \eqref{eq:p1-p2-u} we have that $l_{s_\chi}=1$. Assume that 
\[
(\varphi_1,\chi_1)  \in \FD^{z(\chi)}_{p_1}(\varphi,\chi)
\]
as in Proposition \ref{prop:fd}. From \eqref{chi-a} and \eqref{chi_z},  $l_{s_\chi}=1$ implies that
\[
\wh{\chi_1}(-\beta_{s_\chi}) = (-1)^{p_1} \cdot z(\chi)\cdot \chi_{\wh{\varphi_1},\varphi}(-\beta_{s_\chi}) = z(\chi)\cdot(-1)^{\Fn-p_1-1},
\]
which gives that
\[
z(\chi_1) =\chi_1(\beta_{s_\chi}) = (-1)^{\Fn-p_1-1}\cdot \wh{\chi_1}(-\beta_{s_\chi}) =  z(\chi).
\]
By Definition \ref{LYT}, we have that
\[
[\disc(q^h_1)] = a\cdot z(\chi) = a\cdot z(\chi_1) = [\disc(q^h_2)].
\]
\end{proof}

By Theorem \ref{thm:main}, we obtain Theorem \ref{thm2}. 

\begin{thm}[$F$-Stable Structure of $\wf_a(\varphi,\chi,q)^\mx$] \label{cor:wf-max-unique-archimedean}
For any triple $(\varphi,\chi,q)\in \FT_a(G_n^*)$,
the following hold.
\begin{enumerate}
    \item \label{wfmax1}
    If $\CS_\varphi$ is trivial, then $\wf_a(\varphi,\chi,q)^\mx$ consists of all the $F$-rational regular nilpotent orbits in $\CN_F(\Fg_n)_\circ$.
    \item \label{wfmax2}
    If $\CS_\varphi$ is nontrivial, then
    \[
    \wf_a(\varphi,\chi,q)^\mx = \{\CO_a(\varphi,\chi,q)\},
    \]
    where $\CO_a(\varphi,\chi,q)$ is the unique $F$-rational nilpotent orbit  corresponding to the $L$-descent pre-tableau  
    $\Fs_a(\varphi,\chi,q)$ as in Definition \ref{defn:maxtab} and Proposition \ref{prp:partition}.
    \item \label{wfmax3}
    Every $F$-rational nilpotent orbit in $\wf_a(\varphi,\chi,q)^\mx$  is $F$-distinguished.
\end{enumerate}
\end{thm}

\begin{proof} Part \eqref{wfmax1} and Part \eqref{wfmax2} are consequences of Proposition \ref{exmp:generic}, the construction of $\Fs_a(\varphi,\chi,q)$ in Definition \ref{defn:maxtab},
and Theorem \ref{thm:main}. Part \eqref{wfmax3} follows from Proposition \ref{prp:dist} and the fact that regular nilpotent orbits are $F$-distinguished. 
\end{proof}



\section{Collapse Operations and Proof of Lemma \ref{lem8.5}}\label{sec-CO}

This section is devoted to the proof of Lemma \ref{lem8.5}. To this end, we develop a notion of collapse of enhanced $L$-parameters, and study its interaction with the descent of 
enhanced $L$-parameters as developed in Section \ref{ssec-LPD} in general and in Section 
\ref{ssec-FD} in particular. 

We set up some general notation for later use. For a generic $L$-parameter $\varphi$  of the form \eqref{par-unif}, which is written as 
\[
\varphi=\bigoplus_{i=1}^r m_i\rho_{\alpha_i}\oplus \varphi_{\rm quad} \oplus \xi \oplus {}^c\xi^\vee
\]
with $m_i>0$ for all $i=1,2,\dots,r$, we define its {\it total multiplicity} as 
\begin{equation} \label{total-mult}
\mult\!(\varphi) := \sum^r_{i=1}m_i.
\end{equation}
For $-\infty\leq \alpha<\beta\leq\infty$, put
\begin{equation} \label{interval}
\varphi(\alpha,\beta) := \bigoplus_{\alpha< \alpha_i < \beta} m_i \rho_{\alpha_i},
\end{equation}
so that 
\begin{equation}\label{eq:mult-interval}
\mult\!(\varphi(\alpha,\beta))= \sum_{\alpha< \alpha_i < \beta} m_i.
\end{equation}

\subsection{Definitions and results}

As a preparation, we first generalize the setting of Lemma \ref{pre-tableau-eq} as follows.   

\begin{defn}[$\CZ$-equivalent enhanced $L$-parameters] \label{defn:ZELP}
Assume that $(\varphi,\chi)$ and $(\varphi'', \chi'')$ are respectively generic enhanced $L$-parameters of $G_n^* $ and $G_{n''}^*$ of the same type, and that
\[
\varphi=\bigoplus_{i=1}^r m_i\rho_{\alpha_i}\oplus \varphi_{\rm quad} \oplus \xi \oplus {}^c\xi^\vee \quad{\rm and}\quad
\varphi''=\bigoplus_{i=1}^{r''} m_i''\rho_{\alpha_i''}\oplus \varphi''_{\rm quad} \oplus \xi'' \oplus {}^c {\xi''}^\vee
\]
are of the form \eqref{par-unif}. Then  $(\varphi,\chi)$ and $(\varphi'',\chi'')$ are called  {\rm $\CZ$-equivalent} if $r=r''$, $m_i\equiv m_i''\mod 2$, $i=1,\ldots, r$, and 
\[
\chi'' \circ \iota_{\varphi, \varphi''}\in \ScO_\CZ(\chi),
\]
with $\iota_{\varphi,\varphi''}$ given by \eqref{ord-pre}.
\end{defn}

Note that the $\CZ$-equivalence is clearly an equivalence relation on enhanced $L$-parameters. In particular, we have the following immediate  consequence of Proposition \ref{prop:fd}.

\begin{lem} \label{lem:ZFD}
If $(\varphi,\chi)$ and $(\varphi'',\chi'')$ are $\CZ$-equivalent, then the enhanced $L$-parameters occurring in the first descent of $(\varphi,\chi)$ and in that of $(\varphi'',\chi'')$, 
respectively, are all $\CZ$-equivalent. 
\end{lem}

Let $\varphi'\in\wt{\Phi}_\gen(G_{n'}'^*)$ and $(\varphi', \chi')$ be an enhanced $L$-parameter, where 
\begin{equation} \label{varphi'}
\varphi'=\bigoplus_{i=1}^{r'} m_i'\rho_{\alpha_i'}\oplus \varphi'_{\rm quad} \oplus \xi' \oplus {}^c {\xi'}^\vee
\end{equation} 
is of the form \eqref{par-unif}.
For a summand $\varphi''\subset \varphi'$, define the restriction of $\chi'$ to $\varphi''$ to be
\[
\chi' \vert_{\varphi''}:=\chi'|_{\CS_{\varphi''}}.
\]
We say that $\chi'$ is {\it constant} on $\varphi''$ if it is constant valued on the natural basis vectors of $\CS_{\varphi''}$, that is,
\[
\chi'\vert_{\varphi''} =  \eta_{z} \in \CS_{\varphi''},\quad \textrm{for some }z\in\CZ.
\]

We introduce the notion of collapse operations and  collapse of enhanced $L$-parameters as follows, which is the main combinatorial ingredient in our proof.

\begin{defn}[Collapse of  enhanced $L$-parameters] \label{defn:CLP}
Let $(\varphi',\chi')$ be an enhanced $L$-parameter as above. An enhanced $L$-parameter $(\varphi'', \chi'')$ is called a {\rm collapse} of $(\varphi',\chi')$ if 
it can be obtained from $(\varphi',\chi')$ by an iteration of the following {\rm collapse operations}. 
\begin{enumerate}
	\item \label{C-even} An operation
	\[
	(\varphi',\chi')\mapsto (\varphi'\ominus m'_j\rho_{\alpha'_j},\chi'_{\varphi'\ominus m'_j\rho_{\alpha'_j}})
	\]
	for some $j$ with $m'_j$ even, where 
	\[
	\varphi'\ominus m'_j\rho_{\alpha'_j}:=
	\bigoplus_{i=1, i\neq j}^{r'} m_i'\rho_{\alpha_i'}\oplus \varphi'_{\rm quad} \oplus \xi' \oplus {}^c {\xi'}^\vee.
	\]
	\item \label{C-odd} An operation
	\[
	(\varphi',\chi')\mapsto 
	(\varphi'\ominus (m'_j\rho_{\alpha'_j}\oplus m'_{j+1}\rho_{\alpha'_{J+1}}),
	\chi'_{\varphi'\ominus (m'_j\rho_{\alpha'_j}\oplus m'_{j+1}\rho_{\alpha'_{j+1}})})
	\]
	for some $j$ with $\chi'(\alpha'_j)=\chi'(\alpha'_{j+1})$ and $m'_j+m'_{j+1}$ even.
		\item \label{C-R} An operation to remove $m'_{r'} \rho_{\alpha'_{r'}}$ for $r'>1$ $($and twist  $\varphi'_{\rm quad}$ by $\sigma_-^{m'_{r'}}$ if $G_{n'}'^*=\Sp_{2n'}$$)$, and restrict $\chi'$ to $\varphi'\ominus m'_{r'}\rho_{\alpha'_{r'}}$.
	\item \label{C-L} An operation to remove $m'_1\rho_{\alpha'_1}$ for $m'_1$ odd, $0\not\in \sgn(\chi')$ and $G_{n'}'^*\neq \SO_{2n'}^*$ $($and twist  $\varphi'_{\rm quad}$ by $\sigma_-^{m'_{1}}$ if $G_{n'}'^*=\Sp_{2n'}$$)$, and restrict $\chi'$ to $\varphi'\ominus m'_{1}\rho_{\alpha'_{1}}$. 
	\item \label{C-bp} An operation to remove $ \zeta\oplus{}^c\zeta^\vee$ for a summand $\zeta\subset\xi'$, or remove an even dimensional summand of $\varphi'_{\rm quad}$, and restrict $\chi'$ accordingly. 
	\end{enumerate}
\end{defn}

For example if we indicate $\chi'$ by  $(\chi'(\alpha_1'),\ldots, \chi'(\alpha_{r'}'))$, then 
$
(\sigma_2\oplus \sigma_4\oplus 3\sigma_6\oplus \sigma_-, (+,-,+))
$
gives an enhanced $L$-parameter of $\Sp_{10}$, and the collapse operation of type (3) gives $(\sigma_2\oplus \sigma_4\oplus\sigma_+, (+,-))$. 

\begin{rmk} We make the following comments about Definition \ref{defn:CLP}.
\begin{itemize}
\item In \eqref{C-R} and \eqref{C-L} above, we have used the fact that a twist of $\varphi'_{\rm quad}$ does not change $\CS_{\varphi'}$. 
\item The condition $0\not\in\sgn(\chi')$ in \eqref{C-L} only matters for $G_{n'}'^*=\SO^*_{2n'+1}$,  in which case it is equivalent to that $\chi'(\alpha'_1)=1$ $($see Definition \ref{defn:SAS}$)$. 
\item It is clear that 
 if $(\varphi'',\chi'')$ is a collapse of $(\varphi',\chi')$, then 
 \begin{equation} \label{C-chi}
 \chi''=\chi'|_{\varphi''}\quad and \quad 
 \#\sgn(\chi'') \leq \#\sgn(\chi').
 \end{equation}
 \end{itemize}
 \end{rmk}

Now assume that 
\[
(\varphi,\chi,q)\in \FT_a(G_n^*), \quad (\varphi', \chi', q')\in \FT_{a'}(G_{n'}'^*),
\]
 where $G_n^*$ and $G_{n'}'^*$ are of the same type, and that  $(\varphi',\chi')$ has a collapse $(\varphi'',\chi'')$ which is $\CZ$-equivalent to $(\varphi,\chi)$.  Note that $G_n^*$ and $G_{n'}'^*$ may not lie in the same Witt tower in the case of unitary groups and orthogonal groups. We say that the partition $\udl{p}(\varphi,\chi)$ in Definition \ref{defn:maxtab} is regular if it 
 corresponds to the $F$-stable regular nilpotent orbit in $\CN_F(\Fg_n^*)_\circ^\st$. 
 
 \begin{lem} \label{lem:C-generic}
 Under the above assumption, if $\udl{p}(\varphi',\chi')$ is regular, then $\udl{p}(\varphi,\chi)$ is regular.
 \end{lem}
 
 \begin{proof}
This follows from Definition \ref{defn:ZELP} and \eqref{C-chi}.
 \end{proof}

 From now on, we keep the running hypothesis \eqref{r>0} that 
 $ r = \rank_{\BZ_2}\CS_\varphi>0.$ 
We will prove the following main technical result.

\begin{lem}[Descent-Collapse] \label{lem:DC} 
Assume that $(\varphi_1,\chi_1)\in \FD_{p_1}(\varphi,\chi)$ with $p_1:=\Fl_0(\varphi,\chi)$, and that $(\varphi_1', \chi_1')\in \FD_{p_1'}(\varphi',\chi')$ for some $0<p_1'\leq \Fl_0(\varphi',\chi')$. Then  at least one of the following holds:
\begin{enumerate}
\item \label{lem:DC1} $(\varphi_1',\chi_1')$ has a collapse $(\varphi_1'', \chi_1'')$ which is $\CZ$-equivalent to $(\varphi_1, \chi_1)$;
\item \label{lem:DC2}  $G_{n'}'^*$ is special orthogonal and $\udl{p}(\varphi,\chi)$ is regular, in which case $\rank_{\BZ_2} \CS_{\varphi_1} =0$.
\end{enumerate}
\end{lem}

We picture the Case \eqref{lem:DC1} in the statement of Lemma \ref{lem:DC} as follows:
\begin{equation} \label{C-pic}
\xymatrix{
(\varphi',\chi') \ar[r]^C \ar[d]_D & (\varphi'', \chi'') \ar[r]^\CZ & (\varphi, \chi) \ar[d]^{FD} \\
(\varphi_1', \chi_1') \ar@{-->}[r]^C & (\varphi_1'', \chi_1'') \ar[r]^\CZ & (\varphi_1, \chi_1) 
}
\end{equation}
where the symbols $C$, $D$, $FD$ and $\CZ$ stands for collapse, descent, first descent and $\CZ$-equivalence, respectively. Although we use the same symbol $C$ for both collapses 
$C\colon (\varphi',\chi') \mapsto (\varphi'', \chi'')$ and 
$C\colon (\varphi_1', \chi_1') \mapsto (\varphi_1'', \chi_1'')$, they could be different collapses. 

Assuming Lemma \ref{lem:DC}, an easy  induction gives the following stronger result, which implies Lemma \ref{lem8.5} in the special case that $(\varphi',\chi', q')=(\varphi,\chi, q)\in \FT_a(G_n^*)$.

\begin{prop}
Let $(\varphi_i, \chi_i, q^{(i)})$, $i=0,\ldots, k-1$ be the first descent sequence of $(\varphi,\chi, q)$ given as in \eqref{descentseq}. Let 
$(\varphi_i',\chi_i', {q^{(i)'}})$, $i=0,\ldots, k'-1$, be any descent sequence of $(\varphi',\chi', q')$.  Then for each $1\leq i\leq k-1$, $\varphi_i'$ has a summand $\varphi_i''$ of the same type, such that $(\varphi_i'',\chi_i'')$ with $\chi_i'':=\chi'\vert_{\varphi_i''}$ is $\CZ$-equivalent
to $(\varphi_i, \chi_i)$. 
\end{prop}

Thus it remains to prove Lemma \ref{lem:DC}, which will be given in the rest of this section. 

\subsection{Descent versus collapse} \label{ssec-DVC}

Assume that $\varphi'$ is of the form \eqref{varphi'}. Set 
\[
\alpha_0' := \begin{cases} -\infty, & \textrm{if }E=\BC, \\
-1, & \textrm{if }E=\BR,
\end{cases} \quad \textrm{and}\quad   \alpha_{r'+1}' :=\infty.
\]
For a descent $(\varphi_1', \chi_1')$ of $(\varphi', \chi')$, we decompose  $\varphi_1'$ as
\begin{equation} \label{varphi1'-decomp}
\varphi_1' = \bigoplus^{r'}_{i=0} \psi_i \oplus \varphi'_{\rm 1, quad} \oplus \xi_1' \oplus {}^c\xi_1'^\vee,
\end{equation}
where $\xi_1'$ has no irreducible (conjugate) self-dual summands of the same type as $\varphi_1'$, and
\begin{equation} \label{psi-i}
\psi_i:= \varphi_1'(\alpha'_i, \alpha'_{i+1}),
\end{equation}
defined following \eqref{interval}.
Recall from Definition \ref{defn:pd} that 
\[
\chi_{\varphi', \wh{\varphi_1'}}\in \ScO_\CZ(\chi') \quad \textrm{and} \quad  \chi_{\wh{\varphi_1'}, \varphi'} \in \ScO_\CZ(\chi_1').
\]  
As before, label the indices in $\sgn(\chi')$ as $i_1,\ldots, i_{s_{\chi'}}$, where $s_{\chi'} =\#\sgn(\chi')$, such that 
\[
0\leq i_1 <\cdots < i_{s_{\chi'}} <r'.
\]
The formulas of distinguished characters for $\chi_{\varphi', \wh{\varphi_1'}}$ in Section \ref{sec-ADA} give us the following observation. 

\begin{lem} \label{C-parity} Let the definitions be as above. 
\begin{enumerate}
\item For each $i=1,\ldots, r'-1$, ${\rm mult}(\psi_i)$ is odd if $i\in \sgn(\chi')$, and is even otherwise. 
\item If $i_1 = 0\in \sgn(\chi')$ $($so that $G_{n'}'^* = \SO_{2n'+1}^*$$)$, then ${\rm mult}(\psi_0)$ is odd. 
\end{enumerate}
\end{lem}

Recall that the total multiplicity $\mult\!(\psi_i)$ is defined as in \eqref{total-mult}.
By Lemma \ref{C-parity}, $\mult\!(\psi_{i_j})$ is odd for each $i_j\in \sgn(\chi')$. We choose an isotypic summand 
\begin{equation} \label{psi-i-summand}
 n_j \rho_{\beta_j} \subset \psi_{i_j}
\end{equation}
with $n_j$ odd, such that the total multiplicities of 
\[
\psi_{i_j}^-:=\psi_{i_j}(-\infty, \beta_j)\quad \textrm{and} \quad \psi_{i_j}^+:=\psi_{i_j}(\beta_j, \infty)
\]
are both even. Clearly, such an isotypic summand always exists. With above choices,  define a tempered $L$-parameter
\begin{equation} \label{C-varphi}
\overline{\varphi_1''} := \bigoplus^{s_{\chi'}}_{j=1}n_j \rho_{\beta_j} \oplus  (\overline{\varphi''_1})_{\rm quad},\quad \overline{\chi_1''}:=\chi_1' \vert_{\overline{\varphi_1''}},
\end{equation} 
where 
\[
(\overline{\varphi''_1})_{\rm quad} =\begin{cases} \sigma_-^{s_{\chi'}}, & \textrm{if }G_{n'}'^*=\Mp_{2n'}, \\
0, & \textrm{otherwise}.
\end{cases}
\]
We further define a discrete $L$-parameter
\[
\overline{\phi_1} : = \bigoplus^{s_{\chi'}}_{j=1}  \rho_{\beta_j} \oplus (\overline{\varphi''_1})_{\rm quad}  \subset \overline{\varphi_1''}.
\]
Note that $\CS_{\overline{\varphi_1''}}\cong \CS_{\overline{\phi_1}}$, and by abuse of notation we also write 
\[
\overline{\chi_1''} = \chi_1'\vert_{\overline{\phi_1}}.
\]
Since $n_j$ is odd, $j=1,\ldots, s_{\chi'}$, it is clear that $(\overline{\varphi_1''},\overline{\chi_1''})$ is $\CZ$-equivalent to $(\overline{\phi_1}, \overline{\chi_1''})$.

By the formulas of distinguished characters for $\chi_{\wh{\varphi_1'}, \varphi'}$, we have following  result.

\begin{lem} \label{keylem}
Let the definitions be as above. 
\begin{enumerate}
\item \label{keylem-1} For each $i=0,\ldots, r'$,  $\chi_1'$ is constant on $\psi_i$.

\item  \label{keylem-2} For $j=1,\ldots, s_{\chi}-1$,
\[
\chi_1' (\beta_j)\chi_1' (\beta_{j+1}) = (-1)^{l'_j}.
\] 
\item \label{keylem-3} If $G_{n'}'^* =\SO_{2n'}^*$, then 
\[
\chi_1'(\beta_1) = (-1)^{l'_0}.
\]
\end{enumerate}
Here for $j=1, \ldots, s_{\chi'}-1$, we put  $($cf. \eqref{eq:l_j}$)$ that $l'_j: = m'_{i_j+1}  + \cdots + m'_{i_{j+1}}$, 
with $i_0:=0$. 
\end{lem}

Using the above lemmas, we can prove the following.

\begin{prop} \label{prop:C-fd} Let the definitions be as above. 
\begin{enumerate}
\item  At least one of the following holds:
\begin{itemize}
\item 
$(\overline{\varphi_1''}, \overline{\chi_1''})$ is a collapse of $(\varphi_1', \chi_1')$;
\item
 $G_{n'}'^*$ is special orthogonal and $\udl{p}(\varphi',\chi')$ is regular. 
\end{itemize}
\item $(\overline{\phi_1}, \overline{\chi_1''})$ is $\CZ$-equivalent to the  first descents of $(\varphi',\chi')$.
\end{enumerate}
\end{prop}

In other words, either  $G_{n'}'^*$ is special orthogonal and $\udl{p}(\varphi',\chi')$ is regular, or every descent $(\varphi_1', \chi_1')$ of $(\varphi',\chi')$ has a collapse which is $\CZ$-equivalent to the first descents of $(\varphi',\chi')$, in which case  we have a picture
\[
\xymatrix{
(\varphi', \chi') \ar[d]_D  \ar[rr]^{FD} & & \FD_{\Fl_0'}(\varphi',\chi') \ar[d]^\CZ \\
(\varphi_1', \chi_1') \ar[r]^C &  (\overline{\varphi_1''},\overline{\chi_1''}) \ar[r]^\CZ & (\overline{\phi_1}, \overline{\chi_1''})
}
\]
where $\Fl_0':=\Fl_0(\varphi_1',\chi_1')$.

\begin{proof}
The $L$-parameter $\overline{\varphi_1''}$ can be obtained from $\varphi_1'$  by the following operations.

\begin{itemize}
\item Remove $\psi_i$ with $0<i<r'$, $i\not\in \sgn(\chi')$. By Lemma \ref{C-parity} and Lemma \ref{keylem} \eqref{keylem-1}, $\mult\!(\psi_i)$ is even, and $\chi_1'$ is constant on $\psi_i$. Hence this is a composition of collapse  operations of type \eqref{C-even} and \eqref{C-odd}.

\item Remove $\psi_{i_j}^+$ and $\psi_{i_j}^-$, with $i_j \in \sgn(\chi')$.  This is similar as above, noting that $\mult\!(\psi_{i_j}^\pm)$ is even and $\chi_1'$ is constant on $\psi_{i_j}^\pm$. 

\item Remove  $\psi_{r'}$ (and twist $\varphi'_{\rm 1, quad}$ accordingly if $G_{n'}'^*=\Mp_{2n'}$). This is a composition of collapse operations of type \eqref{C-R}, with possible exceptions only when $G_{n'}'^*$ is special orthogonal and $\udl{p}(\varphi',\chi')$ is regular.

\item Remove $\psi_0$  if $0\not\in \sgn(\chi')$ (and twist $\varphi'_{\rm 1, quad}$ accordingly if $G_{n'}'^*=\Mp_{2n'}$). If $G_{n'}'^*\neq \SO_{2n'+1}^*$, then this is a composition of collapse operations of type  \eqref{C-L}; otherwise, it is a composition of collapse operations of type \eqref{C-even} and \eqref{C-odd}.

\item Remove $\xi_1' \oplus {}^c\xi_1'^\vee$, and a maximal even dimensional summand of $\varphi'_{\rm 1, quad}$ (or its twist). This is a composition of  collapse operations of type \eqref{C-bp}.
\end{itemize} 
This proves the first assertion. The second assertion follows from Proposition \ref{prop:fd} and Lemma \ref{keylem}. 
\end{proof}

\subsection{Proof of Lemma \ref{lem:DC}} 

To prove Lemma \ref{lem:DC}, we need to accomplish one of the following:
\begin{itemize} 
\item Complete the picture \eqref{C-pic};
\item Reduce to the case that $\udl{p}(\varphi,\chi)$ is regular.
\end{itemize}
To this end, it suffices to assume that $(\varphi'', \chi'')$ is a collapse of $(\varphi', \chi')$ by a single collapse operation. 

If  $G_{n'}'^*$ is special orthogonal and $\udl{p}(\varphi',\chi')$ is regular, then $\udl{p}(\varphi,\chi)$ is regular by Lemma \ref{lem:C-generic}, in which case we are done. Hence by Proposition \ref{prop:C-fd} we can assume that 
$(\overline{\varphi_1''}, \overline{\chi_1''})$ given by \eqref{C-varphi} is a collapse  of $(\varphi_1',\chi_1')$.  To complete \eqref{C-pic} we only need to further construct a collapse $(\varphi_1'', \chi_1'')$ of $(\overline{\varphi_1''}, \overline{\chi_1''})$ that is $\CZ$-equivalent to the first descent $(\varphi_1,\chi_1)$ of $(\varphi,\chi)$. 
That is, we need to complete the following picture
\[
\xymatrix{
(\varphi',\chi') \ar[rr]^C  \ar[d]_D & &  (\varphi'', \chi'') \ar[r]^\CZ &  (\varphi, \chi) \ar[d]^{FD}  \\
(\varphi_1', \chi') \ar[r]^C & (\overline{\varphi_1''}, \overline{\chi_1''}) \ar@{-->}[r]^C & (\varphi_1'', \chi_1'') \ar[r]^\CZ &  (\varphi_1,\chi_1)
}
\]
This amounts to the construction of a collapse $(\varphi_1'',\chi_1'')$ of 
$(\overline{\varphi_1''}, \overline{\chi_1''})$ which is $\CZ$-equivalent to the first descents 
$\FD_{p_1}(\varphi'',\chi'')$, where $p_1=\Fl_0(\varphi,\chi) =\Fl_0(\varphi'', \chi'')$. In that case it will finish the proof of Lemma \ref{lem:DC} in view of Lemma \ref{lem:ZFD}; otherwise we have to reduce to the case that
 $G_{n'}'^*$ is special orthogonal and $\udl{p}(\varphi,\chi)$ is regular. 
We can picture this as follows
\[
\xymatrix{
 & &  (\varphi'', \chi'')\ar[d]_{FD} \ar[r]^\CZ &  (\varphi, \chi) \ar[d]^{FD}  \\
 (\overline{\varphi_1''}, \overline{\chi_1''}) \ar@{-->}[r]^C & (\varphi_1'', \chi_1'') \ar[r]^\CZ &  \FD_{p_1}(\varphi'',\chi'') \ar[r]^\CZ &  (\varphi_1,\chi_1)
}
\]
We discuss the different types of collapse operations in Definition \ref{defn:CLP} separately. 

\subsubsection{} \label{ssec-C1} $\varphi'' = \varphi' \ominus m'_i\rho_{\alpha_i'}$ with $m_i'$ even.

\

By the hypothesis \eqref{r>0} we have $r=r'-1>0$, hence $r'>1$. It follows that the case $i=r'$ is a collapse operation of type \eqref{C-R}, which will be discussed in \ref{ssec-C3}. We now assume that $i<r'$. 

If $i\not\in \sgn(\chi')$, then by Proposition \ref{prop:fd}, the first descents of $(\varphi',\chi')$ and $(\varphi'', \chi'')$ are $\CZ$-equivalent, thus $(\varphi_1'',\chi_1'') := (\overline{\varphi_1''}, \overline{\chi_1''})$ is $\CZ$-equivalent to $\FD_{p_1}(\varphi'',\chi'') $ by Proposition \ref{prop:C-fd}.  

If $i=i_j \in\sgn(\chi')$, then there are two cases:
\begin{itemize}
\item $i-1\not\in \sgn(\chi')$. In this case, again we can take $(\varphi_1'',\chi_1'') := (\overline{\varphi_1''}, \overline{\chi_1''})$;
\item  $i-1 = i_{j-1} \in\sgn(\chi')$. Define
\[
\varphi_1'' :=\overline{\varphi_1''} \ominus n_{j-1}\rho_{\beta_{j-1}}\ominus n_j \rho_{\beta_j},\quad \chi_1'' := \overline{\chi_1''}\vert_{\varphi_1''}.
\]
By Lemma \ref{keylem} \eqref{keylem-2}, we have 
\[
\chi_1'(\beta_{j-1}) \chi_1'(\beta_j) = (-1)^{m_i'} =1.
\]
Thus $(\varphi_1'',\chi_1'')$ is a collapse of $(\overline{\varphi_1''}, \overline{\chi_1''})$, noting that the $n_j$'s are odd. Using Lemma \ref{keylem} again,  it is easy to verify that $(\varphi_1'',\chi_1'')$ is $\CZ$-equivalent to $\FD_{p_1}(\varphi'',\chi'')$. 
\end{itemize}

\subsubsection{} \label{ssec-C2}  $\varphi''=\varphi' \ominus m'_i\rho_{\alpha_i}\ominus m'_{i+1}\rho_{\alpha_{i+1}}$ with $\chi'(\alpha'_i)=\chi'(\alpha'_{i+1})$ and $m'_i+m'_{i+1}$ even. 

\

By the hypothesis \eqref{r>0} we have $r'>2$. Hence the case $i+1=r'$ is a composition of two collapse operations of type \eqref{C-R}. We now assume that $i+1<r'$. The construction is similar to \ref{ssec-C2}. We first note that $i\not\in \sgn(\chi')$.  

If $i+1\not\in\sgn(\chi')$, then we can take $(\varphi_1'',\chi_1'') := (\overline{\varphi_1''}, \overline{\chi_1''})$. 

If $i+1 = i_j \in \sgn(\chi')$, then there are two cases:

\begin{itemize}
\item $i-1\not\in \sgn(\chi')$. In this case  we can take $(\varphi_1'',\chi_1'') := (\overline{\varphi_1''}, \overline{\chi_1''})$;
\item  $i-1 = i_{j-1} \in\sgn(\chi')$. Define
\[
\varphi_1'':=\overline{\varphi_1''} \ominus n_{j-1}\rho_{\beta_{j-1}}\ominus n_j \rho_{\beta_j},\quad \chi_1'' := \overline{\chi_1''}\vert_{\varphi_1''}.
\]
By Lemma \ref{keylem} \eqref{keylem-2}, we have 
\[
\chi_1'(\beta_{j-1}) \chi_1'(\beta_j) = (-1)^{m_i'+m_{i+1}'} =1.
\]
Thus $(\varphi_1'',\chi_1'')$ is a collapse of $(\overline{\varphi_1''}, \overline{\chi_1''})$, and $(\varphi_1'',\chi_1'')$ is $\CZ$-equivalent to $\FD_{p_1}(\varphi'',\chi'')$. 
\end{itemize}

\subsubsection{} \label{ssec-C3}  $\varphi''=\varphi'\ominus m'_{r'}\rho_{\alpha'_{r'}}$ with $r'>1$ (and $\varphi'_{\rm quad}$ is twisted by $\sigma_-^{m_r'}$ if $G_{n'}'^*=\Sp_{2n'}$).

\

If $r'-1\not\in\sgn(\chi')$, then we can take $(\varphi_1'',\chi_1'') := (\overline{\varphi_1''}, \overline{\chi_1''})$. 

If $r' -1 = i_{s_{\chi'}} \in \sgn(\chi')$, define
\[
\varphi_1'':=\begin{cases}  \overline{\varphi_1''} \ominus  n_{s_{\chi'}}\rho_{\beta_{s_{\chi'}}}, & \textrm{if }G_{n'}'^*\neq \Mp_{2n'}, \\
 \left(\overline{\varphi_1''} \ominus  n_{s_{\chi'}}\rho_{\beta_{s_{\chi'}}}\right)\otimes \sigma_-, & \textrm{if }G_{n'}'^* =  \Mp_{2n'},
 \end{cases}
\quad  \quad \chi_1'' := \overline{\chi_1''}\vert_{\varphi_1''}.
\] 
Then $(\varphi_1'',\chi_1'')$ is a collapse of $(\overline{\varphi_1''}, \overline{\chi_1''})$ by a type \eqref{C-R} collapse operation and is $\CZ$-equivalent to $\FD_{p_1}(\varphi_1'',\chi'')$, unless 
\begin{equation} \label{DC-generic}
s_{\chi'}=1\quad \textrm{and}\quad G_{n'}'^*\textrm{ is special orthogonal}.
\end{equation}
However, in the latter case \eqref{DC-generic} it is easy to see that $\udl{p}(\varphi'',\chi'')$ hence $\udl{p}(\varphi,\chi)$ is regular.

\subsubsection{} \label{ssec-C4} $\varphi''=\varphi'\ominus m'_{1}\rho_{\alpha'_1}$ with $m_1'$ odd, $0\not\in\sgn(\chi')$, and $G_{n'}'^*\neq \SO_{2n'}^*$ (and $\varphi'_{\rm quad}$ is twisted by $\sigma_-^{m_1'}$ if $G_{n'}'^*=\Sp_{2n'}$).

\

If $1\not\in\sgn(\chi')$, then we can take $(\varphi_1'',\chi_1'') := (\overline{\varphi_1''}, \overline{\chi_1''})$.

If $1\in\sgn(\chi')$, then there are two cases:

\begin{itemize}
\item $G_{n'}'^* = \SO_{2n'+1}^*$. By assumption we have $\chi'(\alpha_1')=1$ and $\chi'(\alpha_2')=-1$. By Proposition \ref{prop:fd}, the first descents of $(\varphi',\chi')$ and $(\varphi'', \chi'')$ are $\CZ$-equivalent, and we can take $(\varphi_1'',\chi_1'') := (\overline{\varphi_1''}, \overline{\chi_1''})$.
\item $G_{n'}'^*$ is not special orthogonal. Define  
\[
\varphi_1'':=
\begin{cases} 
\overline{\varphi_1''} \ominus  n_{1}\rho_{\beta_1}, & \textrm{if }G_{n'}'^*\neq \Mp_{2n'}, \\
 \left(\overline{\varphi_1''} \ominus  n_{1}\rho_{\beta_1}\right)\otimes \sigma_-, & \textrm{if }G_{n'}'^* =  \Mp_{2n'},
 \end{cases}
\quad  \quad \chi_1'' := \overline{\chi_1''}\vert_{\varphi_1''}.
\] 
Then $(\varphi_1'', \chi_1'')$ is a collapse of $(\overline{\varphi_1''}, \overline{\chi_1''})$ by a collapse operation of type \eqref{C-L}, and is $\CZ$-equivalent to $\FD_{p_1}(\varphi'',\chi'')$.
\end{itemize}

\subsubsection{}  \label{ssec-C5} $\varphi''$ is obtained from $\varphi'$ by a collapse operation of type \eqref{C-bp}. In this case $\CS_{\varphi''}\cong \CS_{\varphi'}$, and  $(\varphi_1'',\chi_1''):=(\overline{\varphi_1''}, \overline{\chi_1''})$ is $\CZ$-equivalent to $\FD_{p_1}(\varphi'',\chi'')$. 

\quad

This proves Lemma \ref{lem:DC}, hence finishes the proof of Lemma \ref{lem8.5}.

\section{$F$-rational Maximality of $\CO_a(\varphi,\chi,q)$}\label{sec-FRM}


In this section we continue to work with the archimedean case so that $F=\BR$. In Theorem \ref{cor:wf-max-unique-archimedean}, we show that 
for any enhanced $L$-parameter $(\varphi,\chi)$ with $\varphi$ generic and the component 
group $\CS_\varphi$ non-trivial, and for a given $a\in\CZ$, the unique $F$-rational nilpotent orbit $\CO_a(\varphi,\chi,q)$ associated with 
the $L$-descent pre-tableau $\Fs_a(\varphi,\chi,q)$ is the maximal member in the 
wavefront set $\wf_a(\varphi,\chi,q)$ under the $F$-stable topological order over 
$\CN_F(\Fg_n)_\circ$. 
In this section (Theorem \ref{thm:maximal-rational-order}), we will show that $\CO_a(\varphi,\chi,q)$ is the maximal member 
in the 
wavefront set $\wf_a(\varphi,\chi,q)$ under the $F$-rational topological order over 
$\CN_F(\Fg_n)_\circ$, which is equivalent to that 
the associated $L$-descent pre-tableau $\Fs_a(\varphi,\chi,q)$ defined in Definition \ref{defn:maxtab}  is maximal in the partial order as introduced in \cite{D82}.

We recall the results for the closures of nilpotent orbits of real classical groups from \cite{D81,D82}, which give the $F$-rational topological order in Definition \ref{defn:order}.
In \cite[Theorem 5]{D82}, Djokovi\'c gave an equivalent partial order on their corresponding signed Young diagrams (see Chapter 9 in \cite{CM93}). 
Note that by Proposition \ref{prop:Lp} the $L$-descent partitions in $\CP_a(\varphi,\chi,q)$ are of {\it good parity}.
That is, all parts of such partitions are even  (resp. odd) for $G^*_n=\Sp_{2n}$ or $G^*_{n}=\Mp_{2n}$ (resp. $G^*_{n}=\SO_{\Fn}$), and there is no constraint for $G^*_n=\RU_\Fn$.
Hence, for the purpose of this section we only recall Djokovi\'c's partial order on the nilpotent orbits whose corresponding partitions are of good parity. 
For the partial order over arbitrary nilpotent orbits, one may refer to \cite{D81,D82}.

\begin{defn}[Ordered Signed Young Diagram]\label{defn:OSYD}
An {\it ordered} signed Young diagram is a collection of boxes arranged in left-aligned rows, where each box is filled with a plus or minus sign, such that the signs are alternating across each row.
\end{defn}

Note that in contrast to the usual definition of {\it signed Young diagram} (see \cite{CM93}), we {\it do not require} that the number of boxes in each row is decreasing from top to bottom in an ordered signed Young diagram. We may identify  pre-tableaux with ordered signed Young diagrams in the obvious way. More precisely, a pre-tableau 
 \[
 \Fs= ((p_1, q^h_1), \ldots, (p_k, q^h_k))
 \]
 assigns an ordered signed Young diagram whose $i$-th row has length $p_i$ and ends with the sign of 
 $[\disc(q^h_i)]\in \CZ=\{\pm1\}$, $i=1,\ldots, k$. In the rest of this section we are going to use the same $\Fs$ to denote the ordered signed Young diagram associated with the pre-tableau $\Fs$. 
 
 For ordered signed Young diagrams $\Fs$ and $\Fs'$, 
define a partial order $\Fs'\leq_F \Fs$ by
\begin{equation}\label{ineq:rational-order}
\sum_{i=1}^{j}c^+_i(\Fs')\geq \sum_{i=1}^{j}c^+_i(\Fs)\quad \text{ and }
\quad \sum_{i=1}^{j}c^-_i(\Fs')\geq \sum_{i=1}^{j}c^-_i(\Fs) \qquad  \text{ for all }  j, 
\end{equation}
where $c^{\pm}_i(\Fs)$ is the number of $\pm$-boxes in the $i$-th column of $\Fs$.
This extends the partial order introduced in \cite{D82}\footnote{\cite{D81, D82} use the terminology 
``chromosome", which is equivalent to the more commonly used notion of signed Young diagram (cf. \cite{CM93}).} for signed Young diagrams. 

In order to prove Theorem \ref{thm:maximal-rational-order} below, we further introduce a partial order  $\Fs' \leq_{p} \Fs$ for ordered signed Young diagrams and a positive integer $p$ by 
\[
\Fs'\leq_{p} \Fs \text{ if \eqref{ineq:rational-order} holds for all $j\leq p$.}
\]
Then $\Fs'\leq_F \Fs$ is equivalent to that $\Fs'\leq_{p}\Fs$ for all $p$. 
If $\Fs$ and $\Fs'$ are signed Young diagrams corresponding to decreasing partitions $p_1\geq p_2\geq \cdots$ and $p'_1\geq p'_2\geq \cdots$  of $\Fn$ respectively, with  $p_1\geq p'_1$, then $\Fs'\leq_F \Fs$ is equivalent to that $\Fs'\leq_{p'_1} \Fs$.

Assume that $\Fs, \Fs' \in \CL_a(\varphi,\chi,q)$, which by definition are pre-tableaux corresponding to decreasing partitions of good parity in the sense above Definition \ref{defn:OSYD}. Let $\CO_{\Fs}$ and $\CO_{\Fs'}$ be the $F$-rational nilpotent orbits in $\CN_F(\Fg_n)_\circ$ corresponding to the signed Young diagrams $\Fs$ and $\Fs'$ respectively.
Then by \cite[Theorem 5]{D82}, we have that
\[
\CO_{\Fs'} \leq_F \CO_{\Fs} \quad \textrm{if and only if} \quad \Fs' \leq_F \Fs. 
\]

\begin{thm}[$F$-Rational Structure of $\wf_a(\varphi,\chi,q)^\mx$]\label{thm:maximal-rational-order}
Assume that $\CS_\varphi$ is nontrivial. For any $\Fs'\in \CL_a(\varphi,\chi,q)$, we have that 
\[\Fs'\leq_F \Fs_a(\varphi,\chi,q).
\]
In particular, the $F$-rational nilpotent orbit $\CO_a(\varphi,\chi,q)$ corresponding to $\Fs_a(\varphi,\chi,q)$ is the unique maximal member in $\wf_a(\varphi,\chi,q)$ under the $F$-rational topological order $\leq_F$.
\end{thm}

\begin{rmk}
Theorem \ref{thm:maximal-rational-order} implies Theorem \ref{thm1} rather than the stronger
Theorem \ref{thm:main}. 
More precisely, Theorem \ref{thm:main} asserts that for any $\Fs'\in \CT_a(\varphi,\chi,q)$ which corresponds to an ordered partition $\udl{p}'$, we have that 
\begin{equation}\label{rmk-partition-order}
\udl{p}'\leq \udl{p}(\varphi,\chi)
\end{equation}
under the partial order \eqref{part-order}. Since the length of the second row of $\Fs'$ can be even larger than that of the first row of $\Fs_a(\varphi,\chi,q)$,  in general $\Fs'$ is an ordered signed Young diagram and the above partial order \eqref{ineq:rational-order} does not imply \eqref{rmk-partition-order}.
\end{rmk}

We will prove Theorem \ref{thm:maximal-rational-order} for $E=\BR$ and $E=\BC$ separately. We first set up some general notations before giving the case-by-case proof. By Theorem 
\ref{inv-pi}, without loss of generality  we may fix $a=+1$ for the local Langlands correspondence in Definition \ref{LYT}. Assume that
\[
\varphi=\bigoplus_{i=1}^r m_i\rho_{\alpha_i}\oplus \varphi_{\rm quad} \oplus \xi \oplus {}^c\xi^\vee  
\]
is of the form \eqref{par-unif}.
For convenience, we put
\begin{equation} \label{LP-maxmin}
\max\varphi : = \alpha_r,\quad \min \varphi : = \alpha_1.
\end{equation}

Following \eqref{descentseq}, assume that
\[
 \Fs_a(\varphi, \chi, q) = ((p_1, q^h_1), \ldots,  (p_k, q^h_k) )\in \CL_a(\varphi,\chi,q)
\]
is associated to a first descent sequence  $(\varphi_i, \chi_i, q^{(i)})$, $i=0,\ldots, k-1$ of $(\varphi,\chi,q) =: (\varphi_0, \chi_0, q^{(0)})$, so that
\[
(p_{i}, q^h_{(p_{i})})\star \CT_{a}(\varphi_{i},\chi_{i}, q^{(i)}) \subset \CT_a(\varphi_{i-1}, \chi_{i-1}, q^{(i-1)}), \quad i=1,\ldots, k-1.
\]
We assume that 
\[
 (\varphi_i, \chi_i) \in \FD_{p_i}^{z_i}(\varphi_{i-1}, \chi_{i-1}), \quad z_i\in \CZ, \ i=1,\ldots, k-1,
\]
subject to the conditions introduced in Definition \ref{LYT}.
Similarly,  following \eqref{pretableau-s'} assume that 
\[
\Fs' = ((p_1', {q^{h'}_1}), \ldots,  (p'_{k'}, {q^{h'}_{k'}}))\in \CL_a(\varphi,\chi,q)
\]
 is associated to a descent sequence $(\varphi'_i, \chi'_i, {q^{(i)'}})$, $i=0,\ldots, k'-1$ of $(\varphi, \chi, q) =:  (\varphi_0', \chi_0', {q^{(0)'}})$, so that
 \[
 (p'_{i}, q^{h'}_{(p'_{i})})\star \CT_{a}(\varphi'_{i},\chi'_{i}, q^{(i)'})\subset \CT_{a}(\varphi'_{i-1},\chi'_{i-1}, q^{(i-1)'}),\quad i=1,\ldots, k'-1.
 \]
 We also assume that
 \begin{equation}\label{aizi'}
 (\varphi_i', \chi_i') \in \FD_{p_i'}^{z_i'}(\varphi'_{i-1}, \chi'_{i-1}), 
 \quad z'_i\in \CZ, \ i=1,\ldots, k'-1,
\end{equation}
subject to the conditions in Definition \ref{LYT}.
 Note that $p_1\geq\cdots\geq p_k$ and $p_1'\geq\cdots\geq p_k'$ are decreasing. We will prove  Theorem \ref{thm:maximal-rational-order}, that is,
 \[
 \Fs'\leq_F \Fs_a(\varphi,\chi,q),
 \]
 by induction on $\dim\varphi$.

We recall some constructions from Section \ref{ssec-DVC}, applied to the situation that $\varphi'=\varphi$. Similar to \eqref{varphi1'-decomp}, decompose $\varphi_1'$ as 
\begin{equation}  \label{varphi1'-decomp-2}
\varphi_1' = \bigoplus^r_{i=0}  \psi_i \oplus \varphi'_{\rm 1, quad} \oplus \xi_1'\oplus {}^c\xi_1'^\vee,
\end{equation} 
where 
\[
\psi_i = \varphi_1'(\alpha_i, \alpha_{i+1}),\quad i=0,\ldots, r,
\]
with 
\[
\alpha_0=
\begin{cases}
-1, & \textrm{if }E=\BR, \\
-\infty, & \textrm{if }E=\BC,
\end{cases} \quad 
\textrm{and}\quad  \alpha_{r+1}=\infty.
\]
Similar to \eqref{psi-i-summand}, for each $i_j \in \sgn(\chi)$, $j=1,\ldots, s_\chi$, choose an isotypic summand 
\[
n_j \rho_{\beta_j} \subset \psi_{i_j}
\]
with $n_j$ odd such that the total multiplicities of  $\psi_{i_j}^- = \psi_{i_j}(-\infty, \beta_j)$ and $\psi_{i_j}^+ = \psi_{i_j}(\beta_j, \infty)$ are both even. 
By Propositions \ref{prop:fd} and \ref{prop:unique}, we may take the first descent $(\varphi_1,\chi_1)$  of $(\varphi,\chi)$ with
\begin{equation} \label{varphi1}
\varphi_1  =\bigoplus^{s_\chi}_{j=1}\rho_{\beta_j} \oplus \varphi_{\rm 1, quad},
\end{equation}
where
\[
\varphi_{\rm 1, quad} =
\begin{cases}  \sigma_-^{s_\chi}, & \textrm{if }G_n^* = \Mp_{2n}, \\
0, & \textrm{otherwise.} 
\end{cases}
\]

\subsection{The case $E=\BR$} In this subsection we prove Theorem \ref{thm:maximal-rational-order} when $E=\BR$, that is, for special orthogonal groups, 
symplectic groups and metaplectic groups. 

Define an enhanced tempered $L$-parameter $(\varphi_1^\natural, \chi_1^\natural)$
\begin{equation} \label{LP-natural}
\varphi_1^\natural := \varphi_1 \oplus \psi_r \oplus \xi_1\oplus \xi_1^\vee, \quad \chi_1^\natural:=\chi'_1\vert_{\varphi_1^\natural},
\end{equation} 
where $\xi_1$ is an arbitrary direct sum of non-self-dual unitary characters of $\CW_\BR$ satisfying that 
\[
\dim \varphi_1' = \dim \varphi_1^\natural,
\]
and $\chi'_1\vert_{\varphi_1^\natural}$ is defined via the obvious embedding $\CS_{\varphi_1^\natural} \subset \CS_{\varphi_1'}$. Clearly such $\xi_1$ always exists. 

Recall from \eqref{aizi'} that $(\varphi_1', \chi_1') \in \FD_{p_1'}^{z_1'}(\varphi,\chi)$.  We have the following result.

\begin{lem} \label{lem:natural}
Let the definitions be as above. Then  the following hold.
\begin{enumerate}
\item
$
(\varphi_1^\natural, \chi_1^\natural)  \in \FD^{z_1'}_{p_1'}(\varphi, \chi). 
$
\item If $G_n$ is special orthogonal, then $(\varphi_1^\natural, \chi_1^\natural, {q^{(1)}}')\in \FT_a(\SO_{\Fn-p_1'}^*)$.  That is, $\pi_a(\varphi_1',\chi_1')$ and
$\pi_a(\varphi_1^\natural, \chi_1^\natural)$ are representations of the same pure inner form $\SO({q^{(1)'}})$. 

\item A first descent of $(\varphi_1',\chi_1')$ is also a descent of $(\varphi_1^\natural, \chi_1^\natural)$.
\end{enumerate}
\end{lem}

\begin{proof}
(1) It suffices to show that 
\begin{equation} \label{natural-eq1}
\chi_{\varphi, \varphi_1^\natural} = \chi_{\varphi, \varphi_1'}
\end{equation} 
and
\begin{equation} \label{natural-eq2}
\chi_{\varphi_1^\natural, \varphi} = \chi_{\varphi_1', \varphi}\vert_{\varphi_1^\natural}.
\end{equation}

If $G_n^*\neq \SO_{2n+1}$, then by \eqref{chi1-R}, \eqref{chi1-R'} and \eqref{chi2-R'}, for $i=1,\ldots, r$, 
\[
\chi_{\varphi, \varphi_1'}(\alpha_i) = \prod^r_{k =i} (-1)^{\mult\!(\psi_k)} =  (-1)^{\#\{j\in [1,s_\chi] \ \mid \ i_j> i\}+\mult\!(\psi_r)} = \chi_{\varphi, \varphi_1^\natural}(\alpha_i),
\]
because  $\mult\!(\psi_k)$, $k=1,\ldots, r-1$, is even if $ k\not\in\sgn(\chi)$ and is odd otherwise.  Hence \eqref{natural-eq1} holds when $G_n^*\neq \SO^*_{2n+1}$. It can similarly proved using 
\eqref{chi2-R} when $G_n^*=\SO_{2n+1}^*$, which will be omitted. In view of the above mentioned formulas in Section \ref{sec-ADA}, \eqref{natural-eq2} is obvious. 

(2) By (1) and Proposition \ref{prop:rational} (2), the assertion is clear if $G_n^*=\SO^*_{2n}$. Assume that $G_n=\SO^*_{2n+1}$. By (1)
and Proposition \ref{prop:rational} (1),  we only need to show that $\det(\varphi'_1)= \det(\varphi_1^\natural)$, or equivalently that
\[
\mult\!(\varphi'_1) = \mult\!(\varphi_1^\natural).
\]
This follows from the fact that each $\mult\!(\psi_i)$, $i=0,\ldots, r-1$, is even if $i\not\in \sgn(\chi)$ and is odd otherwise. 

(3) The proof is similar to that of (1). Let $(\phi, \theta)$ be a first descent 
of $(\varphi_1', \chi_1')$.  It suffices to show that 
\begin{equation}\label{natural-eq3}
\chi_{\phi, \varphi_1^\natural}  = \chi_{\phi, \varphi_1'}
\end{equation}
and 
\begin{equation}\label{natural-eq4}
\chi_{\varphi_1^\natural, \phi} = \chi_{\varphi_1', \phi}\vert_{\varphi_1^\natural}.
\end{equation} 
The equation \eqref{natural-eq4} is again obvious.  
 Let
$\psi_{i_1'}, \ldots, \psi_{i'_{r'}}$ be all the nonzero summands among 
$\psi_0,\ldots, \psi_r$, where $0\leq i_1' < \cdots < i'_{r'} \leq r$.  Note that $\chi_1'$ is constant on each $\psi_i$  (see Section \ref{sec-CO}).  
Then by Proposition \ref{prop:fd}, we may 
decompose $\phi$ as 
\[
\phi= \bigoplus^{r'}_{k=0} \phi(\gamma'_k, \beta'_{k+1}) \oplus \phi_{\rm quad},
\]
where $\gamma'_0 := -1$, $\beta'_k := \min \psi_{i'_k}$,  $\gamma'_k := \max\psi_{i'_k}$, $k=1,\ldots, r'$ (cf. \eqref{LP-maxmin}), and $\gamma'_{r'+1}:=\infty$. 
Thus the proof of  \eqref{natural-eq3} is similar to that of \eqref{natural-eq1}.
\end{proof}

For convenience, we also introduce a tempered enhanced $L$-parameter 
$(\phi_1, \theta_1)$ of a simpler form by defining 
\begin{equation} \label{eq:phi1}
\phi_1 = \varphi_1 \oplus m \rho_{\beta_{s_\chi+1}},
\end{equation} 
where $\beta_{s_{\chi}+1}$ is chosen such that $\beta_{s_\chi+1} > \alpha_r$, $\beta_{s_\chi+1} \equiv \alpha_r+1\mod 2$, and $m :=\mult\!(\psi_r)$.
Recall that $\chi_1'$ is constant on $\psi_r$. The character $\theta_1 \in \wh{\CS_{\phi_1}}$ is defined 
such that $\theta_1\vert_{\varphi_1} = \chi'_1\vert_{\varphi_1}$,  and $\theta_1(\beta_{s_\chi+1})$ equals the constant value of $\chi'_1$ on the natural basis 
vectors of $\CS_{\psi_r}$ if $\psi_r\neq 0$. Put 
\begin{equation} \label{p_1''}
p_1'' := \Fn-\dim \phi_1 = p_1' +   \dim\varphi_1^\natural - \dim \phi_1\geq  p_1'.
\end{equation} 
We have the following result.

\begin{lem} \label{lem:phi1} Let the definitions be as above. Then the following hold.
\begin{enumerate}
\item $(\phi_1, \theta_1) \in \FD^{z_1'}_{p_1''}(\varphi,\chi)$.

\item If $G_n$ is special orthogonal, then $(\phi_1, \theta_1, {q^{(1)''}})\in \FT_a(\SO^*_{\Fn-p_1''})$, where ${q^{(1)''}}$ and ${q^{(1)'}}$ belong to the same Witt tower of orthogonal spaces.

\item $(\varphi_1^\natural, \chi_1^\natural)$ and $(\phi_1,\theta_1)$ have the same first descents. That is,
\[
\FD_{p_2^\natural}(\varphi_1^\natural, \chi_1^\natural) = \FD_{p_2''}(\phi_1, \theta_1),
\]
where 
\[
p_2^\natural := \Fl_0(\varphi_1^\natural, \chi_1^\natural) \geq p_2'' := \Fl_0(\phi_1, \theta_1).
\]
\end{enumerate}
\end{lem}

\begin{proof}
(1) and (2) are similar to Lemma \ref{lem:natural}. (3) follows easily from Proposition \ref{prop:fd}.
\end{proof}

If $\CS_{\varphi_1}$ is trivial, then  $\CO_a(\varphi, \chi, q)$ is regular because  $\varphi_1$ is discrete by Proposition \ref{prop:fd}.  In this case 
it is straightforward to verify that
\[
\Fs'\leq_F  \Fs_a(\varphi, \chi, q).
\]

In the rest of this subsection, assume that $\CS_{\varphi_1}$ is nontrivial. Then $\CS_{\varphi_1^\natural}$ and $\CS_{\phi_1}$ are nontrivial as well. Take a first descent sequence
\begin{equation}\label{phi-fd}
    \begin{matrix}
    (\phi_1,\theta_1, {q^{(1)''}}),&(\phi_2,\theta_2, {q^{(2)''}}),&\dots,&(\phi_{k''-1},\theta_{k''-1}, {q^{(k''-1)''}})\\
    (p_2'', {q^{h''}_2}),& (p_3'', {q^{h''}_3}), &\dots,& (p_{k''}'', {q^{h''}_{k''}})
    \end{matrix}
\end{equation}
of $(\phi_1,\theta_1, {q^{(1)''}})$, so that
\[
\Fs_a(\phi_1,\theta_1, {q^{(1)''}}) = ((p_2'', {q^{h''}_2}), \ldots, (p_{k''}'', {q^{h''}_{k''}})).
\]
Assume that
\[
(\phi_i, \theta_i) \in \FD_{p_i''}^{z_i''}(\phi_{i-1}, \theta_{i-1}), \quad z_i''\in\CZ, \ i=2,\ldots, k''-1.
\]

By Lemma \ref{lem:natural}, a first descent sequence of $(\varphi_1', \chi_1', q^{(1)'})$ is also a descent sequence of 
$(\varphi_1^\natural, \chi_1^\natural, q^{(1)'})$.  Applying the induction hypothesis for $\varphi_1'$ and 
$\varphi_1^\natural$ respectively,  we obtain that
\begin{equation}\label{eq:s'-1}
\begin{aligned}
\Fs'  & \leq_F (p_1', {q^{h'}_1})\star \Fs_a(\varphi_1',\chi_1', {q^{(1)'}}) \\
 & \leq_F (p_1', {q^{h'}_1}) \star \Fs_a(\varphi_1^\natural, \chi_1^\natural, {q^{(1)'}}) 
\end{aligned}
\end{equation} 
By Lemma \ref{lem:phi1}, we have that
\[
\Fs_a(\varphi_1^\natural, \chi_1^\natural, {q^{(1)'}})  = ((p_2^\natural, q^{h''}_2), (p_3'', {q^{h''}_3}), \ldots, (p_{k''}'', {q^{h''}_{k''}})).
\]
Note that $p_1'+ p_2^\natural = p_1''+ p_2''$ and $p_1'\leq p_1''$. It follows that
\begin{equation} \label{eq:s'-2}
(p_1', {q^{h'}_1}) \star \Fs_a(\varphi_1^\natural, \chi_1^\natural, {q^{(1)'}})  \leq_F (p_1'', q^{h'}_1) \star \Fs_a(\phi_1,\theta_1, {q^{(1)''}}).
\end{equation} 
Recall that $\Fs'\leq_F \Fs_a(\varphi,\chi,q)$ is equivalent to that
\begin{equation} \label{eq:s'-3}
\Fs'\leq_{p_1'} \Fs_a(\varphi, \chi, q).
\end{equation} 
We will prove the following technical result, which implies \eqref{eq:s'-3} by \eqref{p_1''}, \eqref{eq:s'-1} and \eqref{eq:s'-2}, and therefore finishes the proof
of Theorem \ref{thm:maximal-rational-order} when $E=\BR$.

\begin{prop} \label{prop:p1''}
Let the definitions be as above. Then 
\begin{equation} \label{eq:s'-4}
\Fs'':=(p_1'', q^{h'}_1) \star \Fs_a(\phi_1,\theta_1, {q^{(1)''}}) \leq_{p_1''} \Fs_a(\varphi,\chi,q). 
\end{equation}
\end{prop}

The rest of this subsection is devoted to the proof of Proposition \ref{prop:p1''}. Recall that in this section we have fixed $a=+1$ in Definition \ref{LYT}, and we have that 
\begin{equation} \label{eq:phi-varphi}
\begin{aligned}
& (\varphi_1, \chi_1)\in \FD^{z_1}_{p_1}(\varphi, \chi), \quad (\varphi_2, \chi_2)\in \FD^{z_2}_{p_2}(\varphi_1, \chi_2), \\
& (\phi_1, \theta_1) \in \FD^{z_1'}_{p_1''}(\varphi,\chi), \quad (\phi_2, \theta_2) \in \FD^{z_2''}_{p_2''}(\phi_1, \theta_1).
\end{aligned}
\end{equation}
To unify the notation, we put
\[
z_1'':=z_1', \quad q^{h''}_1: = q^{h'}_1,
\]
so that
\[
\Fs'' = ((p_1'', q^{h''}_1), (p_2'', q^{h''}_2), \ldots, (p_{k''}'', q^{h''}_{k''})).
\]
The key step of the proof of Proposition \ref{prop:p1''} is to compare 
the first descents of $(\varphi_1, \chi_1)$ and $(\phi_1, \theta_1)$, which is given as follows.

\begin{lem} \label{first-step}
Assume that $m\neq 0$ in \eqref{eq:phi1}. Then the following hold.
\begin{enumerate}
\item If $q^h_1 = - q^h_2$, then $(\phi_2, \theta_2)$ is a first descent of $(\varphi_1, \chi_1)$, and it holds that
\begin{itemize}
\item $(p_1'', q^{h''}_1)= (p_1-2m, (-1)^m q^h_1 )$, 
\item $(p_2'', q^{h''}_2) =(p_2+2m, (-1)^m q^h_2)$.
\end{itemize}
\item If $q^h_1 =  q^h_2$, then $\phi_2$ can be taken to be $\varphi_2\oplus \rho_{\gamma}$, where $\gamma \in (\beta_{s_\chi}, \beta_{s_\chi+1})$, $\gamma\equiv \alpha_1\mod 2$, and it holds that
\begin{itemize}
\item $(p_1'', q^{h''}_1)= (p_1-2m, (-1)^m q^h_1 )$, 
\item $(p_2'', q^{h''}_2) =(p_2+2(m-1), (-1)^{m-1} q^h_2)$, 
\item $\chi_1 = \chi_{\varphi_1, \phi_2}\cdot \eta_{ - z_2}$, $\wh{\theta_2} = \chi_{\phi_2, \varphi_1}\cdot \eta_{ - z_2}$.
\end{itemize}
\end{enumerate}
\end{lem}

\begin{proof}
We  prove the lemma case by case, and our computation below uses Propositions  \ref{prop:rational}, \ref{prop:fd} and \ref{prop:pretab}, and the formulas for distinguished characters in Section \ref{sec-ADA}. Recall from 
\eqref{eq:phi1} and \eqref{eq:phi-varphi} that  
\begin{equation}  \label{eq:basic}
\begin{cases}
  \phi_1 = \varphi_1\oplus m \rho_{\beta_{s_\chi+1}} \ \textrm{with } \beta_{s_\chi+1}> \alpha_r, \\
  \chi = \chi_{\varphi, \varphi_1}\cdot \eta_{z_1} = \chi_{\varphi, \phi_1}\cdot \eta_{z''_1}, \\ 
   \wh{\chi_1} = \chi_{\varphi_1, \varphi}\cdot \eta_{z_1} , \quad  \chi_1 = \chi_{\varphi_1, \varphi_2}\cdot \eta_{z_2}, \\
  \wh{\theta_1} = \chi_{\phi_1, \varphi}\cdot \eta_{z''_1},  \quad  \theta_1 = \chi_{\phi_1, \phi_2} \cdot \eta_{z''_2}, \\
 \wh{\theta_2} = \chi_{\phi_2, \phi_1} \cdot \eta_{z''_2}.
\end{cases}
\end{equation}

{\bf Case 1. $G_n^*=\SO_{2n+1}^*$.} Then 
\begin{equation} \label{odd-so-q1''}
\begin{aligned}
& \disc(q^h_1) = -z_1=\disc(q)\cdot \disc(q^{(1)}) = \disc(q)\cdot (-1)^{s_\chi}, \\
& \disc(q^{h''}_1) = -z''_1= \disc(q) \cdot \disc(q^{(1)''}) = \disc(q) \cdot (-1)^{s_\chi+m} = (-1)^m \disc(q^h_1).
\end{aligned}
\end{equation}
This  implies that $z''_1 = (-1)^m z_1$, which together with \eqref{eq:basic} gives that
\[
\theta_1(\beta_{s_\chi}) = (-1)^m \chi_1(\beta_{s_\chi}) = (-1)^m z(\chi_1) = (-1)^m z_2, \quad
\theta_1(\beta_{s_\chi+1}) = z_1'' = (-1)^mz_1.
\]
On the other hand, we have  that
\[
\disc(q^h_2) = z_2, \quad   \disc(q^{h''}_2)  = z_2'' =  z(\theta_1) = \theta(\beta_{s_\chi+1}) = (-1)^mz_1.
\]
It follows that $s_\chi\in \sgn(\theta_1)$, i.e. $\theta_1(\beta_{s_\chi}) = - \theta_1(\beta_{s_\chi+1})$, if and only if 
\begin{equation} \label{odd-so-q1q2}
\disc(q^h_1) = -z_1 = z_2 = \disc(q^h_2),
\end{equation}
which is also equivalent to that 
\begin{equation} \label{odd-so-q2''}
\disc(q^{h''}_2) = (-1)^m z_1 = (-1)^{m-1} z_2 = (-1)^{m-1} \disc(q^h_2). 
\end{equation} 

In view of \eqref{odd-so-q1''}, \eqref{odd-so-q1q2} and \eqref{odd-so-q2''}, it remains to prove the last two equalities in (2). In this case we may take $\phi_2 = \varphi_2\oplus \rho_\gamma$, $\gamma\in (\beta_{s_\chi}, \beta_{s_\chi+1})$, which gives that 
\[
\chi_{\varphi_1, \phi_2}\cdot \eta_{-z_2} = \chi_{\varphi_1, \varphi_2} \cdot  \eta_{z_2} = \chi_1.
\]
Since $\beta_{s_\chi+1}\geq \alpha_r$ and the $\CZ$-action on $\CS_{\phi_2}$ is trivial, we obtain that
\[
 \chi_{\phi_2, \varphi_1} = \chi_{\phi_2, \phi_1} = \theta_2.
\]

\

{\bf Case 2. $G_n^*=\SO_{2n}^*$.}  Then 
\begin{equation} \label{even-so-q1''}
\begin{aligned}
&\disc(q^h_1) = z_1 = z(\chi) = \chi(\alpha_r), \\
&  \disc(q^{h''}_1) = z_1''= (-1)^m \chi(\alpha_r) = (-1)^m \disc(q^h_1).
\end{aligned}
\end{equation}
Using \eqref{eq:basic}, we find that
\[
\begin{aligned}
& \theta_1(\beta_{s_\chi+1}) = \disc(q) = -\disc(q^{(1)})\cdot \disc(q^h_1), \\
& \theta_1(\beta_{s_\chi}) =\chi_1(\beta_{s_\chi}) = \disc(q^{(2)}) = \disc(q^{(1)})\cdot \disc(q^h_2).
\end{aligned}
\]
It follows that $s_\chi\in \sgn(\theta_1)$, i.e. $\theta_1(\beta_{s_\chi}) = - \theta_1(\beta_{s_\chi+1})$, if and only if 
\begin{equation} \label{even-so-q1q2}
z_1 = \disc(q^h_1) = \disc(q^h_2) = -z_2.
\end{equation} 
From this, we obtain that
\[
\det(\varphi_2)\cdot \det(\phi_2) = 
\begin{cases} 
1, & \textrm{if }q^h_1= - q^h_2, \\
-1, & \textrm{if }q^h_1 = q^h_2.
\end{cases}
\]
On the other hand,
\[
\disc(q) = - \disc(q^h_1)\cdot \disc(q^h_2) \cdot \det(\varphi_2) =  - \disc(q^{h''}_1)\cdot \disc(q^{h''}_2) \cdot \det(\phi_2),
\]
which implies that 
\begin{equation} \label{even-so-q2''}
\disc(q^{h''}_2)\cdot\disc(q^h_2) = (-1)^{m}\det(\varphi_2)\cdot\det(\phi_2) = 
\begin{cases}  
(-1)^m, & \textrm{if }q^h_1= - q^h_2, \\
(-1)^{m-1}, & \textrm{if }q^h_1 = q^h_2.
\end{cases}
\end{equation} 

In view of \eqref{even-so-q1''}, \eqref{even-so-q1q2} and \eqref{even-so-q2''}, it remains to prove the last two equalities in (2). In this case again we take $\phi_2 = \varphi_2\oplus \rho_\gamma$, $\gamma\in (\beta_{s_\chi}, \beta_{s_\chi+1})$. Since the 
$\CZ$-action on $\CS_{\varphi_1}$ is trivial, we have that
\[
 \chi_{\varphi_1, \phi_2}  = \chi_{\varphi_1, \varphi_2} = \chi_1.
\]
In this case we have
\[
z_2'' = - \disc(q^{h''}_2) =- (-1)^{m-1} \disc(q^h_2) = (-1)^{m-1} z_2,
\]
which implies that 
\[
\chi_{\phi_2, \varphi_1}\cdot \eta_{ - z_2} = \chi_{\phi_2, \phi_1} \cdot \eta_{(-1)^{m-1} z_2} =  \chi_{\phi_2, \phi_1} \cdot \eta_{z_2''} = \wh{\theta_2}.
\]

\

{\bf Case 3. $G_n^*=\Sp_{2n}$ or $\Mp_{2n}$.} In this case we have
\[
z_i = \disc(q^h_i),\quad z_i''=\disc(q^{h''}_i),\quad i=1, 2.
\]
By \eqref{eq:basic}, we find that
\[
\chi(\alpha_r) = z(\chi) =z_1\quad \textrm{and} \quad \chi(\alpha_r) = (-1)^m z_1'',
\]
which imply that
\begin{equation} \label{sp-q1''}
\disc(q^{h''}_1) = z_1'' = (-1)^m z_1 = (-1)^m \disc(q^h_1).
\end{equation} 
We further compute that
\[
\theta_1(\beta_{s_\chi}) =  z_1''\cdot  z_1\cdot \chi_1(\beta_{s_\chi})  =(-1)^{m} \cdot z(\chi_1) = (-1)^{m} z_2.
\]
On the other hand, 
\[
z_2'' = z(\theta_1) = \theta_1(\beta_{s_{\chi+1}}) = - z_1'' = (-1)^{m-1} z_1.
\]
It follows that $s_\chi\in \sgn(\theta_1)$, i.e. $\theta_1(\beta_{s_\chi}) =- \theta_1(\beta_{s_\chi+1})$, if and only if 
\begin{equation} \label{sp-q1q2}
\disc(q^h_1) = z_1 = z_2 = \disc(q^h_2),
\end{equation}
which is also equivalent to that
\begin{equation} \label{sp-q2''}
\disc(q^{h''}_2) = z_2'' = (-1)^{m-1} z_1 = (-1)^{m-1}z_2 =(-1)^{m-1} \disc(q^h_2).
\end{equation} 

In view of \eqref{sp-q1''}, \eqref{sp-q1q2} and \eqref{sp-q2''}, it remains to prove the last two equalities in (2). In this case take $\phi_2 =\varphi_2 \oplus \rho_\gamma$, $\gamma\in (\beta_{s_\chi}, \beta_{s_\chi+1})$. 
Then we compute that 
\[
 \chi_{\varphi_1, \phi_2}\cdot \eta_{ - z_2} = \chi_{\varphi_1, \varphi_2} \cdot \eta_{z_2} = \chi_1
\]
and that 
\[
 \chi_{\phi_2, \varphi_1}\cdot \eta_{ - z_2} = \chi_{\phi_2, \phi_1} \cdot \eta_{(-1)^{m-1} z_2} =  \chi_{\phi_2, \phi_1} \cdot \eta_{z_2''} = \wh{\theta_2}. 
\]

This finishes the proof of the lemma for all the cases.
\end{proof}

\begin{proof}(of Proposition \ref{prop:p1''})
By \eqref{eq:phi1},
\[
p_1'' = \Fn -\dim \phi_1 = \Fn - \dim \varphi_1 - 2m =p_1 - 2m.
\]
If $m=0$ in \eqref{eq:phi1}, then equality holds in \eqref{eq:s'-4}. Assume that $m\neq 0$. 
Let $k_0$ be the maximal index $i$ such that $q^h_1=q^h_2=\cdots=q^h_i$. The properties of $(\phi_2, \theta_2)$ in Lemma \ref{first-step} enable us to apply this lemma repeatedly, which thereby gives the following results.
\begin{itemize}
\item If $k_0>1$, then 
\[
k''=\begin{cases} k, & \textrm{if }k_0<k, \\
k+1, & \textrm{if }k_0=k,
\end{cases}
\]
and for $i=1,\ldots, k''$,
\begin{equation}\label{eq:Fs'-Fs1}
(p''_i,q^{h''}_i)=\begin{cases}
	(p_1- 2m,(-1)^{m}q^h_1), &\text{ if }i=1,\\
	(p_2+ 2(m-1),(-1)^{m-1}q^h_2), &\text{ if }i=2,\\
	(p_i, q^h_i), &\text{ if } 2<i\leq k_0,\\
	(p_{k_0+1}+2, -q^h_{k_0+1}), &\text{ if }i=k_0+1,\\
	(p_i, q^h_i), &\text{ if }i>k_0+1,
\end{cases}
\end{equation}
where 
$(p_{k+1}, q^h_{k+1}):=(0, -q^h_{k})$ if $k_0=k$.

\item
If $k_0=1$, then $k''=k$ and for $i=1,\ldots, k$, 
\begin{equation} \label{eq:Fs'-Fs2}
(p''_i,q^{h''}_i)=\begin{cases}
	(p_1-2m,(-1)^{m}q^h_1), &\text{ if }i=1,\\
	(p_2+2m, (-1)^{m}q^h_2), &\text{ if }i=2,\\
	(p_i,q^h_i), &\text{ if } i>2.
\end{cases}
\end{equation} 
\end{itemize}
Applying \eqref{eq:Fs'-Fs1} and \eqref{eq:Fs'-Fs2}, a straightforward computation verifies \eqref{eq:s'-4}. \end{proof}

\subsection{The case $E=\BC$}

In this subsection we prove Theorem \ref{thm:maximal-rational-order} when $E=\BC$, that is, for unitary groups. The strategy is similar to the case that $E=\BR$, with more involved details due to
the mixture of Bessel models and Fourier-Jacobi models. 

Recall that $\varphi \in \wt{\Phi}_\gen(\RU_\Fn^*)$ is of the form 
\eqref{par-unif}, and $\chi\in \wh{\CS_\varphi}$.
 It will be convenient to describe the descent  $\FD^z_\Fl(\varphi,\chi)$ more explicitly as follows, by unfolding the relevant definitions. 
 
\begin{lem}\label{lem:Upd}
Let $\phi\in \wt{\Phi}_\gen(\RU_{\Fn-\Fl}^*)$ and $\chi'\in \wh{\CS_\phi}$, where
\[
\phi = \bigoplus^{r'}_{j=1} m_j' \rho_{\beta_j} \oplus \xi'\oplus {}^c\xi'^\vee
\]
is of the form similar to \eqref{par-unif}. Assume that $\varphi$ and $\phi$ are of opposite types. Then $(\phi,\chi')\in \FD^z_\Fl(\varphi,\chi)$ if and only if 
\begin{align*}
 \chi(\alpha_i) &  = z\cdot \prod_{j:\ \beta_j > \alpha_i}(-1)^{m_j'},\quad i=1,\ldots, r, \\
 \chi'(\beta_j) & = (-z) \cdot \prod_{i:\ \alpha_i>  \beta_j} (-1)^{m_i} \\
& = (-1)^{\Fn-1}z \cdot \prod_{i:\  \alpha_i< \beta_j}(-1)^{m_i}, \quad j=1,\ldots, r'.
\end{align*}
\end{lem}

\begin{proof}
By Proposition \ref{prop:dualdata}, \eqref{chi-a} and Definition \ref{defn:pd}, $(\phi, \chi')\in \FD^z_{\Fl}(\varphi,\chi)$ if and only if 
\[
\chi = \chi_{\varphi, \wh{\phi}}\cdot \eta_z, \quad \chi' = \chi_{\wh{\phi}, \varphi} \cdot \eta_{(-1)^\Fl z} \cdot \eta_{(-1)^{\Fn-\Fl-1}} =  \chi_{\wh{\phi}, \varphi} \cdot \eta_{(-1)^{\Fn-1}z},
\]
where we have used the natural identification $\CS_\phi\cong\CS_{\wh\phi}$.
The lemma follows from \eqref{chi1-C} and \eqref{chi2-C}, noting that $\sum^r_{i=1}m_i\equiv \Fn\mod 2$.
\end{proof}

Recall from \eqref{aizi'} that $(\varphi_1', \chi_1')\in \FD^{z_1'}_{p_1}(\varphi,\chi)$, with $\varphi_1'$ decomposed as \eqref{varphi1'-decomp-2}, and we have the first descent 
$(\varphi_1,\chi_1)\in \FD^{z_1}_{p_1}(\varphi,\chi)$, where $\varphi_1$ is given by \eqref{varphi1}. Similar to \eqref{LP-natural}, define an enhanced tempered $L$-parameter
\[
\varphi_1^\natural := \psi_0\oplus \varphi_1 \oplus \psi_r \oplus \xi_1\oplus {}^c\xi_1^\vee, \quad \chi_1^\natural :=\chi_1'\vert_{\varphi_1^\natural},
\]
where $\xi_1$ is an arbitrary direct sum of unitary characters that are not conjugate self-dual,  such that $\dim \varphi_1' = \dim \varphi_1^\natural$. 
Note that $\dim \psi_0=\mult\!(\psi_0)$ is not necessarily even, in contrast to the case $E=\BR$. The following result can be proved using Lemma \ref{lem:Upd}, in the similar way to that of Lemma \ref{lem:natural}, which will be omitted. 

\begin{lem} \label{lem:natural-U}
Let the definitions be as above. Then the following hold.
\begin{enumerate}
\item $(\varphi_1^\natural, \chi_1^\natural)\in \FD^{z_1'}_{p_1'}(\varphi,\chi)$. 
\item A first descent of $(\varphi_1',\chi_1')$ is also a descent of $(\varphi_1^\natural, \chi_1^\natural)$. 
\end{enumerate}
\end{lem}

Similar to \eqref{eq:phi1}, define 
\begin{equation}  \label{eq:phi1-U}
\phi_1 =m_L\rho_{\beta_0}\oplus  \varphi_1\oplus m_R\rho_{\beta_{s_\chi+1}},
\end{equation}
where $\beta_0$ and $\beta_{s_\chi+1}$ are chosen such that $\beta_0 < \alpha_1 \leq\alpha_r < \beta_{s_\chi+1}$, $\beta_0\equiv \beta_{s_\chi+1}\equiv \alpha_1+1\mod 2$,  $m_L :=\dim \psi_0$ and $m_R := \dim \psi_r$.  Define 
\[
\theta_1\in \wh{\CS_{\varphi_1}}
\]
 such that $\theta_1\vert_{\varphi_1} = \chi_1'\vert_{\varphi_1}$, and $\theta_1(\beta_0)$ (resp. $\theta_1(\beta_{s_\chi+1})$) equals the constant value of 
$\chi_1'$ on the natural basis vectors of $\CS_{\psi_0}$ (resp. $\CS_{\psi_r}$) if $m_L\neq 0$ (resp. $m_R\neq 0$). Put 
\[
p_1'' := \Fn -\dim \phi_1 = p_1' + \dim\varphi_1^\natural - \dim \phi_1.
\]
Then $p_1''\geq p_1'$, and $p_1'' \equiv p_1' \mod 2$.  The following result is similar to Lemma \ref{lem:natural-U}.

\begin{lem} \label{lem:phi1-U}
Let the definitions be as above. Then the following hold.
\begin{enumerate}
\item $(\phi_1, \theta_1)\in \FD^{z_1'}_{p_1''}(\varphi,\chi)$. 

\item $(\varphi_1^\natural, \chi_1^\natural)$ and $(\phi_1, \theta_1)$ have the same first descents. 
\end{enumerate}
\end{lem}

Similar to the case that $E=\BR$, we can assume that $\CS_{\varphi_1}$ is nontrivial, and take a first descent sequence of $(\phi_1, \theta_1, q^{(1)''})$ of the form \eqref{phi-fd}, where $q^{(1)''}$ and $q^{(1)'}$ belong to the same Witt tower of (skew)-Hermitian forms.  By Lemmas \ref{lem:natural-U}, \ref{lem:phi1-U} and the arguments used when $E=\BR$, the proof of Theorem \ref{thm:maximal-rational-order} when $E=\BC$ is  reduced to proving the following  result.

\begin{prop} \label{prop:p1''-U}
Let the definitions be as above. Then 
\[
\Fs'':=(p_1'', q^{h'}_1) \star \Fs_a(\phi_1,\theta_1, {q^{(1)''}}) \leq_{p_1''} \Fs_a(\varphi,\chi,q). 
\]
\end{prop}

Recall that we have 
\eqref{eq:phi-varphi}.  As before,  put
$z_1'':=z_1'$,  $q^{h''}_1: = q^{h'}_1$,
so that
\[
\Fs'' = ((p_1'', q^{h''}_1), (p_2'', q^{h''}_2), \ldots, (p_{k''}'', q^{h''}_{k''})).
\]
To prove Proposition \ref{prop:p1''-U}, again we need to compare the consecutive first descents of $(\phi_1,\theta_1)$ and $(\varphi_1,\chi_1)$.  Instead of doing this directly, it is more manageable to take an intermediate pre-tableau as follows.

Define a collapse $(\phi_1^R, \theta_1^R)$ of $(\phi_1, \theta_1)$ as follows:
\[
\phi_1^R:=\varphi_1\oplus  m_R \rho_{\beta_{s_\chi+1}}, \quad \theta_1^R:= \theta_1\vert_{\phi_1^R}.
\]
Put $p_1^R := \Fn- \dim \phi_1^R = p_1''+ m_L\geq p_1''$. By Lemma \ref{lem:Upd}, it is clear that $(\phi_1^R, \theta_1^R)\in \FD^{z_1''}_{p_1^R}(\varphi, \chi)$. Take a first descent sequence 
\[
    \begin{matrix}
    (\phi_1^R,\theta_1^R, {q^{(1)^R}}),&(\phi_2^R,\theta_2^R, {q^{(2)^R}}),&\dots,&(\phi_{k^R-1}^R,\theta_{k^R-1}^R, {q^{(k^R-1)^R}})\\
    (p_2^R, {q^{h^R}_2}),& (p_3^R, {q^{h^R}_3}), &\dots,& (p_{k^R}^R, {q^{h^R}_{k^R}})
    \end{matrix}
\]
of $(\phi_1^R,\theta_1^R, {q^{(1)^R}})$, so that
\[
\Fs_a(\phi_1^R,\theta_1^R, {q^{(1)^R}}) = ((p_2^R, {q^{h^R}_2}), \ldots, (p_{k^R}^R, {q^{h^R}_{k^R}})).
\]
We have that
\[
(\phi_i^R, \theta_i^R) \in \FD_{p_i^R}^{z_i^R}(\phi_{i-1}^R, \theta_{i-1}^R), \quad z_i^R=[\disc(q^{h^R}_i)]\in\CZ, \ i=2,\ldots, k^R-1,
\]
subject to the conditions of Definition \ref{LYT}.
We put $z_1^R := z_1''$,  $q^{h^R}_1:= q^{h''}_1$  and define  
\[
\Fs^R:=(p_1^R, q_1^{h^R})\star \Fs_a(\phi_1^R,\theta_1^R, {q^{(1)^R}}) = ((p_1^R, {q^{h^R}_1}), \ldots, (p_{k^R}^R, {q^{h^R}_{k^R}})). 
\]
Then Proposition \ref{prop:p1''-U} is a consequence of the following two results. 

\begin{prop} \label{prop:p1''-U1}
It holds that $\Fs^R \leq_{p_1^R} \Fs_a(\varphi,\chi,q)$. 
\end{prop}

\begin{prop} \label{prop:p1''-U2}
It holds that $\Fs''  \leq_{p_1''} \Fs^R$. 
\end{prop}

The rest of this section is devoted to the proof of Propositions \ref{prop:p1''-U1} and \ref{prop:p1''-U2}. 
For  unitary groups, it is more convenient to introduce the following elements of $\CZ$,
\[
\begin{cases}
 \varepsilon_i := (-1)^{p_i-1}\cdot [\disc(q^h_i)] = (-1)^{p_i-1} z_i, & i=1,\ldots, k,\\
\varepsilon_i^R := (-1)^{p_i^R-1}\cdot [\disc(q^{h^R}_i)] = (-1)^{p_i^R-1} z_i^R, & i=1,\ldots, k^R,\\
\varepsilon_i'' := (-1)^{p_i''-1}\cdot [\disc(q^{h''}_i)] = (-1)^{p_i''-1} z_i'', & i=1,\ldots, k''.
\end{cases}
\]
That is, $\varepsilon_i$, $\varepsilon_i^R$ and $\varepsilon_i''$ represent the signs in the first boxes  of the $i$-th row of the ordered signed Young diagrams $\Fs_a(\varphi,\chi,q)$, $\Fs^R$ and $\Fs''$, respectively. Then we can  work with the data
$(p_i, \varepsilon_i)$, $(p_i^R, \varepsilon_i^R)$ and $(p_i'', \varepsilon_i'')$ instead of the pre-tableaux $(p_i, q^h_i)$, $(p_i^R, q^{h^R}_i)$ and $(p_i'', q^{h''}_i)$.

\subsubsection{Proof of Proposition \ref{prop:p1''-U1}}

\begin{lem} \label{lem:p1''-U1}
Assume that $m_R\neq 0$ in \eqref{eq:phi1-U}. Then the following hold.
\begin{enumerate}
\item If $q^h_1 = - q^h_2$, then $(\phi_2^R, \theta_2^R)$ is a first descent of $(\varphi_1, \chi_1)$, and it holds that
\begin{itemize}
\item $(p_1^R, \varepsilon_1^R)= (p_1-m_R, \varepsilon_1)$, 
\item $(p_2^R,  \varepsilon_2^R ) =(p_2+m_R,  \varepsilon_2)$.
\end{itemize}
\item If $q^h_1 =  q^h_2$, then $\phi_2^R$ can be taken to be $\varphi_2\oplus \rho_{\gamma}$, where $\gamma \in (\beta_{s_\chi}, \beta_{s_\chi+1})$, $\gamma \equiv \alpha_1\mod 2$, and it holds that
\begin{itemize}
\item $(p_1^R,  \varepsilon_1^R)= (p_1-m_R,  \varepsilon_1 )$, 
\item $(p_2^R, \varepsilon_2^R ) =(p_2+ m_R-1,  \varepsilon_2)$, 
\item $\chi_1 = \chi_{\varphi_1, \wh{\phi_2^R}}^{-z_2}$, $\wh{\theta_2^R} = \chi_{\wh{\phi_2^R}, \varphi_1}^{-z_2}$.
\end{itemize}
\end{enumerate}
\end{lem}

\begin{proof}
Similar to the symplectic-metaplectic case in the proof of Lemma \ref{first-step}, we have that
\[
\begin{aligned}
& \theta_1(\beta_{s_\chi})  = (-1)^{m_R} z_2, \\
& \theta_1(\beta_{s_\chi+1})  = z_2^R = -z_1^R =(-1)^{m_R-1}z_1.
\end{aligned}
\]
It follows that $\theta_1(\beta_{s_\chi}) = - \theta_1(\beta_{s_\chi+1})$ if and only if $z_1= z_2$, i.e. $q^h_1 = q^h_2$, which is also equivalent to that $z_2^R = (-1)^{m_R-1} z_2$. In any case, it is easy to see that we always have 
$\varepsilon_i^R= \varepsilon_i$, $i=1,2$.

It remains to prove the last  two equalities in (2). In this case take $\phi_2^R = \phi_2\oplus \rho_\gamma$, $\gamma\in (\beta_{s_\chi}, \beta_{s_\chi+1})$. Using the proof of Lemma \ref{lem:Upd}, we compute that
\[
\chi_{\varphi_1, \wh{\phi_2^R}}^{-z_2} = \chi_{\varphi_1, \wh{\phi_2}}^{z_2} = \chi_1
\]
and that
\[
\chi_{\wh{\phi_2^R}, \varphi_1}^{-z_2} =  \chi_{\wh{\phi_2^R}, \phi_1^R}^{-z_2} \cdot \eta_{(-1)^{m_R}} = \chi_{\wh{\phi_2^R}, \phi_1^R}^{z_2^R} = \wh{\theta_2^R}.
\]
The lemma is proved.  
\end{proof}
 
Now we prove Proposition \ref{prop:p1''-U1}. If $m_R=0$, then equality holds in this proposition. Assume that $m_R\neq 0$. Let $k_0$ be the maximal index $i$ such that $q^h_1=q^h_2=\cdots=q^h_i$. Applying Lemma \ref{lem:p1''-U1} repeatedly, we obtain the following results.
\begin{itemize}
\item If $k_0>1$, then 
\[
k^R =\begin{cases} k, & \textrm{if }k_0 <k, \\
k+1, & \textrm{if }k_0 =k, 
\end{cases}
\]
and for $i=1,\ldots, k^R$, 
\begin{equation}\label{eq:Fs'-Fs1-R}
(p_i^R, \varepsilon_i^R)=\begin{cases}
	(p_1- m_R, \varepsilon_1), &\text{ if }i=1,\\
	(p_2+ m_R-1, \varepsilon_2), &\text{ if }i=2,\\
	(p_i, \varepsilon_i), &\text{ if } 2<i\leq k_0,\\
	(p_{k_0+1}+1,  \varepsilon_{k_0+1}), &\text{ if }i=k_0+1,\\
	(p_i, \varepsilon_i), &\text{ if }i>k_0+1,
\end{cases}
\end{equation}
where $(p_{k+1}, \varepsilon_{k+1}):=(0, [\disc(q^h_k)])$ if $k_0=k$.

\item
If $k_0=1$, then $k_R=k$ and for $i=1,\ldots, k$, 
\begin{equation} \label{eq:Fs'-Fs2-R}
(p_i^R, \varepsilon_i^R)=\begin{cases}
	(p_1-m_R, \varepsilon_i), &\text{ if }i=1,\\
	(p_2+m_R, \varepsilon_i ), &\text{ if }i=2,\\
	(p_i, \varepsilon_i), &\text{ if } i>2.
\end{cases}
\end{equation} 
\end{itemize}
Applying \eqref{eq:Fs'-Fs1-R} and \eqref{eq:Fs'-Fs2-R}, a straightforward computation verifies the inequality in Proposition \ref{prop:p1''-U1}.  

\subsubsection{Proof of Proposition \ref{prop:p1''-U2}}

Following the above proof, we can write 
\[
\varepsilon^R_i = \varepsilon_i, \quad  i=1,\ldots, k^R.
\] 

\begin{lem} \label{lem:p1''-U2}
Assume that $m_L\neq 0$ in \eqref{eq:phi1-U}. Then the following hold.
\begin{enumerate}
\item If $\varepsilon_1 = - \varepsilon_2$, then $(\phi_2, \theta_2)$ is a first descent of $(\phi_1^R, \theta_1^R)$, and it holds that
\begin{itemize}
\item $(p_1'', \varepsilon_1'')= (p_1 -m_L , (-1)^{m_L} \varepsilon_1 )$, 
\item $(p_2'',  \varepsilon_2'' ) =(p_2+m_L, (-1)^{m_L} \varepsilon_2)$.
\end{itemize}
\item If $\varepsilon_1 =  \varepsilon_2$, then $\phi_2$ can be taken to be $\rho_\gamma\oplus \phi_2^R$, where $\gamma \in (\beta_0,  \beta_1)$, $\gamma\equiv \alpha_1\mod 2$, and it holds that
\begin{itemize}
\item $(p_1'', \varepsilon_1'')= (p_1 -m_L , (-1)^{m_L} \varepsilon_1 )$,
\item $(p_2'',  \varepsilon_2'' ) =(p_2+m_L, (-1)^{m_L-1} \varepsilon_2)$, 
\item $\theta_1^R = \chi_{\phi_1^R, \wh{\phi_2}}^{z_2^R}$, $\wh{\theta_2} = \chi_{\wh{\phi_2}, \phi_1^R}^{ z_2^R}$.
\end{itemize}
\end{enumerate}
\end{lem}

\begin{proof} 
Using Lemma \ref{lem:Upd}, we find that
\[
\begin{aligned}
\theta_1(\beta_0) & = (-1)^{\Fn-1}z_1'', \\
\theta_1(\beta_1) & = z_1''\cdot z_1^R \cdot \theta_1^R(\beta_1) = z_1''\cdot z_1^R \cdot  z_2^R \cdot (-1)^{s_{\theta_1^R}} = z_1''\cdot z_1^R \cdot z_2^R \cdot (-1)^{\Fn- p_1^R - p_2^R} =  - \theta_1(\beta_0) \cdot \varepsilon_1 \cdot \varepsilon_2.
\end{aligned}
\]
It follows that 
\begin{equation} \label{varepsilon-cond}
\theta_1(\beta_0) = - \theta_1(\beta_1) \quad \textrm{if and only if }\quad \varepsilon_1= \varepsilon_2.
\end{equation}  
Recall that by definition $z_1''= z_1^R=(-1)^{p_1^R-1}\varepsilon_1$, and by Lemma \ref{lem:Upd} it is clear that 
$
z_2'' = z_2^R = (-1)^{p_2^R-1}\varepsilon_2.
$
We have that $p_1'' = p_1^R -m_L$, and it follows from  \eqref{varepsilon-cond} that
\[
p_2'' = \begin{cases} p_2^R+m_L, & \textrm{if }\varepsilon_1= - \varepsilon_2, \\
p_2^R+m_L-1, & \textrm{if }\varepsilon_1 = \varepsilon_2.
\end{cases} 
\]
Thus $\varepsilon_1'' =(-1)^{p_1''-1}z_1''$ and $\varepsilon_2'' = (-1)^{p_2''-1}z_2''$ are computed as in the lemma. 

It remains to prove the last two equalities in (2). In this case take $\phi_2 = \rho_\gamma\oplus \phi_2^R$, $\gamma\in (\beta_0, \beta_1)$. Using the proof of Lemma \ref{lem:Upd}, we compute that
\[
 \chi_{\phi_1^R, \wh{\phi_2}}^{ z_2^R} =  \chi_{\phi_1^R, \wh{\phi_2^R}}^{ z_2^R} = \theta_1^R \quad \textrm{and} \quad \chi_{\wh{\phi_2}, \phi_1^R}^{ z_2^R} = \chi_{\wh{\phi_2}, \phi_1}^{ z_2''} =\wh{\theta_2}.
\]
The lemma is proved. 
\end{proof}

Finally, we prove Proposition \ref{prop:p1''-U2}. If $m_L=0$, then equality holds in this proposition. Assume that $m_L\neq 0$. Let $k_0'$ be the maximal index $i$ such that $\varepsilon_1=\varepsilon_2=\cdots=\varepsilon_i$. Applying Lemma \ref{lem:p1''-U2} repeatedly, we obtain the following results.
\begin{itemize}
\item If $k_0'>1$, then 
\[
k'' =\begin{cases} k^R, & \textrm{if }k_0' <k^R, \\
k^R+1, & \textrm{if }k_0' =k^R, 
\end{cases}
\]
and for $i=1,\ldots, k''$, 
\begin{equation}\label{eq:Fs'-Fs1-R2}
(p_i'', \varepsilon_i'')=\begin{cases}
	(p_1^R- m_L, (-1)^{m_L}\varepsilon_1), &\text{ if }i=1,\\
	(p_2^R+ m_L-1, (-1)^{m_L-1}\varepsilon_2), &\text{ if }i=2,\\
	(p_i^R, \varepsilon_i), &\text{ if } 2<i\leq k_0',\\
	(p_{k_0'+1}^R+1,  -\varepsilon_{k'_0+1}), &\text{ if }i=k_0'+1,\\
	(p_i^R, \varepsilon_i), &\text{ if }i>k_0'+1,
\end{cases}
\end{equation}
where $(p_{k^R+1}, \varepsilon_{k^R+1}):=(0,  -\varepsilon_{k^R})$ if $k_0'=k^R$.

\item
If $k_0'=1$, then $k''=k^R$ and for $i=1,\ldots, k''$, 
\begin{equation} \label{eq:Fs'-Fs2-R2}
(p_i'', \varepsilon_i'')=\begin{cases}
	(p_1^R-m_L,  (-1)^{m_L}\varepsilon_i), &\text{ if }i=1,\\
	(p_2^R+m_L, (-1)^{m_L}\varepsilon_i ), &\text{ if }i=2,\\
	(p_i^R, \varepsilon_i), &\text{ if } i>2.
\end{cases}
\end{equation} 
\end{itemize}
Applying \eqref{eq:Fs'-Fs1-R2} and \eqref{eq:Fs'-Fs2-R2}, a straightforward computation verifies the inequality in Proposition \ref{prop:p1''-U2}.


\appendix

\section{Proof of Proposition \ref{prop:LDPT}}\label{App-A}


We are going to prove Proposition \ref{prop:LDPT}, which claims that 
for any triple $(\varphi,\chi,q)\in \FT_a(G_n^*)$, the set of $L$-descent pre-tableaux $\CL_a(\varphi,\chi,q)$ is not empty, when $F$ is non-archimedean. The archimedean case 
follows easily from the explicit result for the $L$-parameter descent in Section \ref{sec-PCR}, in particular from Proposition \ref{prop:fd}. 

By \cite[Proposition 1.6]{JZ18} (which similarly holds for the Bessel case of unitary groups) and Definition \ref{defn:pd},   $\FD_{1}(\varphi,\chi)$ is non-empty when $G_n^*$ is orthogonal or unitary, therefore by induction $\CL_a(\varphi,\chi,q)$ contains some pre-tableau $\Fs_{\udl{\Fl}}$ with $\udl{\Fl}=[1^\Fn]$.

It remains to prove the symplectic-metaplectic case. The strategy of the proof is to reduce to the orthogonal case. To this end, we first prove a lemma. 

\begin{lem}\label{claim}
For any $(\varphi,\chi,q)\in \FT_a(\SO_\Fn^*)$ with $\Fn\geq 5$, it holds that
\begin{equation}
\FD_3(\varphi,\chi)\neq \varnothing.
\end{equation}
\end{lem}

\begin{proof}
We may decompose $\varphi$ as in \eqref{edec} and define its discrete part $\varphi_\square=\bigoplus_{i\in \RI_{\gp}}\varphi_i$ (cf. \cite[(4-16)]{JZ18}).
By Theorem 4.6 in \cite{JZ18}, we have $\Fl_{0}(\varphi,\chi)\geq \Fl_{0}(\varphi_\square,\chi)$. 
Thus it suffices to show that $\FD_3(\varphi,\chi)\neq \varnothing$ for discrete $L$-parameters $\varphi$.

Fix a choice of Whittaker datum and assume that $\pi$ is the discrete series representation of $\SO_\Fn(V)$ corresponding to $(\varphi,\chi)$ under the local Langlands correspondence.
Note that the Witt index $\Witt(V)\geq 1$ for any quadratic space defined over a non-archimedean field $F$ when $\Fn\geq 5$.

Although a purely local proof is possible, we think that a global argument may be more straightforward. Let $k$ be a number field with $\BA$ the ring of adeles of $k$. Fix a  
finite local place $\nu_0$ of $k$ such that $k_{\nu_0}=F$. One can easily find 
an $\Fn$-dimensional non-degenerate quadratic space $\FV$ over $k$ such that such that 
$\FV_{\nu_0}=V$ and $\Witt(\FV)=\Witt(V)\geq 1$. 
Since $\pi$ is an irreducible discrete series representation of $\SO_\Fn(V,F)$, by means of 
a simple trace formula $\SO_\Fn(\FV,\BA)$, 
one can embed $\pi$ as the local component at $\nu_0$ of an irreducible cuspidal automorphic representation $\Pi=\bigotimes_\nu\Pi_\nu$ of $\SO_{\Fn}(\FV,\BA)$, i.e. $\Pi_{\nu_0}\cong\pi$. 

Let $P_\Fn=M_\Fn N_\Fn$ be a (standard) parabolic subgroup defined over $k$ of $\SO_\Fn(\FV)$ whose Levi subgroup $M_\Fn$ is isomorphic to $\GL_1\times \SO_{\Fn-2}$.
In this case, the unipotent radical $N_\Fn$ is abelian and can be identified with 
a non-degenerate quadratic subspace of $\FV$, whose vector space is isomorphic to $k^{\Fn-2}$ and whose quadratic form is the restriction of the quadratic form $q_{\FV}$ on $\FV$.
For any cuspidal automorphic form $f\in\Pi$, we have a Fourier expansion of the compact abelian group $N_\Fn(k)\bks N_\Fn(\BA)$. More precisely, by taking 
a nontrivial additive character $\psi_\BA$ of $k\bks \BA$, we have the Fourier expansion along $[N_\Fn]:=N_\Fn(k)\bks N_\Fn(\BA)$:
\begin{align}\label{eq:Fourier-expansion-1}
f(g)
=&\int_{[N_\Fn]}f(ng)\ud n\nonumber\\
&\qquad +\sum_{v\ne 0\in N_\Fn(k); q_{\FV}(v,v)=0}
\int_{[N_\Fn]}f(ng)\psi_\BA(q_{\FV}(v,n))\ud n\\
&\qquad\qquad  +\sum_{v \in N_\Fn(k); q_{\FV}(v,v)\ne0}
\int_{[N_\Fn]}f(ng)\psi_\BA(q_{\FV}(v,n))\ud n.\nonumber
\end{align}
By the cuspidality of $f$, the constant term 
\[
\int_{[N_\Fn]}f(ng)\ud n=0.
\]
Hence we obtain that 
\begin{align}\label{eq:Fourier-expansion-2}
f(g)
=&\sum_{v\ne 0\in N_\Fn(k); q_{\FV}(v,v)=0}
\int_{[N_\Fn]}f(ng)\psi_\BA(q_{\FV}(v,n))\ud n\\
&\qquad  +\sum_{v\in N_\Fn(k); q_{\FV}(v,v)\ne0}
\int_{[N_\Fn]}f(ng)\psi_\BA(q_{\FV}(v,n))\ud n.\nonumber
\end{align}

Assume that there exists an anisotropic vector $v_*\in N_\Fn(k)$ such that the Bessel-Fourier coefficient 
\[
\int_{[N_\Fn]}f(ng)\psi_\BA(q_{\FV}(v_*,n))\ud n
\]
of $f$, from the second summation in \eqref{eq:Fourier-expansion-2}, is nonzero
(we refer to \cite{JZ20} for more details on the Bessel-Fourier coefficients of automorphic forms).
Then the local component $\pi=\Pi_{\nu_0}$ has a nonzero twisted Jacquet module (via local Bessel model) $\CJ_{\RI_{\Fn}-v}(\pi)$ defined in \eqref{whit-quo}, associated to the nilpotent element $\RI_{\Fn}-v_*$ in $\FN_\Fn(F)$,  where $\FN_\Fn$ is the Lie algebra of $N_{\Fn}$.
Since $v_*$ is $k$-anisotropic, the rational nilpotent orbit of $\RI_{\Fn}-v_*$ is parameterized by the partition $[3,1^{\Fn-3}]$. 
Following the arguments in the proof of \cite[Proposition 3.3]{JZ18} (the stability of local descent),  $\FD_3(\varphi,\chi)$ contains tempered enhanced $L$-parameters and is non-empty. 
Then the lemma is proved. 

If $\Witt(\FV)=1$, then there is no nonzero isotropic vector. In this case, we must have 
that the first summation in \eqref{eq:Fourier-expansion-2} vanishes and that 
\[
f(g)
=\sum_{v\in N_\Fn(k); q_{\FV}(v,v)\ne0}
\int_{[N_\Fn]}f(ng)\psi_\BA(q_{\FV}(v,n))\ud n\neq 0
\]
since $f\neq 0$. We finish the proof of the lemma for this case.  

If $\Witt(\FV)>1$, and if for any anisotropic vector $v_*\in N_\Fn(k)$, the Bessel-Fourier coefficient 
\[
\int_{[N_\Fn]}f(ng)\psi_\BA(q_{FV}(v_*,n))\ud n
\]
of $f$, from the second summation in \eqref{eq:Fourier-expansion-2}, is always zero, then 
we have 
\begin{align}\label{eq:Fourier-expansion-3}
f(g)
=\sum_{v\ne 0\in N_\Fn(k); q_{\FV}(v,v)=0}
\int_{[N_\Fn]}f(ng)\psi_\BA(q_{\FV}(v,n))\ud n.
\end{align}
Since $f$ is nonzero, we must have that 
there exists a nonzero isotropic vector $v_1$ such that 
\[
f_{\Fn-2}(g):=\int_{[N_\Fn]}f(ng)\psi_\BA(q_{\FV}(v_1,n))\ud n\ne 0.
\]
Note that the Levi subgroup $M_\Fn(k)$ acts on the characters of $[N_\Fn]$ via conjugation.
Denote by $P_{\Fn-2}$  the stabilizer of $\psi_\BA(q_{\FV}(v_1,\cdot))$ under this action of $M_{\Fn}(k)$.
Then $f_{\Fn-2}(g)$ is left $P_{\Fn-2}(k)$-invariant. Write $P_{\Fn-2}=M_{\Fn-2}N_{\Fn-2}$, which can be identified as a parabolic subgroup of $\SO_{\Fn-2}$ whose Levi subgroup $M_{\Fn-2}\cong \GL_1\times \SO_{\Fn-4}$ and unipotent radical $N_{\Fn-2}(k)\cong k^{\Fn-4}$. 
We may then take the Fourier expansion $f_{\Fn-2}(g)$ along  the compact abelian group $[N_{\Fn-2}]:=N_{\Fn-2}(k)\bks N_{\Fn-2}(\BA)$
and obtain an expression for $f_{\Fn-2}(g)$, which is similar to that in \eqref{eq:Fourier-expansion-1}.
By repeating the same arguments, we obtain that either $\FD_5(\varphi,\chi)\ne \emptyset$ or there exists a nonzero isotropic vector $v_2$ in $N_{\Fn-2}(k)$ such that 
\[
\int_{[N_{\Fn-2}N_{\Fn-1}]}f(n_1n_2g)\psi_\BA(q_{\FV}(v_1,n_1)+q_{\FV}(v_2,n_2)\ud n_2\ud n_1\ne 0.
\]
Finally, by the induction on the Witt index of $\FV$, we may repeat the above process until  we obtain the first odd integer $3\leq \Fl\leq 2\Witt(V)+1$ such that $\FD_\Fl(\varphi,\chi)\ne \emptyset$,
which finishes the proof of the lemma. 
\end{proof}

Now assume that $G_n^*=\Sp_{2n}$ or $\Mp_{2n}$, and that $(\varphi,\chi,q)\in \FT_a(G_n^*)$. We will prove  that 
\begin{equation}\label{geq2}
\FD_2(\varphi,\chi)\neq \varnothing,
\end{equation}
which by induction implies that $\CL_a(\varphi,\chi,q)$ contains some pre-tableau $\Fs_{\udl{\Fl}}$ with $\udl{\Fl}=[2^n]$ (this is also consistent with Li's results in \cite{Li89}). This will finish the proof of the proposition. 

The statement \eqref{geq2} is well-known when $n=1$. Assume that $n\geq 2$, and we consider two cases separately. 

{\bf Case 1 $G_n^* = \Sp_{2n}$:} In this case, $\varphi$ is of orthogonal type with  dimension $2n+1\geq 5$. Applying Lemma \ref{claim} to $\varphi_1=\varphi\oplus\BC$ and an arbitrary extension  of $\chi$ from $\CS_\varphi$ to $\CS_{\varphi_1}$, we obtain 
an orthogonal parameter $\phi$ of dimension $2n$ and $z\in\CZ$, such that
\[
\chi_{\varphi,\phi}=\chi_{\varphi_1, \phi}|_{\CS_\varphi} = \chi\cdot\eta_z\in\ScO_\CZ(\chi).
\]
By the proof of \cite[Proposition 18.1]{GGP12}, we have that $\chi_{\varphi, \phi}=\chi_{\varphi(z), \phi(z)}$. Then from \eqref{char-spmp} we compute that
\[
\chi^z_{\varphi,\phi(z)} = \chi_{\varphi(z),\phi(z)}\cdot\eta_z =\chi_{\varphi,\phi}\cdot\eta_z= \chi.
\]
It follows that $(\phi(z), \chi^z_{\phi(z),\varphi})\in \FD^z_2(\varphi,\chi)$. Hence we obtain that $\FD_2(\varphi,\chi)\neq\varnothing$.

{\bf Case 2 $G_n^*= \Mp_{2n}$:} In this case, $\varphi$ is  of symplectic type with dimension $2n\geq 4$. By Lemma~\ref{claim}, we obtain an orthogonal parameter $\phi$ of dimension $2n-2$  such that $\chi_{\varphi,\phi}=\chi$. Assume that $\det\phi=\BC(z)$ with $z\in\CZ$, and put $\phi' =  \phi(z)\oplus \BC(z)$. Then $\phi'$ is of dimension $2n-1$ with $\det\phi'=\BC$. From \eqref{char-spmp} we compute that
\[
\chi^z_{\varphi,\phi'} = \chi_{\varphi, \phi'(z)\oplus\BC} = \chi_{\varphi, \phi\oplus 2\cdot\BC} = \chi_{\varphi,\phi}=\chi.
\]
It follows that $(\phi',\chi^z_{\phi',\varphi})\in \FD^z_2(\varphi,\chi)$. Hence $\FD_2(\varphi,\chi)\neq\varnothing$.

This completes the proof of Proposition \ref{prop:LDPT}. 

\quad

It is clear from the proof above that we have the following corollary. 

\begin{cor}\label{cor:DNE-NA}
Proposition \ref{prop:DNE} holds when $F$ is non-archimedean. 
\end{cor}

\end{document}